\documentclass[]{interact}

\usepackage{epstopdf}
\usepackage[center]{caption}
\usepackage{wrapfig}
\usepackage[numbers,sort&compress]{natbib}
\bibpunct[, ]{[}{]}{,}{n}{,}{,}
\renewcommand\bibfont{\fontsize{10}{12}\selectfont}

\theoremstyle{plain}
\newtheorem{theorem}{Theorem}[section]
\newtheorem{lemma}[theorem]{Lemma}

\newtheorem{proposition}[theorem]{Proposition}

\theoremstyle{definition}
\newtheorem{definition}[theorem]{Definition}

\theoremstyle{remark}

\usepackage[table,xcdraw]{xcolor}
\newtheorem{assumption}{Assumption}

\usepackage{wrapfig,lipsum}
\usepackage{microtype}
\usepackage{graphicx}
\usepackage{subfigure}
\usepackage{booktabs}
\usepackage{algorithm,algorithmic}
\usepackage{hyperref}

\usepackage{amsmath,amsfonts,bm}

\def\1{\bm{1}}

\DeclareMathAlphabet{\mathsfit}{\encodingdefault}{\sfdefault}{m}{sl}
\SetMathAlphabet{\mathsfit}{bold}{\encodingdefault}{\sfdefault}{bx}{n}

\newcommand{\E}{\mathbb{E}}

\newcommand{\R}{\mathbb{R}}

\DeclareMathOperator*{\argmax}{arg\,max}
\DeclareMathOperator*{\argmin}{arg\,min}

\usepackage{amsmath,amsthm,amssymb}
\usepackage{amsmath,amsfonts,bm}
\usepackage{mathtools}
\usepackage{epsfig}

\newcommand{\bx}{{\mathbf x}}
\newcommand{\by}{{\mathbf y}}
\newcommand{\bz}{{\mathbf z}}
\newcommand{\ba}{{\mathbf a}}

\newcommand{\bc}{{\mathbf c}}

\newcommand{\X}{{\mathcal X}}
\newcommand{\Y}{{\mathcal Y}}

\newcommand{\bxi}{\boldsymbol{\xi}}
\newcommand{\hbg}{\hat{\mathbf{g}}}
\newcommand{\bzt}{\boldsymbol{\zeta}}
\newcommand{\bdelta}{\boldsymbol{\Delta}}
\newcommand{\reals}{\mathbb{R}}

\newcommand{\bg}{\mathbf{g}}
\newcommand{\Pc}{\mathcal{P}}
\newcommand{\Dc}{\mathcal{D}}

\begin{document}

\articletype{ARTICLE}

\title{Weakly-Convex Concave Min-Max Optimization: Provable Algorithms and Applications in Machine Learning}

\author{
\name{Hassan Rafique\textsuperscript{a}, Mingrui Liu\textsuperscript{b}, Qihang Lin\textsuperscript{c}\thanks{CONTACT Qihang Lin. Email: qihang-lin@uiowa.edu} and Tianbao Yang\textsuperscript{b}}
\affil{\textsuperscript{a}Program in Applied Mathematical and Computational Sciences, The University of Iowa, Iowa City, IA, USA; \textsuperscript{b}Department of Computer Science, The University of Iowa; \textsuperscript{c}Department of Business Analytics, The University of Iowa }
}

\maketitle

\begin{abstract}
Min-max problems have broad applications in machine learning, including learning with non-decomposable loss and learning with robustness to data distribution. Convex-concave min-max problem is an active topic of research with efficient algorithms and sound theoretical foundations developed. However, it remains a challenge to design provably efficient algorithms for non-convex min-max problems with or without smoothness. In this paper, we study a family of non-convex min-max problems, whose objective function is weakly convex in the variables of minimization and is concave in the variables of maximization. We propose a proximally guided stochastic subgradient method and a proximally guided stochastic variance-reduced method for the non-smooth and smooth instances, respectively, in this family of problems. We analyze the time complexities of the proposed methods for finding a nearly stationary point of the outer minimization problem corresponding to the min-max problem. 
\end{abstract}


\section{Introduction}
\label{sec:intro}

The goal of this paper is to design provably efficient algorithms for solving min-max (aka saddle-point) problems of the following form:
\begin{eqnarray}\label{eqn:P1}
\min_{\bx\in\mathbb{R}^p} \Big\{\psi(\bx):=\max_{\by\in\mathbb{R}^q}\left[ f(\bx, \by)-r(\by)\right]+g(\bx)\Big\},
\end{eqnarray}
where $f:\mathbb{R}^p\times\mathbb{R}^q\rightarrow\mathbb{R}$ is continuous, and $r:\mathbb{R}^q\rightarrow\mathbb{R}\cup\{+\infty\}$ and $g:\mathbb{R}^p\rightarrow\mathbb{R}\cup\{+\infty\}$ are proper and closed.\footnote{A closed function can be also called a lower semi-continuous function.} Detailed assumptions (e.g. weak convexity) on $f$, $r$ and $g$ will be given in Assumptions~\ref{assume:stochastic},~\ref{assume:stochastic1} and~\ref{assume:finitesum}. Many machine learning tasks can be  formulated as \eqref{eqn:P1} and some applications are given in Section~\ref{sec:apps}. Although there exist many studies on min-max problems in the form of~\eqref{eqn:P1}, most of them focus on {\it convex-concave} problems, where both $r$ and $g$ are convex and $f(\bx, \by)$ is convex in $\bx$ given $\by$ and is concave in $\by$ given $\bx$. On the contrary, this paper focuses on developing provably efficient algorithms for~(\ref{eqn:P1}) that exhibits non-convexity in $\bx$.

When designing an algorithm for~(\ref{eqn:P1}) with non-convexity, the important questions are what type of solutions can the algorithm guarantees and what runtime the algorithm needs to find such solutions. In the recent studies on non-convex minimization~\citep{davis2018stochastic,Davis2018,davis2018complexity,davis2017proximally,drusvyatskiy2017proximal,Drusvyatskiy2018,DBLP:journals/siamjo/GhadimiL13a,DBLP:journals/mp/GhadimiL16,DBLP:journals/corr/abs/1805.05411,paquette2018catalyst,Reddi:2016:SVR:3045390.3045425,DBLP:conf/cdc/ReddiSPS16}, polynomial-time algorithms have been developed for finding a \textbf{nearly stationary point}, namely, a point close to another point where the subdifferential of objective function almost contains zero. However, these developments are for a general minimization problem without utilizing the max-structure in~\eqref{eqn:P1} when it arises, and thus are not directly applicable to~(\ref{eqn:P1}). For example, the method in \cite{davis2017proximally} required that a stochastic subgradient of $\psi$ is available at each iteration. This might not be true if the maximization over $\by$ in~(\ref{eqn:P1}) is non-trivial or if $f$ is given as an expectation.

To fill this gap, in this paper, we analyze the complexity of a first-order algorithm designed to find a nearly stationary point of $\psi(\bx)$ in (\ref{eqn:P1}) without solving the possibly complex inner maximization problem exactly for each $\bx$. We focus on a typical case of~(\ref{eqn:P1}) that is non-convex but weakly convex in $\bx$ and is concave in $\by$. We refer to this case of \eqref{eqn:P1} as a {\bf weakly-convex-concave} (WCC) min-max problem. We provide complexity analysis for two primal-dual stochastic first-order methods under different settings of $f$. We will discuss some important applications of WCC min-max problems in machine learning next and summarize our main contributions at the end of this section. 

\subsection{Applications Of WCC Min-Max Problem}\label{sec:apps}
{\bf Distributionally robust learning.} \cite{DBLP:conf/nips/NamkoongD17} formulates the problem of distributionally robust learning as
\begin{eqnarray}\label{eqn:P2}
\min_{\bx\in\X}\max_{\by\in\Y}\sum_{i=1}^ny_i f_i(\bx) - r(\by),
\end{eqnarray}
where $f_i(\bx)$ denotes the loss of a model $\bx$ on the $i$th data point,  $\Y = \{\by \in \reals^n~|~ \sum^n_{i=1} y_i = 1,~ y_i \geq 0, i=1,....,n\}$, 
$\mathcal{X} \subseteq \reals^d$ is a closed convex  set, and $r: \Y \rightarrow \reals$ is a closed convex function, e.g.,
$r(\by) = \theta D(\by, \mathbf 1/n)$ for some $\theta>0$, where $\mathbf 1$ is an all-one vector and $D(\cdot, \cdot)$ denotes some distance measure (e.g., KL divergence and Euclidean distance) between two vectors. The loss $f_i(\bx)$ can be a non-convex function such as the loss function defined by a neural network or the non-convex truncated loss~\cite{DBLP:journals/corr/abs-1805-07880}. It is easy to see that \eqref{eqn:P2} is a special case of \eqref{eqn:P1}.

{\bf Learning with non-decomposable loss.} Problem \eqref{eqn:P1} also covers some models where the objective function cannot be represented as a summation of the losses over a training set. 
One example is minimizing {\it the sum of top-$k$ losses}~\citep{DBLP:conf/nips/FanLYH17,DBLP:conf/nips/DekelS06}:
\begin{align*}
\min_{\bx\in\mathcal X}\sum_{i=1}^k f_{\pi_i}(\bx),
\end{align*}
where $(\pi_i)_{i=1}^n$ is a permutation of $\{1,\ldots, n\}$ such that $f_{\pi_1}(\bx)\geq f_{\pi_2}(\bx)\cdots \geq f_{\pi_n}(\bx)$. When $k=1$, it reduces to the minimization of maximum loss~\citep{DBLP:conf/icml/Shalev-ShwartzW16}. 
This problem can be reformulated as the form~\eqref{eqn:P1} according to~\cite{bennett1988interpolation} and~\cite{DBLP:conf/nips/DekelS06} and our methods can be applied even if $f_i(\bx)$ is non-convex.

{\bf Robust learning from multiple distributions.} Let $ P_1, \ldots,  P_m$ denote $m$ distributions of data, e.g., different perturbed versions of the underlying true distribution $ P_0$. Robust learning from multiple distributions is to minimize the maximum of the expected losses over the $m$ distributions, i.e., 
\begin{align}\label{eqn:P4}
\min_{\bx\in\X}\max_{i} \E_{\bxi\sim P_i}[F(\bx; \bxi)]
=\min_{\bx\in\X}\max_{\by\in\Y} \sum_{i=1}^my_if_i(\bx),
\end{align}
where $\Y = \{\by \in \reals^m~|~ \sum^m_{i=1} y_i = 1,~ y_i \geq 0, i=1,....,m\}$, $F(\bx; \bxi)$ is a non-convex loss function (e.g., the loss defined by a deep neural network) of a model parameterized by $\bx$ on an example $\bxi$, and  
$f_i(\bx)=\E_{\bxi\sim P_i}[F(\bx; \bxi)]$ denotes the expected loss on distribution $P_i$. It is an instance of~(\ref{eqn:P1}) that is non-convex in $\bx$ and concave in $\by$. Such a formulation also has applications in adversarial machine learning~\citep{DBLP:journals/corr/MadryMSTV17,DBLP:journals/corr/abs-1710-10571}, in which $P_i$ corresponds to a distribution used to generate adversarial examples.

\subsection{Contributions} 
\label{sec:contrib}
Despite the concavity of $f$ in $\by$, there still remain challenges in developing provably efficient  algorithms for (\ref{eqn:P1}) due to the non-convexity of $f$ in $\bx$.
The major contribution of this work is the development and analysis of two efficient stochastic algorithms for the WCC min-max problem~(\ref{eqn:P1}) without evaluating the maximization in~(\ref{eqn:P1}) exactly at each iteration. We establish the convergence of the proposed algorithms to a nearly stationary point of $\psi(\bx)$. The main \textbf{key assumption} we make is that $f(\bx, \by)$ is weakly convex in $\bx$ and concave in $\by$. We summarize our contributions as follows.
\begin{itemize}
	\item When $f$ is potentially non-smooth and given as the expectation of a stochastic function, we present a stochastic primal-dual subgradient algorithm for \eqref{eqn:P1}. This method achieves time complexity of $\tilde O(1/\epsilon^6)$ for finding a nearly $\epsilon$-stationary solution.~\footnote{Here and in the rest of the paper, $\widetilde O$ suppresses all logarithmic factors.} 
	\item When $r$ is strongly convex and $f$ can be represented as the average of $n$ potentially non-smooth functions linear in $\by$ (see D2 in Assumption~\ref{assume:stochastic1}), the complexity of our method can be improved to $\tilde O(n/\epsilon^2+1/\epsilon^4)$. 
	To the best of our knowledge, this is the first result on accelerating a stochastic primal-dual subgradient algorithm for a non-convex and non-smooth min-max problem under strong concavity.
	\item When $f$ is represented as the average of $n$ smooth functions, we present a stochastic variance-reduced gradient method for \eqref{eqn:P1}. When $r$ is strongly convex, our method has time complexity of $\tilde O(n/\epsilon^2)$ for finding a nearly $\epsilon$-stationary solution. When  $r$ is convex, this method has time complexity of $\tilde O(n/\epsilon^2+1/\epsilon^6)$.
\end{itemize}


\section{Related Work}
Recently, there is a growing interest in stochastic algorithms for non-convex optimization.  When the objective function is smooth, stochastic gradient descent~\citep{DBLP:journals/siamjo/GhadimiL13a} and its variants~\citep{yangnonconvexmo,DBLP:journals/mp/GhadimiL16} have been analyzed for finding a stationary point of the objective function. When the objective function has a finite-sum structure, accelerated stochastic algorithms with improved complexity have been developed based on the variance reduction techniques~\citep{DBLP:conf/cdc/ReddiSPS16,Reddi:2016:SVR:3045390.3045425,DBLP:journals/corr/abs/1805.05411,DBLP:conf/icml/Allen-Zhu17,DBLP:conf/icml/ZhuH16}. 
However, these methods are not directly applicable to \eqref{eqn:P1} because of the non-smoothness of $\psi$.

Several recent papers have proposed algorithms and provided theoretical guarantees for non-convex optimization with non-smooth objective functions by assuming the objective function is weakly convex~\citep{Davis2018,davis2018stochastic,Drusvyatskiy2018,davis2017proximally,chen18stagewise,zhang2018convergence}. Although the objective function $\psi$ in our problem is also weakly convex, most of the aforementioned methods cannot be directly applied to \eqref{eqn:P1} because they need to compute the (stochastic) subgradient of $\psi(\bx)$ which requires exactly solving the concave maximization problem $\max_{\by\in\mathbb{R}^q}[f(\bx, \by)-r(\by)]$ that can be challenging, e.g., when $f$ involves expectations as in~(\ref{eqn:P4}).

The following weakly convex composite optimization has been studied by~\cite{Drusvyatskiy2018},
\begin{eqnarray}
\label{eq:composite}
\min_{\bx\in\mathbb{R}^p}h(\mathbf{c}(\bx))+g(\bx),
\end{eqnarray}
where $g:\mathbb{R}^p\rightarrow\mathbb{R}\cup\{+\infty\}$ is closed and convex, $h:\mathbb{R}^q\rightarrow\mathbb{R}$ is Lipschitz-continuous and convex and $\mathbf{c}:\mathbb{R}^p\rightarrow\mathbb{R}^q$ is smooth with a Lipschitz-continuous Jacobian matrix. It is a special case of~(\ref{eqn:P1}) because we can reformulate \eqref{eq:composite} as
\begin{eqnarray}
\label{eq:compositenew}
\min_{\bx\in\mathbb{R}^p}\max_{\by\in\mathbb{R}^q}\by^\top\mathbf{c}(\bx)-h^*(\by)+g(\bx),
\end{eqnarray}
where $h^*(\bz)=\sup_{\by\in\mathbb{R}^q}\{\by^\top\bz-h(\by)\}$ is the conjugate  of $h$. 
\cite{Drusvyatskiy2018} proposed a prox-linear method for \eqref{eq:composite} where the iterate $\bx^{(t)}$ is updated by approximately solving
\begin{align}
\label{eq:compositesub}
&\bx^{(t+1)}\approx ~~~
\argmin_{\bx\in\mathbb{R}^p} \left\{
\begin{array}{l}h\left(\mathbf{c}(\bx^{(t)})+\nabla \mathbf{c}(\bx^{(t)})(\bx-\bx^{(t)})\right)+g(\bx)+\frac{1}{2\eta}\|\bx-\bx^{(t)}\|_2^2
\end{array}
\right\}.
\end{align}
Here, $\nabla \mathbf{c}$ is the Jacobian matrix of $\mathbf{c}$ and $\eta>0$ is a step length. They considered using accelerated gradient method to solve \eqref{eq:compositesub}. However, this approach requires the exact evaluation of $\mathbf{c}$ and $\nabla\mathbf{c}$ which is computationally expensive, for example, when $\mathbf{c}$ is given as an expectation or a finite sum of many terms (see \eqref{eqn:P2} and \eqref{eqn:P4}). One approach to avoid this issue is to solve \eqref{eq:compositesub} as a min-max subproblem
\begin{align}
\label{eq:compositesubnew}
&\bx^{(t+1)}\approx ~~~
\argmin_{\bx\in\mathbb{R}^p}\max_{\by\in\mathbb{R}^q}\left\{
\begin{array}{l}
\by^\top\left[\mathbf{c}(\bx^{(t)})+\nabla \mathbf{c}(\bx^{(t)})(\bx-\bx^{(t)})\right]\\
-h^*(\by)+g(\bx)+\frac{1}{2\eta}\|\bx-\bx^{(t)}\|_2^2
\end{array}
\right\}
\end{align}
using primal-dual stochastic first-order methods that only require the sample approximation of $\mathbf{c}$ and $\nabla\mathbf{c}$. However, they still require $\mathbf{c}$ to be smooth to ensure the convergence of $\bx^{(t)}$.
In contrast, our convergence analysis covers the case where $\mathbf{c}$ is non-smooth.

The authors of \cite{NIPS2017_7056} considered robust learning from multiple datasets to tackle uncertainty in the data, which is a special case of~(\ref{eqn:P4}). They assumed that the minimization over $\bx$ for any fixed $\by$ can be solved up to a certain optimality gap relative to the global minimum value. However, such an assumption does not hold when the minimization over $\bx$ is non-convex. \cite{DBLP:journals/corr/abs-1805-07588} considered a similar problem to~(\ref{eqn:P4}) and analyzed a primal-dual stochastic algorithm. There are several differences between their results and ours: (i) they require $\E_{\ba\sim  P_i}[f(\bx; \ba)]$ to be smooth while we assume it is weakly convex but not necessarily smooth; (ii) they proved the convergence of the partial gradient of $f(\bx, \by)$ with respect to $\bx$ while we provide the convergence of the gradient of the Moreau envelope of $\psi(\bx)$; (iii) when $f(\bx, \by)$ is smooth and has a finite-sum structure, we provide an accelerated stochastic algorithm based on variance reduction which they did not consider. 
\citep{thekumparampil2019efficient,lu2019block,lu2019hybrid,nouiehed2019solving} propose deterministic algorithms for a problem similar to~(\ref{eqn:P1}) and  \citep{lin2019gradient,boct2020alternating} consider both deterministic and stochastic methods for~(\ref{eqn:P1}). These studies, except \cite{boct2020alternating}, all assume $f$ is smooth, and \cite{nouiehed2019solving} further assume $f$ satisfies the Polyak-Lojasiewicz condition in $\by$. On the contrary, our results cover the case where $f$ is non-smooth and, without strong concavity in $y$, our stochastic algorithm has lower complexity than \citep{lin2019gradient,boct2020alternating}. In particular, the complexity of our method for finding a nearly $\epsilon$-stationary point is $\tilde O(\frac{1}{\epsilon^6})$ without strong concavity while the complexity of \citep{lin2019gradient,boct2020alternating} is $\tilde O(\frac{1}{\epsilon^8})$.



\section{Preliminaries}\label{sec:pre}
We consider the min-max problem \eqref{eqn:P1} where $f:\mathbb{R}^p\times\mathbb{R}^q\rightarrow\mathbb{R}$ is continuous, $r:\mathbb{R}^q\rightarrow\mathbb{R}\cup\{+\infty\}$ and $g:\mathbb{R}^p\rightarrow\mathbb{R}\cup\{+\infty\}$ are proper and closed. 
According to most of the applications of~(\ref{eqn:P1}) (see Section~\ref{sec:apps}), the roles of $\bx$ and $\by$ are often not symmetric, e.g.,  
$\by$ are often restricted to a simplex in $\mathbb{R}^q$ but $\bx$ are not. Hence, 
we only equip $\mathbb{R}^p$ with the Euclidean norm $\|\cdot\|_2$ but equip $\mathbb{R}^q$ with a generic norm $\|\cdot\|$ so that the algorithm can better adapt to the geometry of the feasible set of $\by$. We denote the dual norm of $\|\cdot\|$ by $\|\cdot\|_*$ and define $\text{Dist}(\bx,\mathcal{S}):=\min_{\bx'\in\mathcal{S}}\|\bx'-\bx\|_2$ for any set $\mathcal{S}\subset\mathbb{R}^p$. 

Given $h:\reals^d\rightarrow \reals\cup\{+\infty\}$, the  subdifferential of $h$ is
\begin{align*}
\partial h(\bx)=
\left\{\bzt\in\mathbb{R}^d\bigg| 
\begin{array}{l}
h(\bx')
\geq h(\bx)+\bzt^\top(\bx'-\bx)
+o(\|\bx'-\bx\|_2), ~\bx'\rightarrow\bx
\end{array}
\right\},
\end{align*}
where each element in $\partial h(\bx)$ is called a subgradient of $h$ at $\bx$. Let $\partial_x f(\bx,\by)$ be the subdifferential of $f(\bx,\by)$ with respective to $\bx$ and  $\partial_y[-f(\bx,\by)]$ be the subdifferential of $-f(\bx,\by)$ with respective to $\by$.

Let 
$$
\X:=\text{dom}(g)\text{ and }\Y:=\text{dom}(r).
$$
Let $d_y:\Y\rightarrow\reals$ be a \emph{distance generating function} with respect to $\|\cdot\|$, meaning that $d_y$ is convex on $\Y$,  continuously differentiable and $1$-strongly convex with respect to $\|\cdot\|$ on  $\Y^o:=\{\by\in\Y|\partial d_y(\by)\neq\emptyset\}$. The \emph{Bregman divergence} associated to $d_y$ is defined as
$V_y:\Y\times\Y^o\rightarrow \mathbb{R}$ such that 
\begin{eqnarray}
\label{eq:breg}
V_y(\by,\by'):=d_y(\by)-d_y(\by')-\left\langle\nabla  d_y(\by'),\by-\by'\right\rangle.
\end{eqnarray}
Note that we have $V_y(\by,\by')\geq \frac{1}{2}\|\by-\by'\|^2$.
We say a function $h:\Y\rightarrow\reals$ is \emph{$\mu$-strongly convex ($\mu\geq0$) with respect to $V_y$} if 
$$
h(\by)\geq h(\by')+\bzt^\top(\by-\by')+\mu V_y(\by,\by')
$$
for any $(\by,\by')\in\Y\times\Y^o$ and any $\bzt\in\partial h(\by')$. We say a function $h:\X\rightarrow\reals$ is \emph{$\rho$-weakly convex} ($\rho\geq0$) if
$$
h(\bx)\geq h(\bx')+\bzt^\top(\bx-\bx')-\frac{\rho}{2}\|\bx-\bx'\|_2^2
$$
for any $\bx,\bx'\in\X$ and any $\bzt\in\partial h(\bx')$.

The following assumptions are made throughout the entire paper:
\begin{assumption}
	\label{assume:stochastic}
	The following statements hold: 
	\begin{itemize}
		\item[(A)]  $f(\bx, \by)$ is $\rho$-weakly convex in $\bx$ for any $\by\in\Y$.
		\item[(B)]  $f(\bx, \by)$ is concave in $\by$ for any $\bx\in\X$. 
		\item[(C)] $r$ is closed and $\mu$-strongly convex with respect to $V_y$ with $\mu\geq 0$ ($\mu$ can be zero) and $g$ is closed and convex. 
		\item[(D)] $\psi^*:=\min_{\bx\in\mathbb{R}^p}\psi(\bx)>-\infty$.
		\end{itemize}
	\end{assumption}

Under Assumption~\ref{assume:stochastic}, $\psi$ is $\rho$-weakly convex (but not necessarily convex) so that finding the global optimal solution in general is difficult. An alternative goal is to find a \textbf{stationary} point of (\ref{eqn:P1}), i.e., a point $\bx_*$ with
$\mathbf{0}\in\partial \psi(\bx_*)$.
An exact stationary point, in general, can only be approached in the limit as the number of iterations increases to infinity. Within finitely many iterations, a more reasonable goal is to find an $\epsilon$-stationary point defined below. 
\begin{definition}
	A point $\widehat\bx\in\X$  is an \textbf{$\epsilon$-stationary} point if $\text{Dist}(\mathbf{0},\partial\psi(\widehat\bx))
	\leq \epsilon$.
\end{definition}
However, when $\psi$ is non-smooth, computing an $\epsilon$-stationary point is still difficult even for a convex problem. A simple example is $\min_{x\in\reals}|x|$ where the only stationary point is $0$ but $x\neq 0$ is not an $\epsilon$-stationary point with $\epsilon<1$ no matter how close $x$ is to $0$. This situation is likely to occur in problem (\ref{eqn:P1}) because of not only the potential non-smoothness of $f$ and $r$ but also the inner maximization. Therefore, following \cite{Davis2018}, \cite{davis2017proximally}, \cite{davis2018complexity} and \cite{zhang2018convergence}, we consider the \emph{Moreau envelope} of $\psi$ defined as 
\begin{equation}\label{moreau}
\psi_{\gamma}(\bx):= \min_{\bz \in \mathbb{R}^p} \left\{\psi(\bz)+\frac{1}{2\gamma}\|\bz-\bx\|_2^2\right\}
\end{equation}
for a constant $\gamma>0$. For a $\rho$-weakly convex function $\psi$, it can be shown that $\psi_{\gamma}$ is smooth  when $\frac{1}{\gamma}>\rho$ \cite{davis2017proximally} and its gradient is 
\small
\begin{equation}\label{moreaugrad}
\nabla\psi_\gamma (\bx)=\gamma^{-1}(\bx-\text{prox}_{\gamma\psi}(\bx)),
\end{equation}
\normalsize
where 
\begin{align*}
    \text{prox}_{\gamma\psi}(\bx):=\argmin_{\bz \in \mathbb{R}^p} \left\{\psi(\bz)+\frac{1}{2\gamma}\|\bz-\bx\|_2^2\right\}.
\end{align*}
Note that, when $\frac{1}{\gamma}>\rho$, the minimization in (\ref{moreau}) and (\ref{moreaugrad}) is strongly convex so that $\text{prox}_{\gamma\psi}(\bx)$ is uniquely defined. We then give the following definition
\begin{definition}
A point $\bar\bx\in\X$  is a \textbf{nearly $\epsilon$-stationary} point if $\|\nabla\psi_{\gamma} (\bar\bx)\|\leq \epsilon$.
\end{definition}
According to \cite{Davis2018}, \cite{davis2017proximally}, \cite{davis2018complexity} and \cite{zhang2018convergence}, the norm of the gradient $\|\nabla\psi_{\gamma} (\bx)\|$ can be used as a measure of the quality of a solution $\bx$. In fact, let $\bx_\dagger=\text{prox}_{\gamma\psi}(\bar\bx)$. The definition \eqref{moreaugrad} and the optimality condition of $\bx_\dagger$ directly imply that
$$ 
	\|\bx_\dagger-\bar\bx\|=\gamma\|\nabla\psi_{\gamma} (\bar\bx)\|\quad\text{ and }\quad\text{Dist}(\mathbf{0},\partial \psi(\bx_\dagger))\leq\|\nabla\psi_{\gamma} (\bar\bx)\|.
$$
Therefore, if we can find a nearly $\epsilon$-stationary point $\bar\bx$, we will ensure
$\|\bx_\dagger-\bar\bx\|\leq \gamma\epsilon$ and $\text{Dist}(\mathbf{0},\partial \psi(\bx_\dagger))\leq \epsilon$, or in other words, we will find a solution $\bar\bx$ that is $\gamma\epsilon$-closed to  $\epsilon$-stationary point (i.e., $\bx_\dagger$). Hence, we will focus on developing first-order methods for  (\ref{eqn:P1}) and analyze their time complexity for finding a nearly $\epsilon$-stationary point.


\section{Proximally Guided Approach}
The method we proposed is largely inspired by the proximally guided stochastic subgradient method by \cite{davis2017proximally},  which is a variant of the inexact proximal point method~\cite{citeulike:9472207}. \cite{davis2017proximally} considers $\min_{\bx\in\mathbb{R}^p}\psi(\bx)$ with $\psi(\bx)$ being a $\rho$-weakly convex function that does not necessarily have the maximization structure in (\ref{eqn:P1}). Given an iterate $\bar\bx^{(t)}\in\mathcal{X}$, their method uses the standard stochastic subgradient (SSG) method to approximately solve \eqref{moreau} with $\bx=\bar\bx^{(t)}$ and $\frac{1}{\gamma}>\rho$, namely, to compute a solution $\bar\bx^{(t+1)}\in\X$ so that
\begin{eqnarray}
\label{eq:subproblem}
\bar\bx^{(t+1)}\approx \bx_\dagger^{(t)}= \text{prox}_{\gamma\psi}(\bar\bx^{(t)})=\argmin_{\bz \in \mathbb{R}^p} \left\{\psi(\bz)+\frac{1}{2\gamma}\|\bz-\bar\bx^{(t)}\|_2^2\right\}.
\end{eqnarray} 
Then $\bar\bx^{(t+1)}$ returned by the SSG method will be used as the next iterate. They assume that a (stochastic) subgradient of $\psi(\bx)$ can be computed during the SSG method, which does not hold in our setting due to the maximization in (\ref{eqn:P1}).  

To address this issue, we consider the following min-max problem according to (\ref{eq:subproblem}): 
\begin{equation}\label{subminmaxproblem0}
\min_{\bx \in \mathbb{R}^p}\max_{\by\in\mathbb{R}^q} \left\{\phi_{\gamma}(\bx,\by;\bar\bx):=f(\bx,\by)-r(\by)+g(\bx)+\frac{1}{2\gamma}\|\bx-\bar\bx\|_2^2\right\},
\end{equation}
and equivalently solve \eqref{eq:subproblem} as
\begin{equation}\label{subminmaxproblem1}
\bar\bx^{(t+1)}\approx \argmin_{\bx \in \mathbb{R}^p}\max_{\by\in\mathbb{R}^q}\phi_{\gamma}(\bx,\by;\bar\bx^{(t)}).
\end{equation}

Existing primal-dual first-order methods can be applied to (\ref{subminmaxproblem0}), or specifically, (\ref{subminmaxproblem1}), and the returned approximate solution $\bar\bx^{(t+1)}$ is then used as the next iterate. In this section, we will consider two cases for $f$: when $f$ is given as an expectation and potentially non-smooth, and when $f$ is a finite sum and smooth. 

For a non-smooth min-max problem, strong convexity and strong concavity are simultaneously needed to accelerate a primal-dual first-order method. Under Assumption~\ref{assume:stochastic}, when $\frac{1}{\gamma}>\rho$, $\phi_{\gamma}(\bx,\by;\bar\bx^{(t)})$ is $(\frac{1}{\gamma}-\rho)$-strongly convex in $\bx$ and $\mu$-concave with respect to $V_y$ in $\by$. Although a deterministic primal-dual method for (\ref{subminmaxproblem1}) with a non-smooth $f$ can be accelerated when $\mu>0$, it is surprising that there is no existing stochastic primal-dual methods that can be accelerated when $\mu>0$. In Section~\ref{sec:stochastic}, we propose a new stochastic primal-dual method for (\ref{eqn:P1}) whose complexity for finding a nearly $\epsilon$-stationary point can be improved from $\tilde {O}(\epsilon^{-6})$ when $\mu=0$ to $\tilde {O}(\epsilon^{-4})$ when $\mu>0$ under a structural assumption (D2 in Assumption~\ref{assume:stochastic1}). This is the first stochastic algorithm for non-smooth non-convex min-max problems that can be accelerated under strong concavity.
\subsection{Stochastic Min-Max Problem}\label{sec:stochastic}
In this section, we make the following assumptions in addition to Assumption~\ref{assume:stochastic}. 
\begin{assumption}
	\label{assume:stochastic1}
	The following statements hold: 
	\begin{itemize}
		\item[(A)] For any $(\bx,\by)\in\mathcal{X}\times \mathcal{Y}$, we can compute two random vectors, denoted by $\bg_x(\bx,\by)\in\mathbb{R}^p$ and $\bg_y(\bx,\by)\in\mathbb{R}^q$, such that 
		$(\mathbb{E}\bg_x(\bx,\by),-\mathbb{E}\bg_y(\bx,\by))\in\partial_x f(\bx,\by)\times \partial_y [-f(\bx,\by)]$.
		\item[(B)] There exist constants $M_x>0$ and $M_y>0$ such that
		$\mathbb{E}\|\bg_x(\bx,\by)\|_2^2\leq M_x^2$ and  $\mathbb{E}\|\bg_y(\bx,\by)\|_*^2\leq M_y^2$ for any $(\bx,\by)\in\mathcal{X}\times\mathcal{Y}$. 
		\item[(C)] $Q_g:=\max_{\bx\in\X}g(\bx)-\min_{\bx\in\X}g(\bx)<+\infty$ and $Q_r:=\max_{\by\in\Y}r(\by)-\min_{\by\in\Y}r(\by)<+\infty$.
		\item[(D)] Either one of the following two statements holds:
		\begin{itemize}
			\item[D1.]$\mu=0$; 
		 $D_y:=\sqrt{2\max\limits_{\by\in\Y}d_y(\by)-2\min\limits_{\by\in\Y}d_y(\by)}<+\infty$; and $D_x:=\max\limits_{\bx,\bx'\in\mathcal{X}}\|\bx-\bx'\|_2<+\infty$.
			\item[D2.]$\mu>0$; $f(\bx,\by)=\frac{1}{n}\sum_{i=1}^n\by\top\mathbf{c}_i(\bx)$ where $\by^\top\mathbf{c}_i(\bx)$ is $\rho$-weakly convex in $\bx$ for any $\by\in\Y$;  and 
			$\mathbf{c}(\bx):=\frac{1}{n}\sum_{i=1}^n\mathbf{c}_i(\bx)$ satisfies $\|\mathbf{c}(\bx)-\mathbf{c}(\bx')\|_*\leq M_c\|\bx-\bx'\|_2$ for some $M_c\geq 0$ and any $\bx$ and $\bx'$ in $\mathcal{X}$. Moreover, $\bg_x$ and $\bg_y$ in Assumption~\ref{assume:stochastic1} (A) above satisfy $\bg_x(\bx,\by)\in\partial_x[\by^\top \bc_{\bxi}(\bx)]$ and $\bg_y(\bx,\by)=\bc_{\bxi}(\bx)$, where $\bxi$ is a uniformly random index from $\{1,\dots,n\}$.
		\end{itemize}
	\end{itemize}
\end{assumption}

\begin{algorithm}[t]
	\caption{SMD$(\bar\bx,\bar\by,\eta_x,\eta_y,\gamma,J)$: Stochastic Mirror Descent Method for \eqref{subminmaxproblem0}}\label{alg:SMD}
	\begin{algorithmic}[1]
		\STATE \textbf{Input:} $(\bar\bx,\bar\by) \in \mathcal{X}\times \mathcal{Y}^o$, $\eta_x>0$, $\eta_y>0$, $\gamma \in (0,1/(2\rho)]$, and $J\geq 2$.  
		\STATE Set $(\bx^{(0)},\by^{(0)})=(\bar\bx,\bar\by)$.
		\FOR {$j = 0,...,J -2$}
		\STATE Compute
	 $\bg_x^{(j)}=\bg_x(\bx^{(j)},\by^{(j)})$
	 and 
	$\bg_y^{(j)}=\bg_y(\bx^{(j)},\by^{(j)})$.
		\STATE Solve  
		\small
		\begin{align*}
		\bx^{(j+1)}&=\argmin_{\bx\in\mathbb{R}^p} \left\{
		\begin{array}{l}
		\bx^\top\bg_x^{(j)} +\frac{1}{2\eta_x}\|\bx-\bx^{(j)}\|_2^2	
		+g(\bx) +\frac{1}{2\gamma}\|\bx-\bar\bx\|_2^2
		\end{array}
		\right\}\\
		\normalsize
		 \by^{(j+1)}&=\argmin_{\by\in\mathbb{R}^q} \left\{ -\by^\top\bg_y^{(j)}
		+\frac{1}{\eta_y}V_y(\by,\by^{(j)})+r(\by) \right\}.
		\end{align*}

		\ENDFOR
		\STATE 
		\textbf{Output:} 
		$ \widehat\bx = \frac{1}{J} \sum^{J - 1}_{j=0}\bx^{(j)}$.
	\end{algorithmic}
\end{algorithm}

We propose a proximally guided stochastic mirror descent (PG-SMD) method for (\ref{eqn:P1}) in Algorithms~\ref{alg:SMD} and~\ref{alg:PGM}. Algorithm~\ref{alg:PGM} is the main algorithm which,  in its $t$th iteration, uses Algorithm~\ref{alg:SMD} to approximately solve (\ref{subminmaxproblem1}). The returned solution $\bar\bx^{(t+1)}$ will be used as the $(t+1)$th iterate of Algorithm~\ref{alg:PGM}. 
Algorithm~\ref{alg:SMD} is exactly the primal-dual SMD method~\cite{nemirovski2009robust} applied to (\ref{subminmaxproblem0}). It is initialized at $(\bar\bx^{(t)},\bar\by^{(t)})$ 
where $\bar\by^{(t)}=\argmin_{\by\in\mathbb{R}^q}d_y(\by)$ when Assumption~\ref{assume:stochastic1} D1 holds and $\bar\by^{(t)}=\argmax_{\by\in\mathbb{R}^q}f(\bar\bx^{(t)},\by)-r(\by)$ 
when Assumption~\ref{assume:stochastic1} D2 holds. Assumption~\ref{assume:stochastic1} D2 (i.e., $f(\bx,\by)$ is linear in $\by$) ensures this maximization can be solved easily, for example, in a closed form.  We emphasize that such an initialization of $\bar\by^{(t)}$ in the case of D2 is the key to utilize the strong concavity of $\phi_{\gamma}(\bx,\by;\bar\bx^{(t)})$ in $\by$ to reduce the complexity from $\tilde {O}(\epsilon^{-6})$ when $\mu=0$ to $\tilde {O}(\epsilon^{-4})$ when $\mu>0$.

\begin{algorithm}[t]
	\caption{PG-SMD: Proximally Guided Stochastic Mirror Descent Method}\label{alg:PGM}
	\begin{algorithmic}[1]
		\STATE \textbf{Input:} $\bar\bx^{(0)} \in \mathcal{X}$ and $\gamma \in (0,1/(2\rho)]$.  
		\FOR{$t=0,..., T-2$}
		\STATE 
		$\bar\by^{(t)}=\left\{\begin{array}{ll}\argmin\limits_{\by\in\mathbb{R}^q}d(\by)\text{ if  D1 in Assumption~\ref{assume:stochastic1} holds} \\
		\argmax\limits_{\by\in\mathbb{R}^q}f(\bar\bx^{(t)},\by)-r(\by)\text{ if D2 in Assumption~\ref{assume:stochastic1} holds.} 
		\end{array}
		\right.$
		\STATE 
		$
		(\eta_x^t,\eta_y^t,j_t)=\left\{\begin{array}{ll}\left(\frac{D_x}{M_x\sqrt{j_t}},\frac{D_y}{M_y\sqrt{j_t}},(t+3)^2\right)\text{ if D1 in Assumption~\ref{assume:stochastic1} holds}\\
		\left(\frac{60\gamma}{(j_t-30)},\frac{8M_c^2\gamma}{\mu^2j_t},t+32\right)\text{ if D2 in Assumption~\ref{assume:stochastic1} holds.} 
		\end{array}
		\right.
		$
		
		\STATE 
		$(\bar\bx^{(t+1)},\bar\by^{(t+1)})=\text{SMD}(\bar\bx^{(t)},\bar\by^{(t)},\eta_x^t,\eta_y^t,\gamma, j_t)$.
		\ENDFOR
		\STATE Sample $\tau$ uniformly randomly from $\{0,1,\dots,T-1\}$.
		\STATE\textbf{Output:} 
		$ \bx^{(\tau)}$
	\end{algorithmic}
\end{algorithm}
Theorem~\ref{thm:totalcomplexitySD} characterizes the time complexity for PG-SMD to find a nearly $\epsilon$-stationary point. We define the time complexity as the total number of the stochastic subgradients $(\bg_x,\bg_y)$ calculated and we count the complexity of solving $\max_{\by\in\mathbb{R}^q}f(\bar\bx^{(t)},\by)-r(\by)$ in case D2 as computing $n$ stochastic subgradients. The proof and the exact expressions of $T$ are provided in Section~\ref{sec:prooftheorem1} in Appendix. 
\begin{theorem}
	\label{thm:totalcomplexitySD}
	Suppose Assumptions~\ref{assume:stochastic} and~\ref{assume:stochastic1} hold. Let $T$ be the total number of iterates in Algorithm~\ref{alg:PGM} and let $\tau$ be a random index uniformly sampled from $\{0,1,\dots,T-1\}$. Algorithm~\ref{alg:PGM} ensures
	$\mathbb{E}[\|\nabla\psi_{\gamma} (\bar\bx^{(\tau)})\|^2]\leq \epsilon^2$ after $T=\tilde O(\epsilon^{-2})$ iterations with total time complexity of
	\begin{enumerate}
	    \item[(a)]  $\tilde {O}(\epsilon^{-6})$ if Case D1 in Assumption~\ref{assume:stochastic1} holds.
	    \item[(b)] $\tilde {O}(n\epsilon^{-2}+\epsilon^{-4})$ if Case D2 in Assumption~\ref{assume:stochastic1} holds.
	\end{enumerate}
\end{theorem}

\subsection{Smooth Finite-Sum Min-Max Problem}
In this section, we no longer need Assumption~\ref{assume:stochastic1} but make the following assumption in addition to Assumption~\ref{assume:stochastic}. 
\begin{assumption}
	\label{assume:finitesum}
	The following statements hold:
	\begin{itemize}
		\item[(A)] $f(\bx,\by) = \frac{1}{n}\sum_{i=1}^nf_i(\bx,\by)$ with $f_i:\mathbb{R}^p\times \mathbb{R}^q\rightarrow\mathbb{R}$.
		\item[(B)]$f_i$ is differentiable with $\nabla_xf_i$  $L_x$-Lipschitz continuous and $\nabla_yf_i$ $L_y$-Lipschitz continuous.~\footnote{$\nabla_xf_i$ and $\nabla_yf_i$ are the partial gradients of $f_i$ with respect to $\bx$ and $\by$, respectively.}
		\item[(C)]  $D_x:=\max\limits_{\bx,\bx'\in\mathcal{X}}\|\bx-\bx'\|_2<+\infty$ and \small$D_y:=\sqrt{2\max\limits_{\by\in\Y}d_y(\by)-2\min\limits_{\by\in\Y}d_y(\by)}<+\infty$\normalsize.
	\end{itemize}
\end{assumption}

Different from the previous section, we can exactly evaluate $\nabla f$ but the computational cost can be still high when $n$ is large. To reduce the cost, in Algorithm~\ref{alg:SVRG} and Algorithm~\ref{alg:PGSVRG}, we propose a proximally guided stochastic variance-reduced  gradient (PG-SVRG) method for (\ref{eqn:P1}).
We consider a min-max problem
\begin{equation}\label{subminmaxproblem0finitesum}
\min_{\bx \in \mathbb{R}^p}\max_{\by\in\mathbb{R}^q} \left\{\phi_{\gamma,\lambda}(\bx,\by;\bar\bx,\bar\by):=\phi_{\gamma}(\bx,\by;\bar\bx)-\frac{1}{\lambda}V_y(\by,\bar\by)\right\},
\end{equation}
where $\lambda\in(0,+\infty]$ and $\phi_{\gamma}(\bx,\by;\bar\bx)$ is defined in \eqref{subminmaxproblem0}. Given the current iterate $\bar\bx^{(t)}$, Algorithm~\ref{alg:PGSVRG} applies Algorithm~\ref{alg:SVRG}, which is the stochastic variance-reduced  gradient (SVRG) method~\cite{palaniappan2016stochastic,shi2017bregman}, to approximately solve 
\begin{equation}\label{subminmaxproblem}
\bar\bx^{(t+1)}\approx\argmin_{\bx \in \mathbb{R}^p}\max_{\by\in\mathbb{R}^q} \phi_{\gamma,\lambda_t}(\bx,\by;\bar\bx^{(t)},\bar\by^{(t)})
\end{equation}
with some $\lambda_t\in(0,+\infty]$ and $\bar\by^{(t)}=\argmin_{\by\in\Y}d_y(\by)$ for any $t$. The returned solution $\bar\bx^{(t+1)}$ will be used as the next iterate. 
Compared to \eqref{subminmaxproblem0}, problem \eqref{subminmaxproblem0finitesum} has an additional proximal term $-\frac{1}{\lambda}V_y(\by,\bar\by)$ introduced to make (\ref{subminmaxproblem0finitesum}) strongly convex and strongly concave so that \eqref{subminmaxproblem} can be solved efficiently with the SVRG method.  When $r$ is already strongly concave with respect to $V_y$, we will set $\lambda_t=+\infty$ for any $t$ and follow the convention that $\frac{1}{\lambda_t}=0$ so that $-\frac{1}{\lambda_t}V_y(\by,\bar\by^{(t)})$ will be removed from \eqref{subminmaxproblem} and thus \eqref{subminmaxproblem}  is reduced to \eqref{subminmaxproblem1}.

\begin{algorithm}[h]
	\caption{SVRG$(\bar\bx,K,\lambda,\gamma)$: SVRG Method for \eqref{subminmaxproblem0finitesum}}\label{alg:SVRG}
	\begin{algorithmic}[1]
		\STATE \textbf{Input:} $\bar\bx \in \mathcal{X}$, $K\geq 2$, $\lambda>0$, and $\gamma \in (0,1/(2\rho)]$.
		\STATE Set $\mu_x=\frac{1}{2\gamma}$, $\mu_y=\frac{1}{\lambda}+\mu$, $\Lambda=\frac{52\max\left\{L_x^2,L_y^2\right\}}{\min\left\{\mu_x^2,\mu_y^2\right\}}+\frac{3}{2}$, $\eta_x=\frac{1}{\mu_x\Lambda}$, $\eta_y=\frac{1}{\mu_y\Lambda}$, and $J=\left\lceil 1+(\frac{3}{2}+3\Lambda)\log(4)\right\rceil$.
		\STATE Set $\widehat\bx^{(0)}=\bar\bx$ and $\widehat\by^{(0)}=\argmin_{\by\in\Y}d_y(\by)$.
		\FOR {$k = 0,...,K-2$}
		\STATE 
		\small
		$\hbg_x^{(k)}=\nabla_xf(\widehat\bx^{(k)},\widehat\by^{(k)})~ \text{  and  }~ \hbg_y^{(k)}=\nabla_yf(\widehat\bx^{(k)},\widehat\by^{(k)})$.
		\normalsize
		\STATE Set $(\bx^{(0)},\by^{(0)})=(\widehat\bx^{(k)},\widehat\by^{(k)})$.
		\FOR {$j = 0,...,J -2$}
		\STATE Sample $l$ uniformly from $\{1,2,\dots,n\}$.
		\STATE Compute
		\small
		\begin{align*}
		\bg_x^{(j)}=\hbg_x^{(k)}-\nabla_x  f_l(\widehat\bx^{(k)},\widehat\by^{(k)})+\nabla_x f_l(\bx^{(j)},\by^{(j)})\\
		\bg_y^{(j)}=\hbg_y^{(k)}-\nabla_y  f_l(\widehat\bx^{(k)},\widehat\by^{(k)})+\nabla_y f_l(\bx^{(j)},\by^{(j)}).
		\end{align*}
		\normalsize
		\STATE Solve
		\small
		\begin{align*}
			\bx^{(j+1)}&=\argmin_{\bx\in\mathbb{R}^p}\left\{ 
			\begin{array}{l}
			\bx^\top\bg_x^{(j)} +\frac{1}{2\eta_x}\|\bx-\bx^{(j)}\|_2^2
			+\frac{1}{2\gamma}\|\bx-\bar\bx\|_2^2+g(\bx)
			\end{array}
			\right\}\\
			\by^{(j+1)}&=\argmin_{\by\in\mathbb{R}^q}
			\left\{ 
			\begin{array}{l}
			-\by^\top\bg_y^{(j)} +\frac{1}{\eta_y}V_y(\by,\by^{(j)})
			+\frac{1}{\lambda}V_y(\by,\bar\by)+r(\by)
				\end{array}
			\right\}.
		\end{align*}
		\normalsize
		\ENDFOR
		\STATE 	$ \widehat\bx^{(k+1)} =\bx^{(J-1)},\quad\widehat \by^{(k+1)} = \by^{(J-1)}$.
		\ENDFOR
		\STATE 
		\textbf{Output:} $\widehat\bx^{(K-1)}$
	\end{algorithmic}
\end{algorithm}

Theorem~\ref{thm:totalcomplexitySVRG} characterizes the time complexity for PG-SVRG to find a nearly $\epsilon$-stationary point. 
We define the time complexity as the total number of the stochastic gradients $(\bg_x^j,\bg_y^j)$ calculated and count the complexity of computing  $(\hbg_x^{(k)},\hbg_x^{(k)})$ as computing $n$ stochastic gradients. The proof and the exact expressions of $T$ are provided in Section~\ref{sec:prooftheorem2} in Appendix. 

\begin{theorem}
	\label{thm:totalcomplexitySVRG}
	Suppose Assumptions~\ref{assume:stochastic} and~\ref{assume:finitesum} hold. Let $T$ be the total number of iterates in Algorithm~\ref{alg:PGSVRG} and let $\tau$ be a random index uniformly sampled from $\{0,1,\dots,T-1\}$. Algorithm~\ref{alg:PGSVRG} ensures $\mathbb{E}[\|\nabla\psi_{\gamma} (\bar\bx^{(\tau)})\|^2]\leq \epsilon^2$ after $T=\tilde O(\epsilon^{-2})$ iterations with total time complexity of $\tilde {O}(n\epsilon^{-2})$ if $\mu>0$ and $\tilde {O}(n\epsilon^{-2}+\epsilon^{-6})$ if $\mu=0$.
\end{theorem}

\begin{algorithm}[h]
	\caption{PG-SVRG: Proximally Guided Stochastic Variance Reduced Gradient Method}\label{alg:PGSVRG}
	\begin{algorithmic}[1]
		\STATE \textbf{Input:} $(\bar\bx^{(0)},\bar\by^{(0)}) \in \mathcal{X}\times \mathcal{Y}^o$ and $\gamma \in (0,1/(2\rho)]$.  
		\FOR{$t=0,..., T-2$}
		\STATE  $\lambda_t=\left\{\begin{array}{ll}+\infty\text{ if }\mu>0 \\
		t+2\text{ if } \mu=0
		\end{array}
		\right.$, $\mu_x=\frac{1}{2\gamma}$, $\mu_y^t=\frac{1}{\lambda_t}+\mu$, $\Lambda_t=\frac{52\max\left\{L_x^2,L_y^2\right\}}{\min\left\{\mu_x^2,(\mu_y^t)^2\right\}}+\frac{3}{2}$.
		\STATE  
		\small
		$k_t =\left\lceil 1+ 4/3\log\big[9(t+1)^2\left(1/4+\Lambda_t/2\right)(\mu_xD_x^2+\mu_y^tD_y^2) \big]\right\rceil$.
		\normalsize
		\STATE $\bar\bx^{(t+1)}=\text{SVRG}(\bar\bx^{(t)},k_t,\lambda_t,\gamma)$.
		\ENDFOR
		\STATE Sample $\tau$ uniformly randomly form $\{0,1,\dots,T-1\}$.
		\STATE\textbf{Output:} 
		$ \bar\bx^{(\tau)}$
	\end{algorithmic}
\end{algorithm}


\section{Numerical Experiments}
\label{sec:exp}
In this section, we present numerical experiments to demonstrate the performance of our proposed methods. We first consider the distributionally robust learning problem in ~\eqref{eqn:P2} with smooth and non-smooth loss function $f_i$'s. 
Then, we will test the algorithms on the robust learning problem from multiple distributions~\eqref{eqn:P4} with non-smooth losses.

\subsection{Robust Learning with Smooth Loss}
\label{sec:expsmooth}
In this section, we consider learning from imbalanced data by solving the robust learning problem in~\eqref{eqn:P2} with $\mathcal{X}=\{\bx\in\mathbb{R}^d|\|\bx\|\leq d\}$ and $r(\by) = \theta D_{KL}(\by, \mathbf 1/n)$, where $D_{KL}$ is the KL-divergence between two vectors in the simplex $\Delta\subset\mathbb{R}^d$. We conduct two experiments with different non-convex smooth losses $f_i$'s in~\eqref{eqn:P2}.

In the first experiment, we learn a linear model $\bx$ with the loss formed by truncating the logistic loss for binary classification, i.e., $f_i(\bx)=f(\bx; \mathbf a_i, b_i)=\phi_\alpha(\ell(\bx; \mathbf a_i, b_i))$
where $\mathbf a_i\in\R^d$ denotes the feature vector, $b_i\in\{1,-1\}$ denotes its binary class label for $i=1,2,\dots,n$, 
$\ell(\bx; \mathbf a, b)=\log(1+\exp(-b\mathbf a^\top \bx))$, and $\phi_\alpha(s) =\alpha \log(1+s/\alpha)$ for  $\alpha=2$. A truncated loss makes it robust to outliers and noisy data~\cite{loh2017,DBLP:journals/corr/abs-1805-07880}, though it is not our goal to demonstrate the benefit of using truncation here. 
For this experiment, we focus on comparison with two baseline methods for solving~\eqref{eqn:P2}. Note that \eqref{eqn:P2} is also a special case of \eqref{eq:composite} with $h(\bz)=\max_{\by\in\Delta}\{\by^\top\bz-r(\by)\}$ and $\bc(\bx)=(f_1(\bx),...,f_n(\bx))^\top$. Hence, we compare our method with 
the proximal linear (PL) method \citep{Drusvyatskiy2018} that solves \eqref{eq:composite} using the updating scheme \eqref{eq:compositesub}.
\cite{Drusvyatskiy2018} suggested using the accelerated gradient method to solve~\eqref{eq:compositesub} which 
has a high computational cost when $n$ is large. Hence, we also use the SMD and SVRG methods to solve \eqref{eq:compositesub} to reduce the complexity of the PL method through sampling $\nabla f_i$. We denote the PL methods using the SMD and SVRG methods to solve \eqref{eq:compositesub} as PL-SMD and PL-SVRG, respectively.

\begin{table}[ht]
	\caption{Statistics Of Datasets Used}
	\vspace{-0.2in}
	\begin{center}
		{	\begin{tabular}{|l|c|c|c|c|}
				\hline
				Datasets & \#Train & \#Test & \#Feat. & \multicolumn{1}{l|}{- : +} \\ \hline
				covtype & 217594 & 30000 & 54 & 1:7.5 \\ \hline
				connect-4 & 29108 & 32000 & 126 & 1:5.3 \\ \hline
				protein & 10263 & 4298 & 357 & 1:4.0 \\ \hline
				rcv1.binary & 12591 & \multicolumn{1}{l|}{677399} & 47236 & 1:5.1 \\ \hline
		\end{tabular}}
	\end{center}
	\label{table:data-stats}
\end{table} 
\vspace{-0.2in}

We perform the comparisons on four imbalanced datasets from the LIBSVM library. Table~\ref{table:data-stats} contains information about the four datasets used in  the experiments. The information includes the sizes of training and testing sets, the number of features, and the ratio of negative instances to positive instances in training sets (testing tests are almost balanced). For datasets \emph{covtype} and \emph{connect-4}, 
we selected two imbalanced classes from the original data and split them appropriately into imbalanced training and balanced testing sets. 
Datasets \emph{protein} and \emph{rcv1.binary} were already split into training and testing sets, so we select two classes and randomly removed some negative instances from the training sets to make them imbalanced. 

In this experiment, we set $\theta=10$. Mini-batches of size $5$, $10$, $10$, and $200$ were chosen for \emph{protein}, \emph{rcv1.binary}, \emph{connect-4}, and \emph{covtype}, repectively, when we compute stochastic gradients in all methods. We implement PG-SMD in case D1. The values of some control parameters in the algorithms are selected from a discrete grid of candidate values using the objective value of \eqref{eqn:P2} at termination as the criterion. Those control parameters include $\eta$ in \eqref{eq:compositesub} for the two PL methods, $\gamma$ in both PG-SMG and PG-SVRG, the ratios $D_x/M_x$ and $D_y/M_y$ in the expressions of $\eta_x$ and $\eta_y$ in SMD as well as $J$, $K$, $\eta_x$ and $\eta_y$ in SVRG. 


{
	\begin{figure*}[t]
		\centering
		\begin{subfigure}
			\centering
			\includegraphics[scale=0.255]{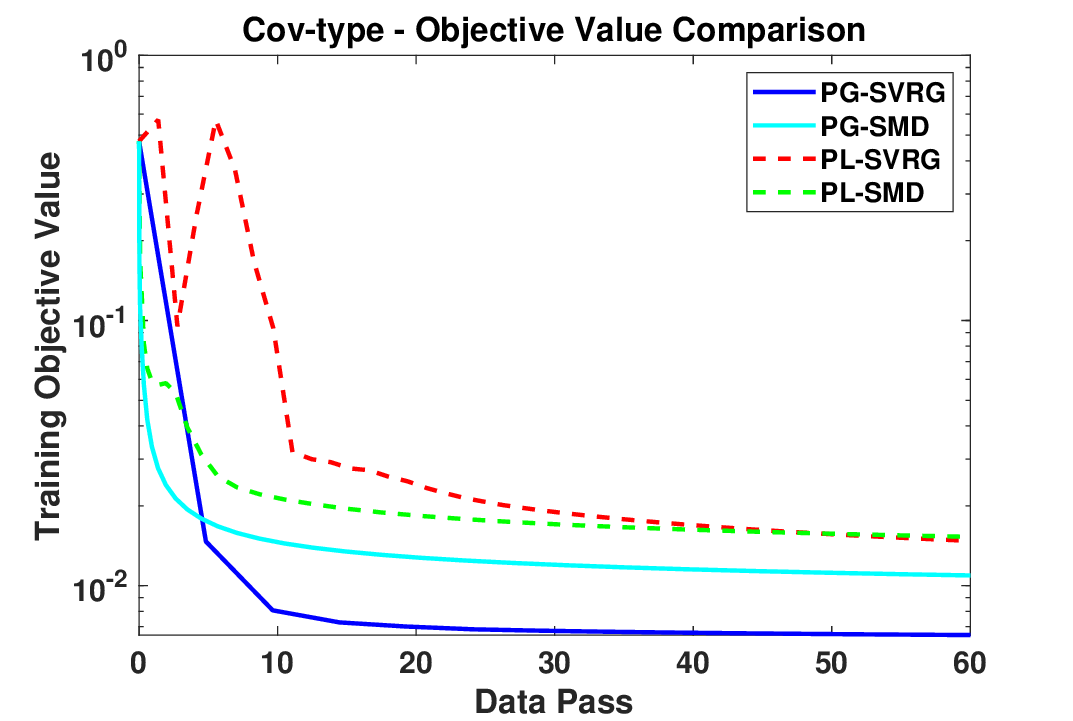}
		\end{subfigure}%
		\begin{subfigure}
			\centering
			\includegraphics[scale=0.255]{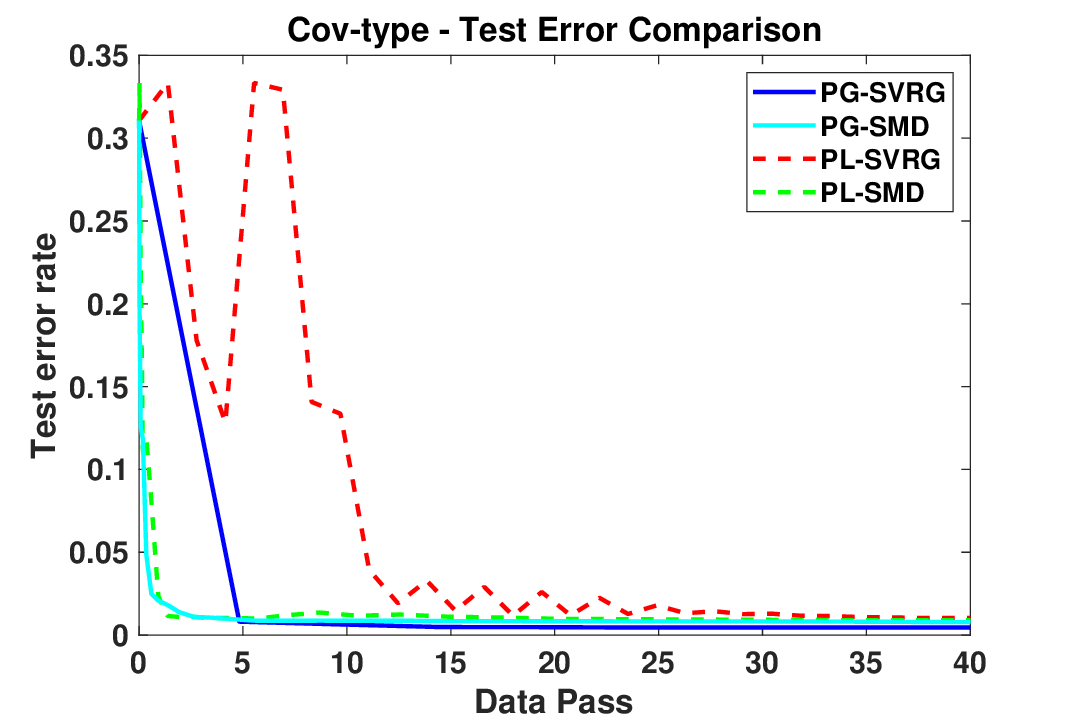}
		\end{subfigure}%
		\begin{subfigure}
			\centering
			\includegraphics[scale=0.255]{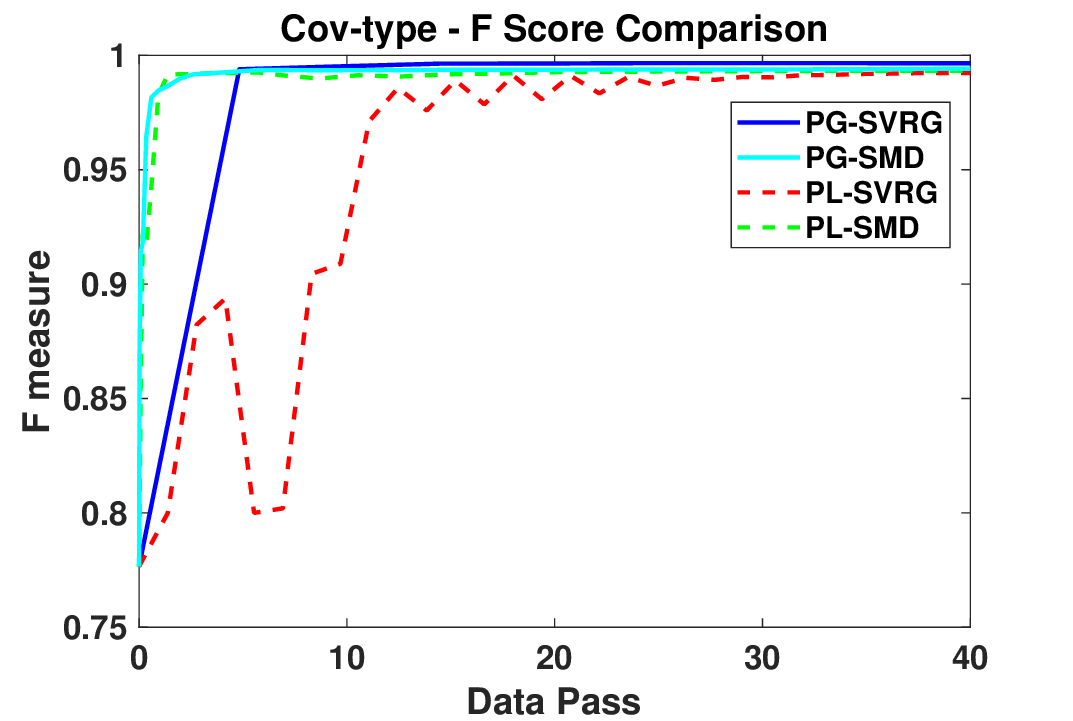}
		\end{subfigure}%
		\begin{subfigure}
			\centering
			\includegraphics[scale=0.25]{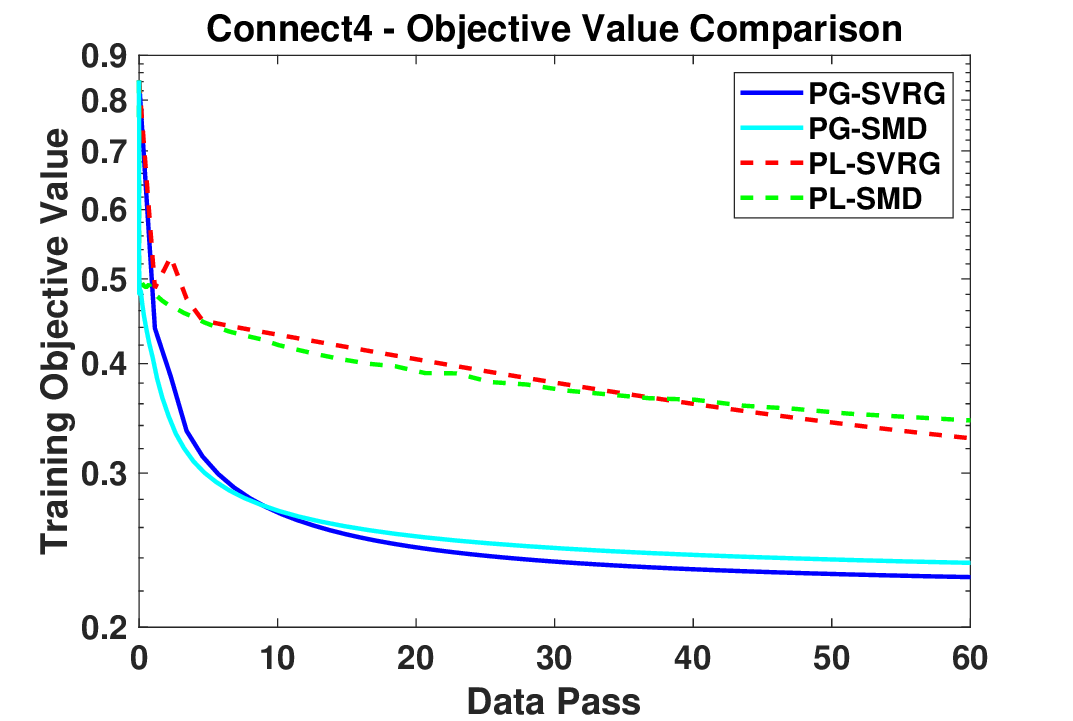}
		\end{subfigure}
		\begin{subfigure}
			\centering
			\includegraphics[scale=0.25]{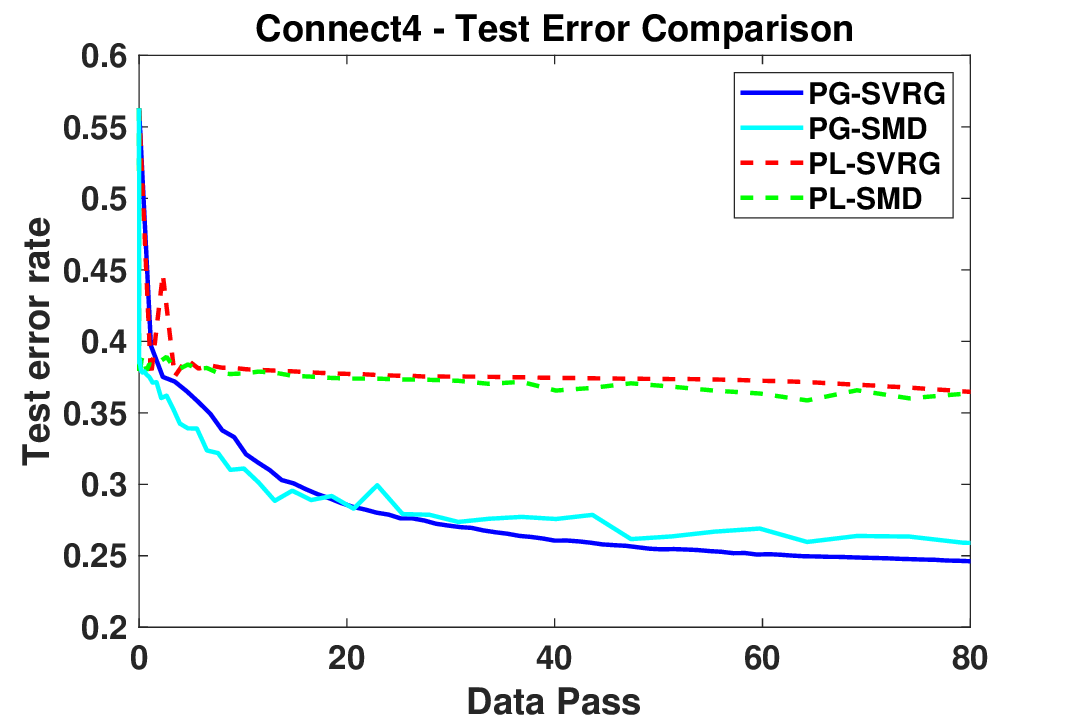}
		\end{subfigure}
		\begin{subfigure}
			\centering
			\includegraphics[scale=0.25]{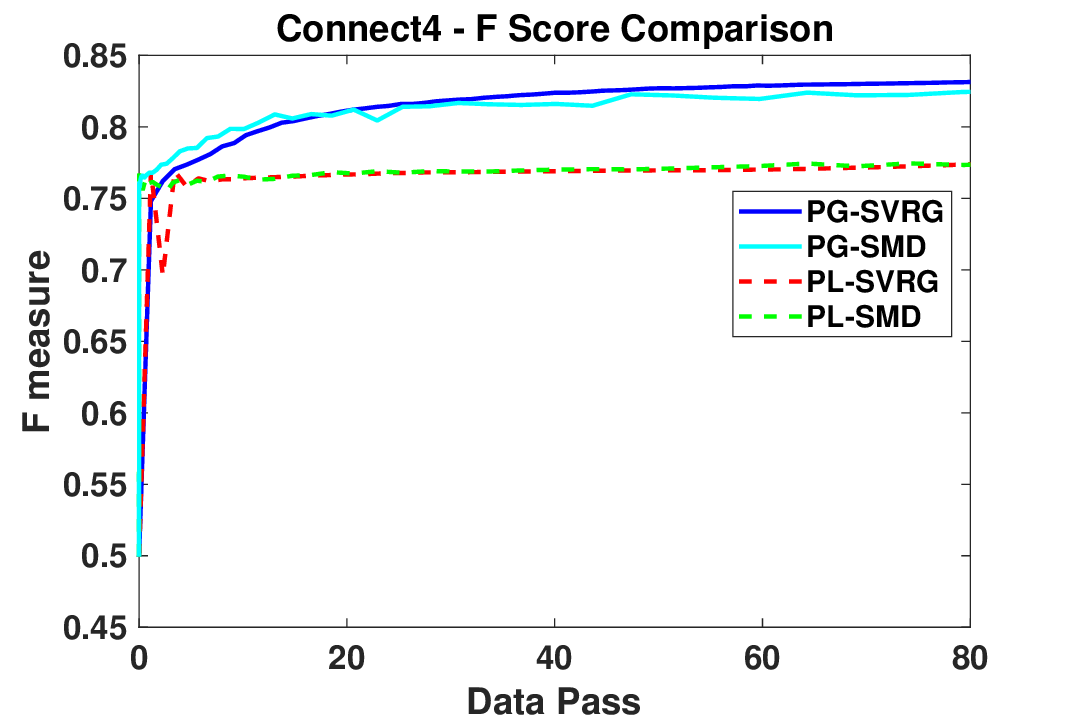}
		\end{subfigure}
		\begin{subfigure}
			\centering
			\includegraphics[scale=0.25]{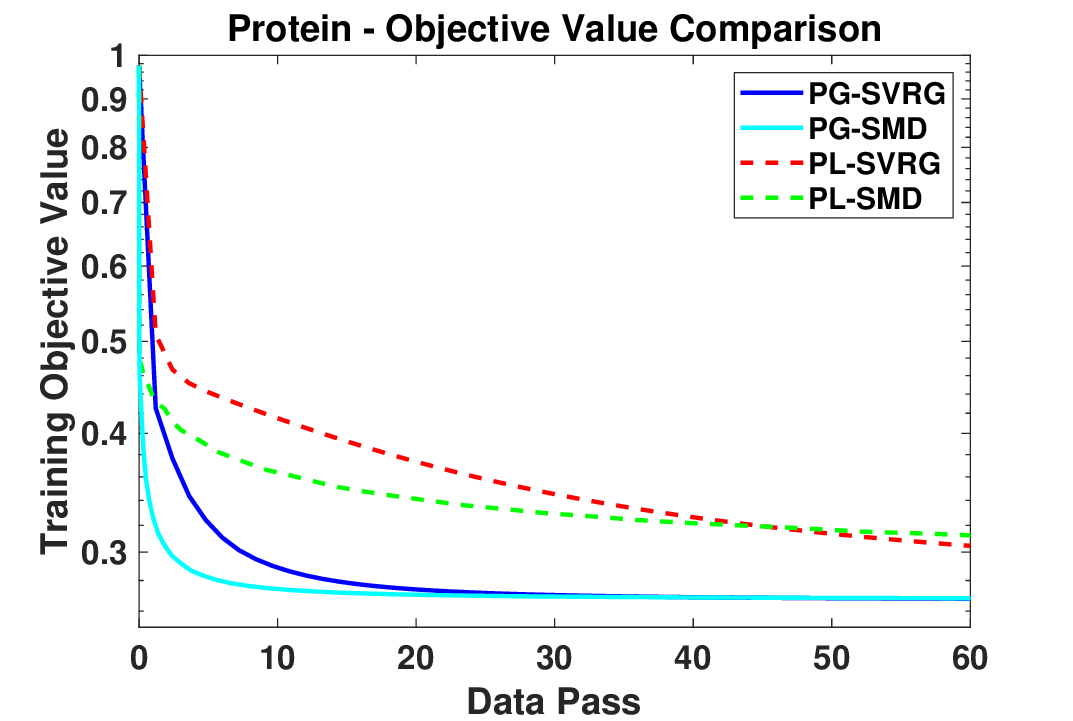}
		\end{subfigure}
		\begin{subfigure}
			\centering
			\includegraphics[scale=0.25]{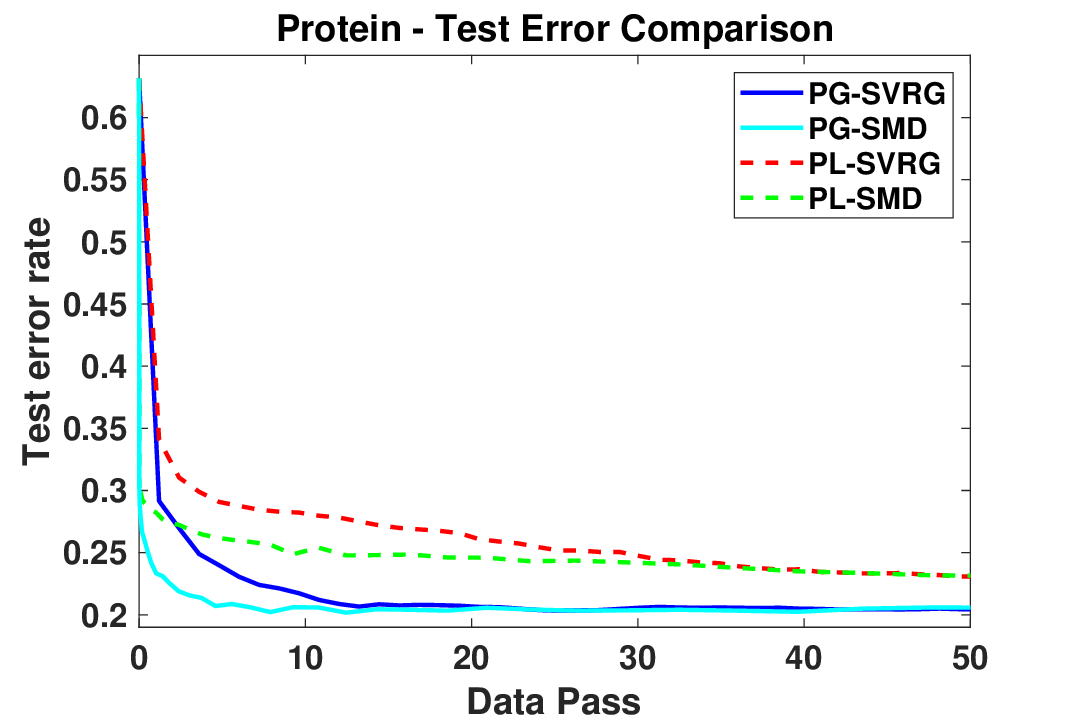}
		\end{subfigure}
		\begin{subfigure}
			\centering
			\includegraphics[scale=0.25]{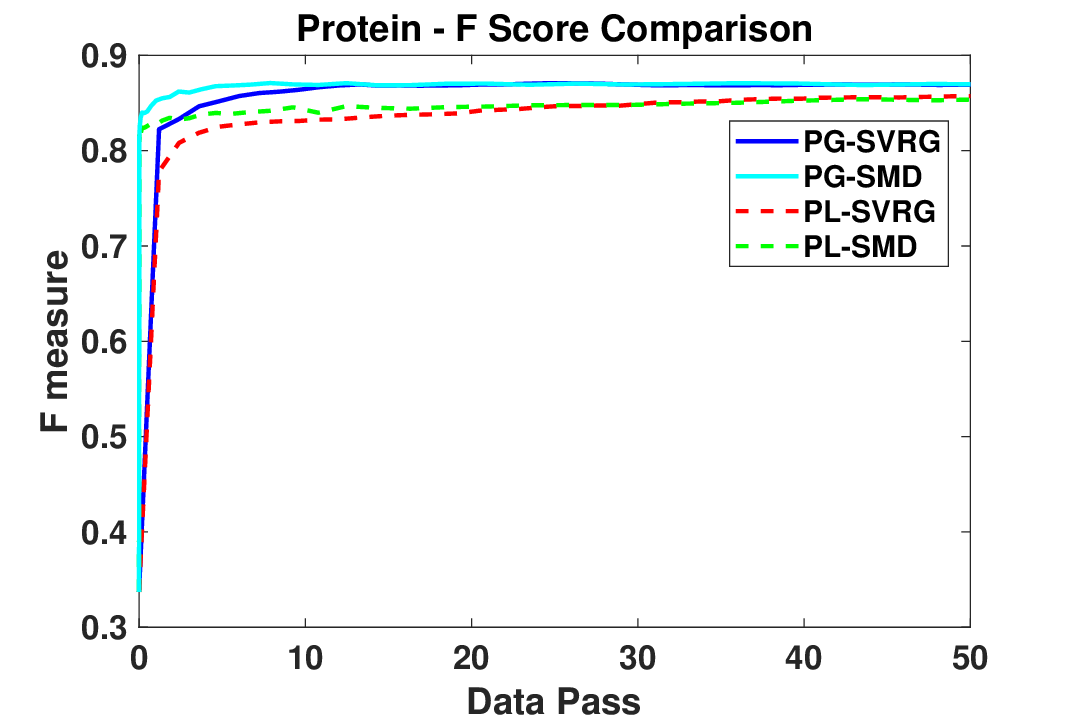}
		\end{subfigure}
		\begin{subfigure}
			\centering
			\includegraphics[scale=0.25]{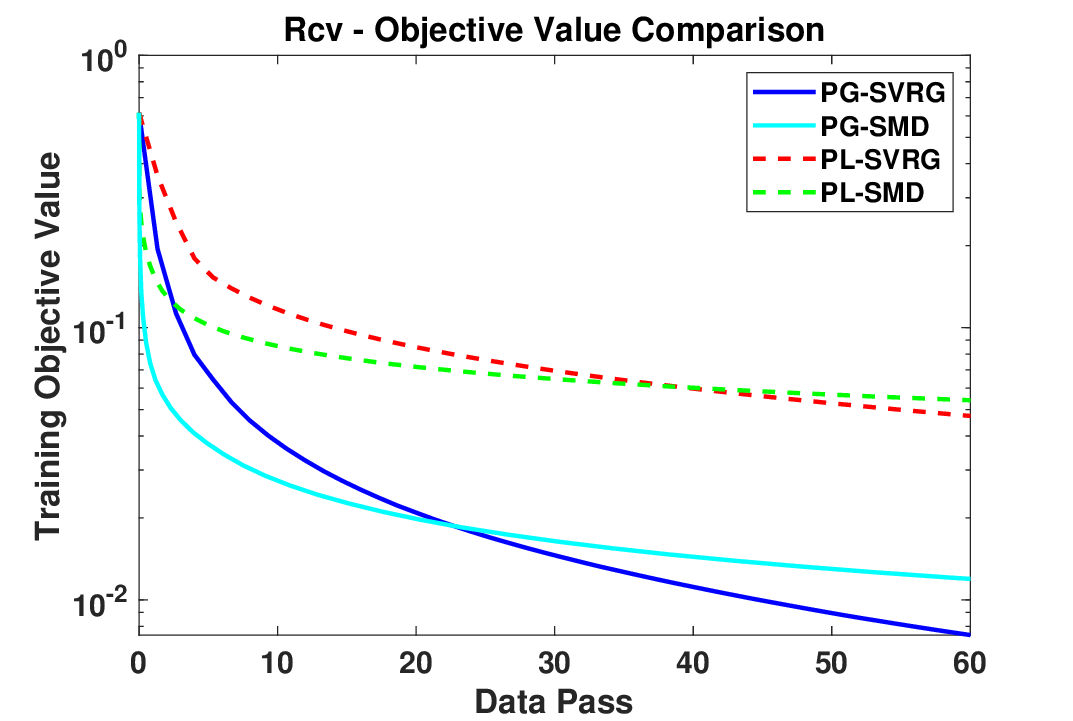}
		\end{subfigure}
		\begin{subfigure}
			\centering
			\includegraphics[scale=0.25]{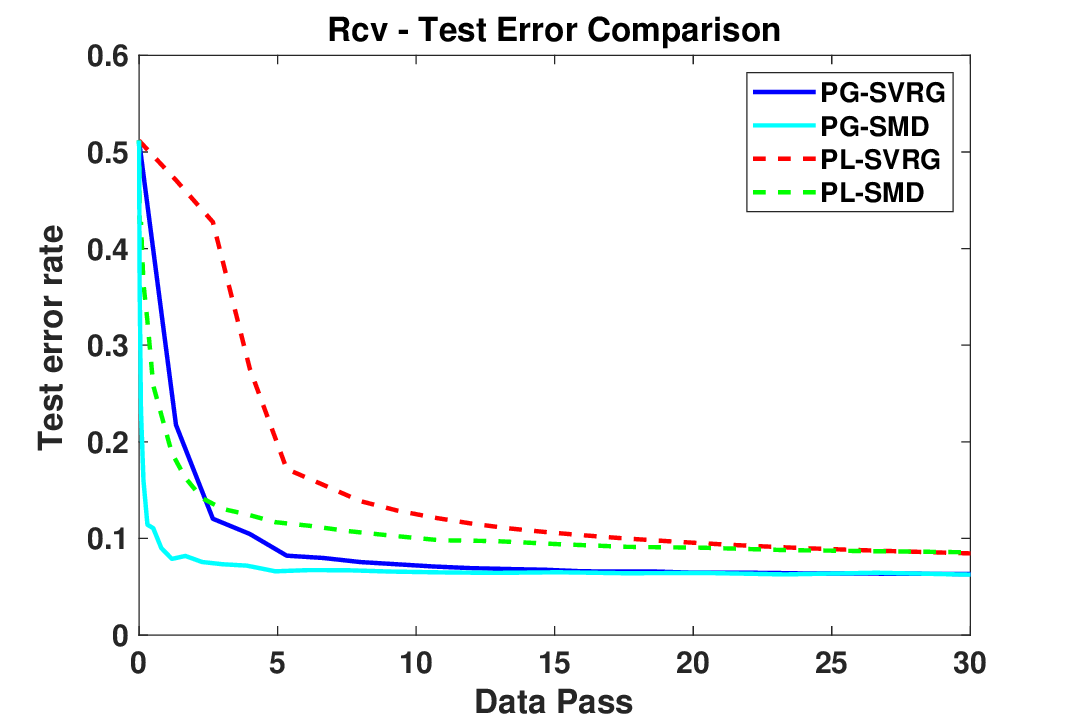}
		\end{subfigure}
		\begin{subfigure}
			\centering
			\includegraphics[scale=0.25]{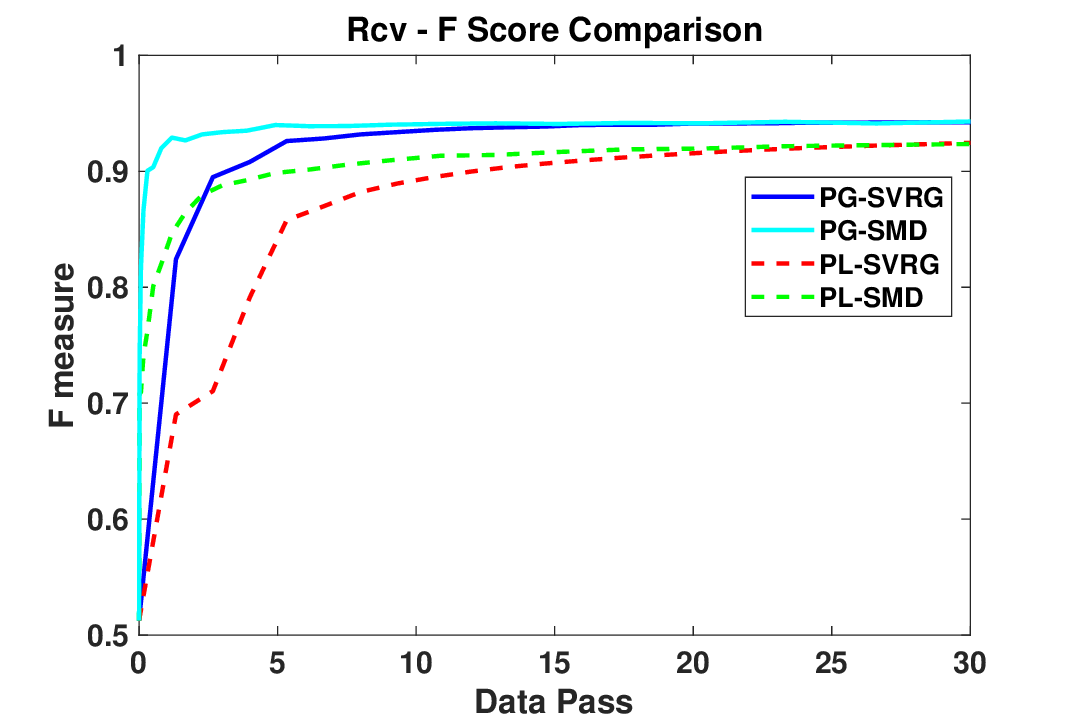}
		\end{subfigure}
		{\centering
			\caption{Comparisons of PL and PG methods on robust learning with imbalanced
				data and smooth loss. }\label{fig:simulations}}
	\end{figure*}
}

Figure \ref{fig:simulations} shows the objective value $\psi(\bx)$ of \eqref{eqn:P2}, testing classification error rate and F-score as the number of data passes increases. On these datasets, the proposed PG methods are more advantageous than the PL methods. This is mainly because the gradient information $\nabla \mathbf{c}(\bx^{(t)})$ (subject to stochastic noise) used in \eqref{eq:compositesub} in the PL methods is only updated in each outer loop while our PG methods update this in each inner loop, which ensures that the solutions are updated with the latest gradient information. 
Overall, PG-SVRG does well in the long run in reducing the objective value, while PG-SMD is faster in yielding a low test error rate. 
Additional comparisons are given in Figure~\ref{fig:additionaltime}, where we show the objective value (the training loss) obtained by the PL and PG methods at different wall-clock times in seconds. Figure~\ref{fig:additionaltime} also shows that the PG methods are more time efficient than the PL methods.

\begin{figure*}[t!]
	\centering
	\includegraphics[scale=0.30]{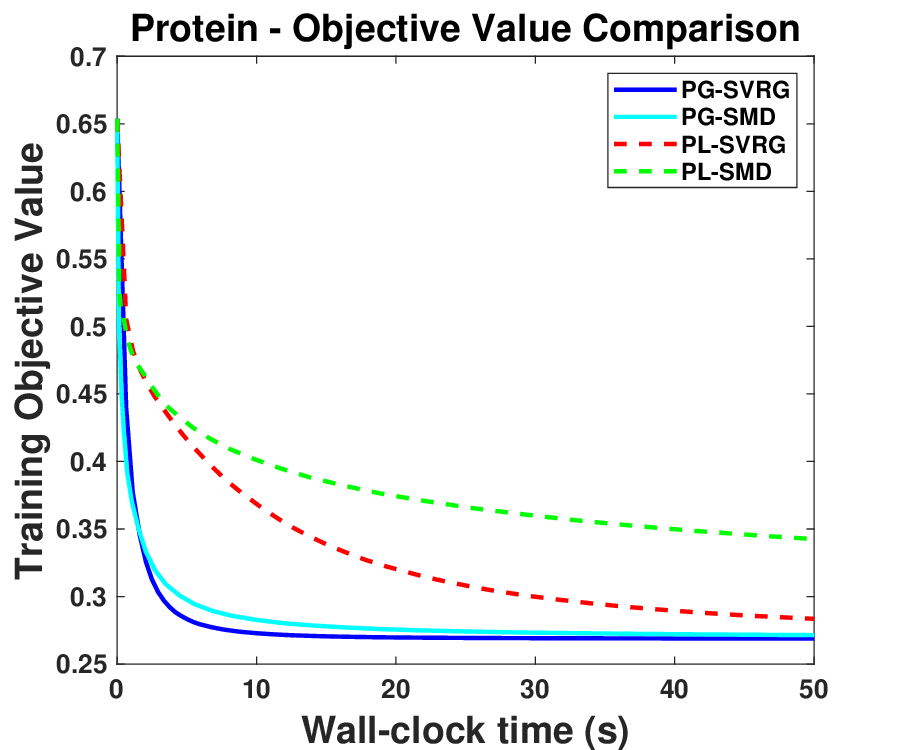}
	\includegraphics[scale=0.30]{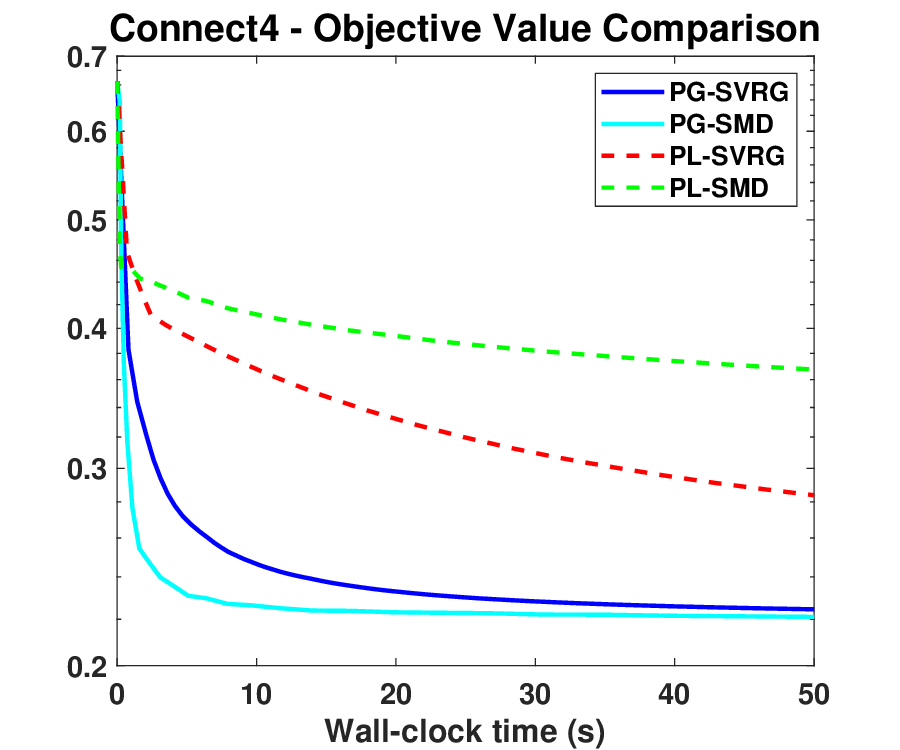}
	\includegraphics[scale=0.30]{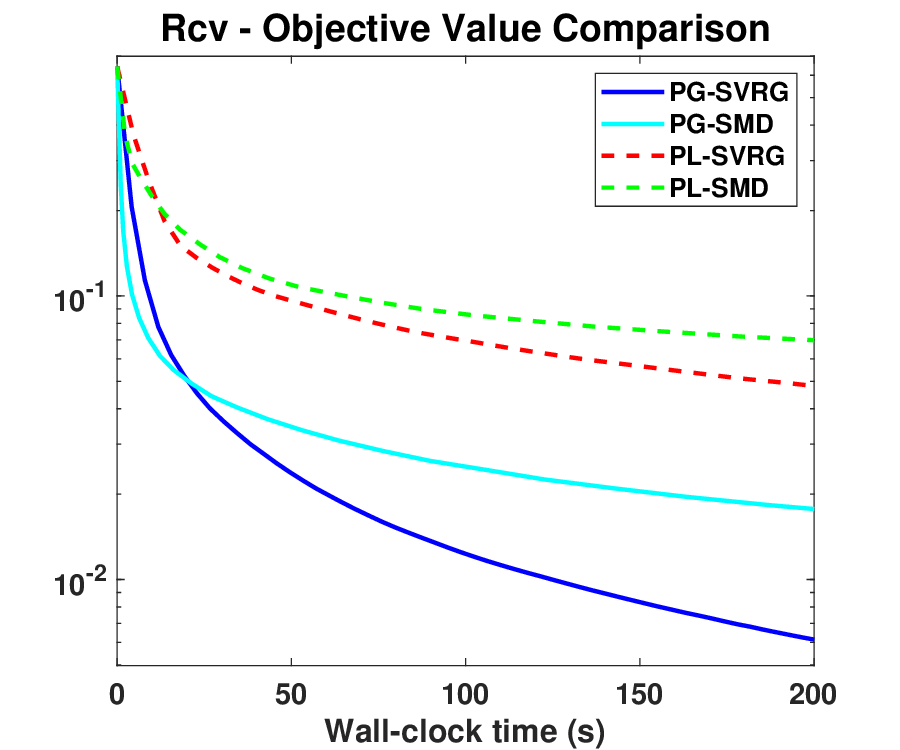}
	\vspace{-0.1in}
	\caption{Comparisons of PL and PG methods in objective value at different wall-clock times in seconds on robust learning with imbalanced data and smooth loss.}
	\label{fig:additionaltime}
	\vspace{-0.15in}
\end{figure*}




\begin{figure*}[t!]
	\centering
	\includegraphics[scale=0.35]{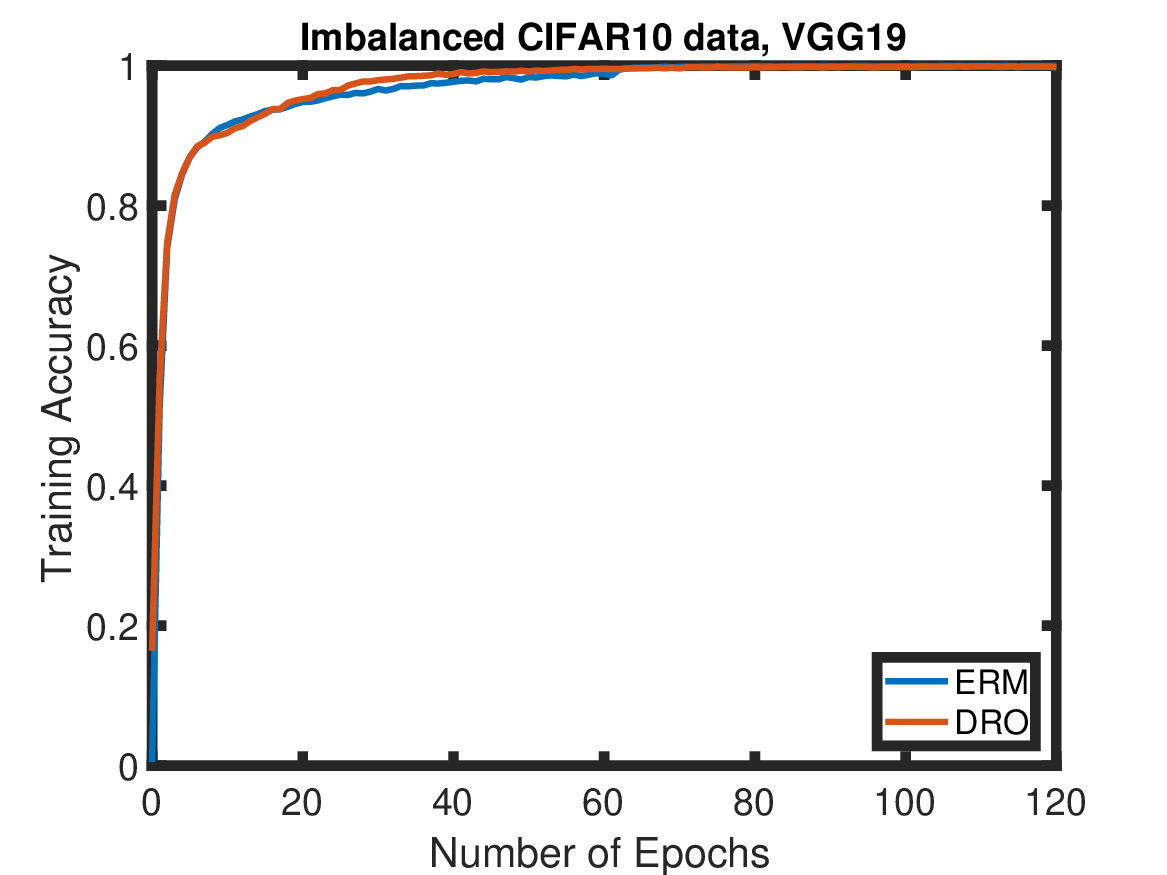}
	\includegraphics[scale=0.35]{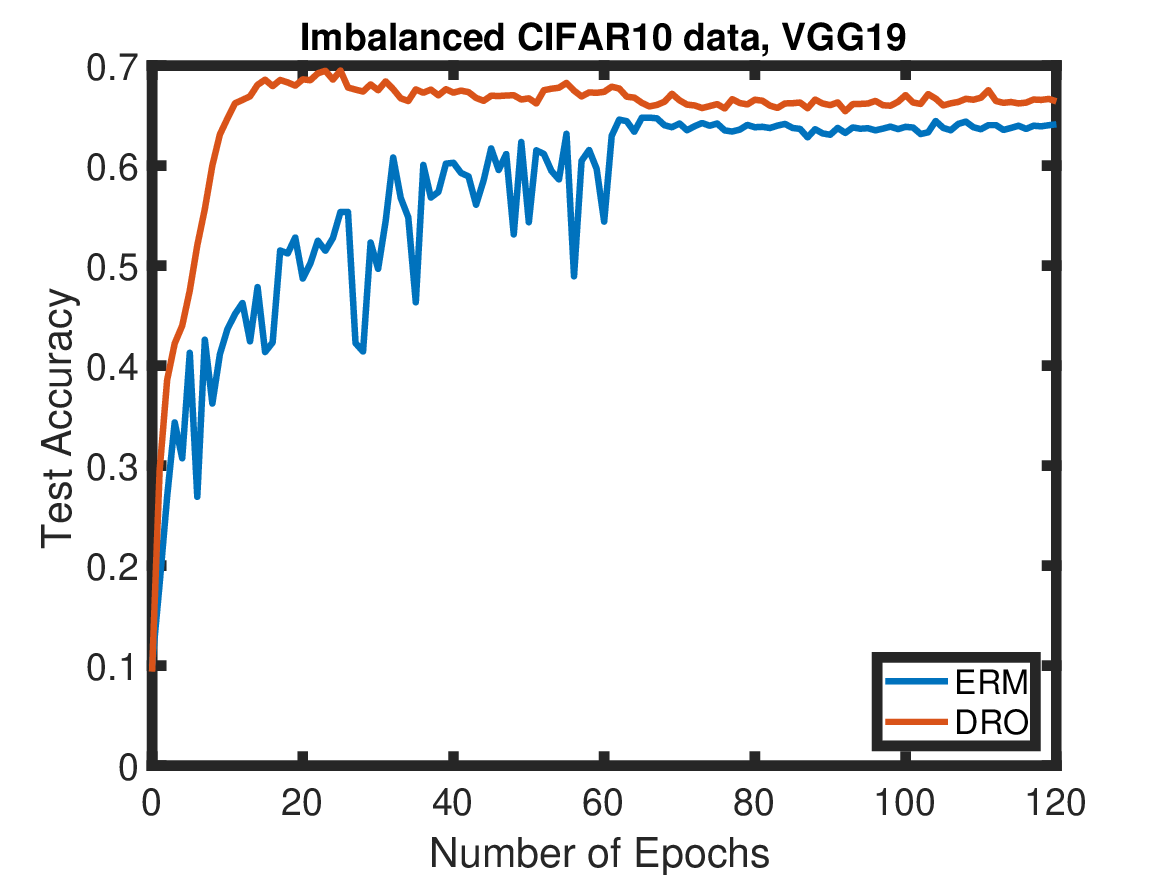}\\
	\includegraphics[scale=0.35]{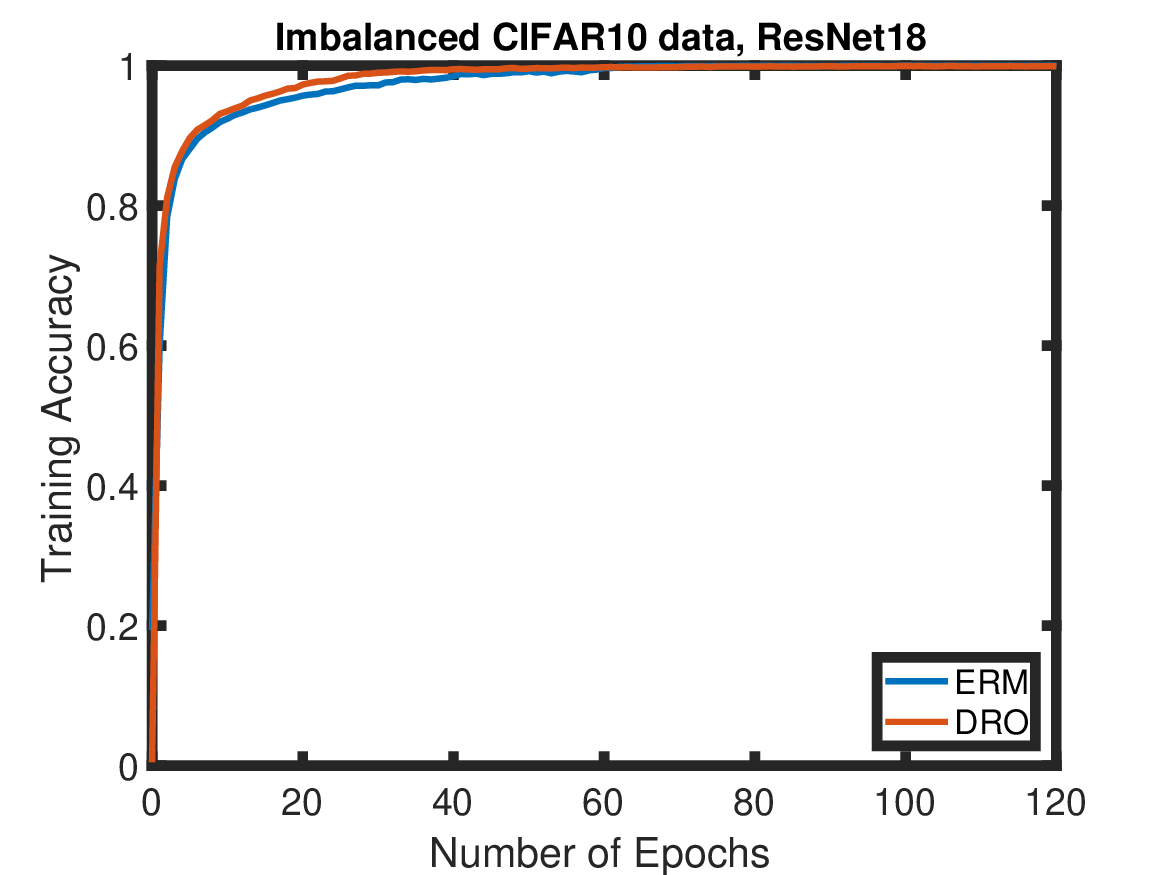}
	\includegraphics[scale=0.35]{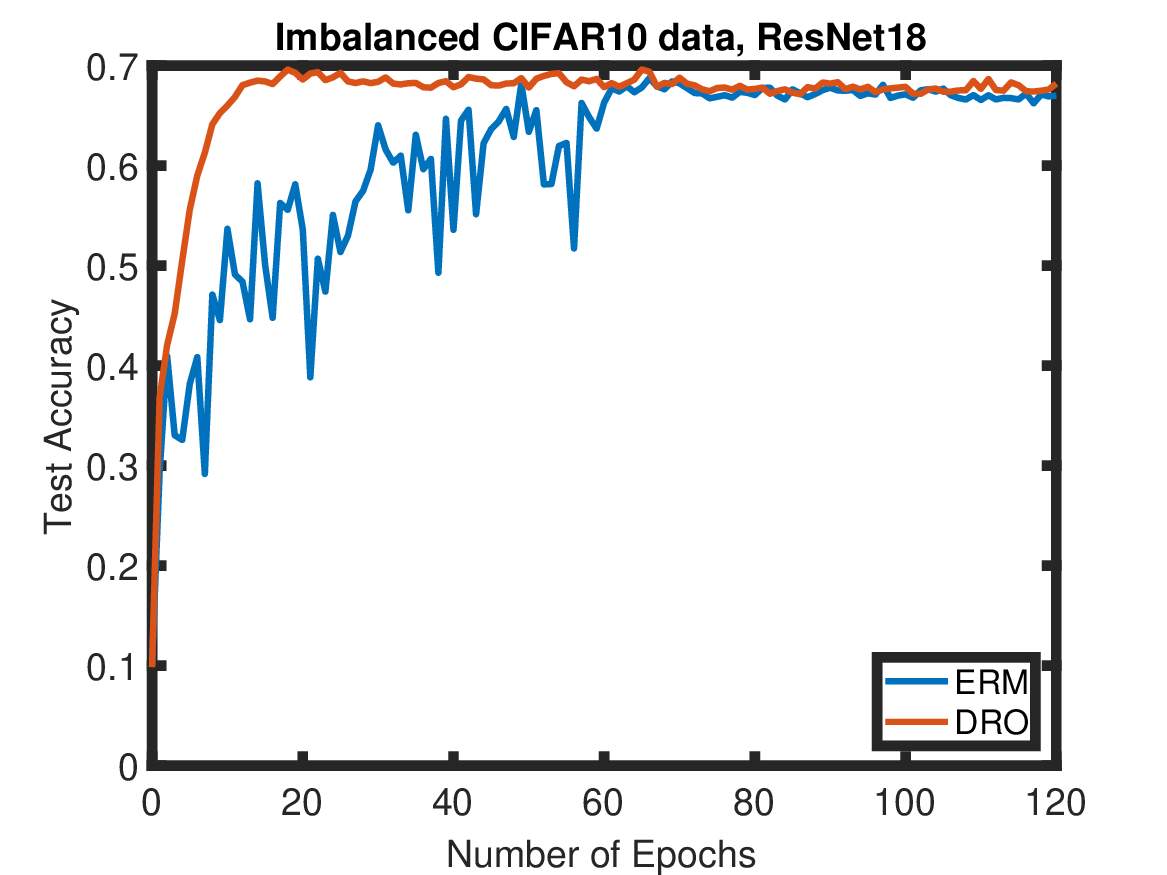}
	\vspace{-0.1in}
	\caption{Comparison of ERM and DRO for robust learning ResNet18 and VGG19 networks on imbalanced data.}
	\label{fig:1}
\end{figure*}

In the second experiment, we learn a non-linear model with the loss defined by a deep neural network. For this experiment, we focus on demonstrating the power of the proposed method for improving the generalization performance. We do that by solving the distributionally robust optimization (DRO) (\ref{eqn:P2}) in comparison with traditional empirical risk minimization (ERM) by SGD.  
To this end, we use an imbalanced CIFAR10 data and two popular deep neural networks (ResNet18~\citep{he2016deep} and VGG19~\citep{simonyan2014very}). The original training data of CIFAR10 has ten classes, each of which has 5000 images. We remove 4900 images from five classes to make the training set imbalanced. The test data remains intact.

\begin{wrapfigure}{h!}{0.5\textwidth}
	\begin{center}
		\includegraphics[scale=0.4]{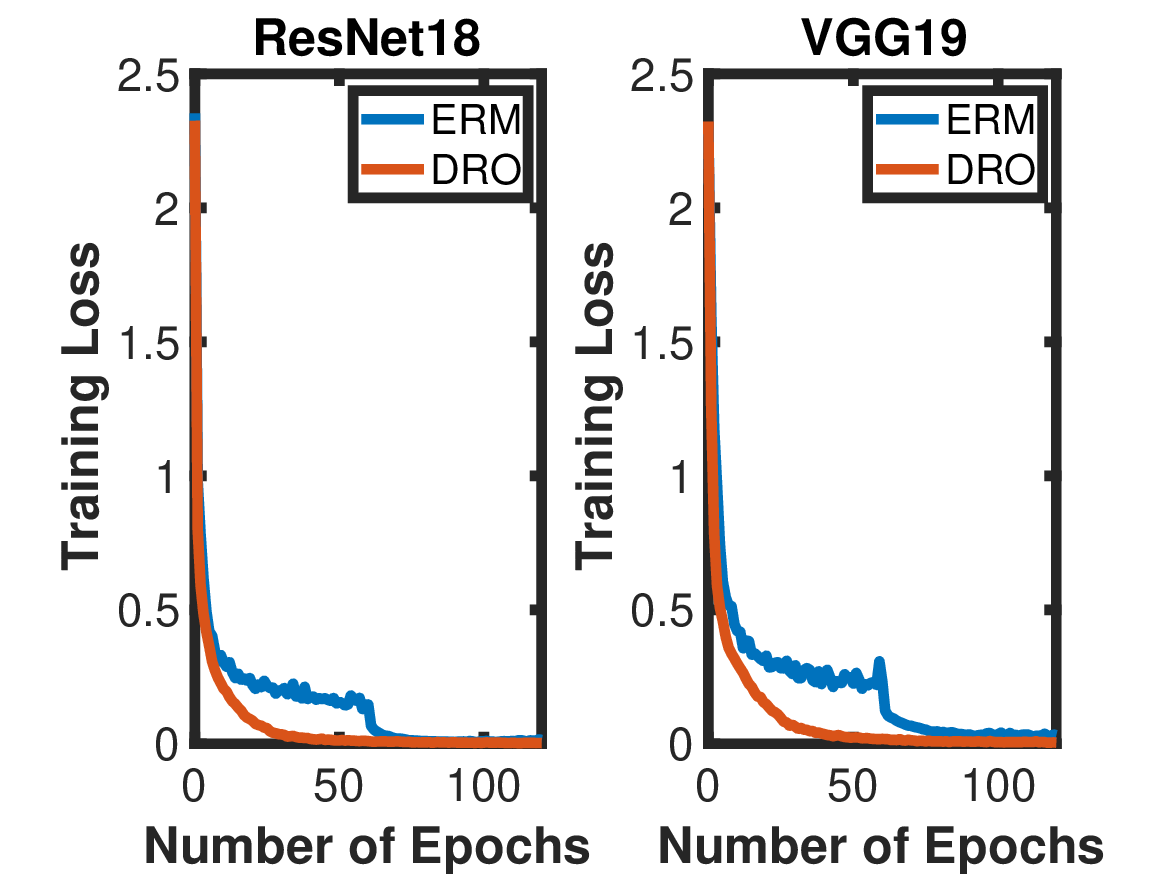}
		\vspace{-0.1in}
		\caption{Comparison of ERM and DRO models in training loss at different epochs for robust learning of ResNet18 and VGG19 networks on imbalanced data.}
		\label{fig:additional}
	\end{center}
	\vspace{-0.15in}
\end{wrapfigure}

The value of $\theta$ is fixed to $5$ in this experiment, and the mini-batch size is fixed to 128 for both methods.   We use SGD to solve ERM with a stepsize of 0.1 for epochs $1\sim 60$ and $0.01$ for epochs $61\sim 120$ following the practical training strategy~\citep{he2016deep}, where one epoch means one pass of data. We implement PG-SMD in case D1 using the similar practical training strategy except that we choose $(\eta_x,\eta_y)=(0.1, 10^{-5})$, $(0.05, 5\times 10^{-6})$, $(0.03, 3\times 10^{-6})$, $(0.025, 2.5\times 10^{-6})$ in epochs $1\sim 5$, $6\sim 25$, $26\sim 70$, $70\sim 120$ respectively. $\gamma$ is selected by grid search as in the first experiment.
The curves for training and testing accuracy are plotted in Figure~\ref{fig:1}, which show that a robust optimization scheme is considerably better than ERM when dealing with imbalanced data in these two specific examples. Additional comparison is given in Figure~\ref{fig:additional}, where we show the training loss (the objective value) obtained by the ERM and DRO methods at different epochs.


\subsection{Robust Learning with Non-Smooth Loss}
\label{sec:expnonsmooth}
In this section, we perform numerical experiments on the same robust learning problem as in Section~\ref{sec:expsmooth} except that we use non-smooth loss function $f_i$'s in \eqref{eqn:P2} this time. In particular, we consider the loss function for binary classification defined by a two-layer fully-connected neural network where the loss in the output layer is calculated using the hinge loss. Let $\mathbf a_i\in\R^d$ denote the feature vector and $b_i\in\{1,-1\}$ denote its binary class label for $i=1,2,\dots,n$. Let $p$ be the number of neurons in the hidden layer and $\bx:=(\mathbf{X}_1,\bx_2)$ be the weights with $\mathbf{X}_1\in\mathbb{R}^{p\times d}$ and $\bx_2\in\mathbb{R}^{p}$. Then the loss from the $i$th data point is defined as 
\begin{eqnarray}
\label{eq:annloss}
f_i(\bx)=f(\bx; \mathbf a_i, b_i)=\max\left(1-b_i\sigma(\mathbf{X}_1\mathbf a_i)^\top\bx_2,0\right), 
\end{eqnarray}
where $\sigma(z):\mathbb{R}^p\rightarrow \mathbb{R}^p$ is the component-wise sigmoid function, i.e., 
\begin{eqnarray}
	\label{eq:sigmoid}
	\sigma(z)=\left(\frac{1}{1+\exp(-z_i)}\right)_{i=1}^p.
\end{eqnarray}
It can be proved that \eqref{eqn:P2} with this loss function is still weakly convex in $\bx$. However, since $f_i$'s are non-smooth, we cannot apply the PG-SVRG, PL-SMD, and PL-SVRG methods in the previous subsection to \eqref{eqn:P2}. An alternating stochastic gradient descent (Alter-SGD) method is proposed in \cite{boct2020alternating}, which can be applied to \eqref{eqn:P2} with non-smooth $f_i$'s. Hence, we will focus on comparing PG-SMD and Alter-SGD. 

We perform the comparisons on imbalanced datasets \emph{covtype} and \emph{connect-4} whose construction is described in Section~\ref{sec:expsmooth} and statistics can be found in Table~\ref{table:data-stats}. In this experiment, we set $\theta=10$ and $p=10$. Mini-batches of size $200$ and $100$ were chosen for \emph{covtype} and \emph{connect-4}, repectively, when we compute stochastic gradients in both methods. We implement PG-SMD in case D1. The values of control parameters, including  $\gamma$ and the ratios $D_x/M_x$ and $D_y/M_y$ in PG-SMG as well as $\eta_x$ and $\eta_y$ in PG-SMD and Alter-SGD, are selected in the same way as in Section~\ref{sec:expsmooth}.

Figure \ref{fig:1nonsmooth} shows the objective value $\psi(\bx)$ of \eqref{eqn:P2}, testing classification error rate and F-score as the number of data passes increases. The proposed PG-SMD method is a little more advantageous than the Alter-SGD method on these specific examples. Overall, PG-SMD reduces the objective value slightly faster than Alter-SGD when the number of data passes is small.

\begin{figure*}[t!]
	\centering
	\includegraphics[scale=0.33]{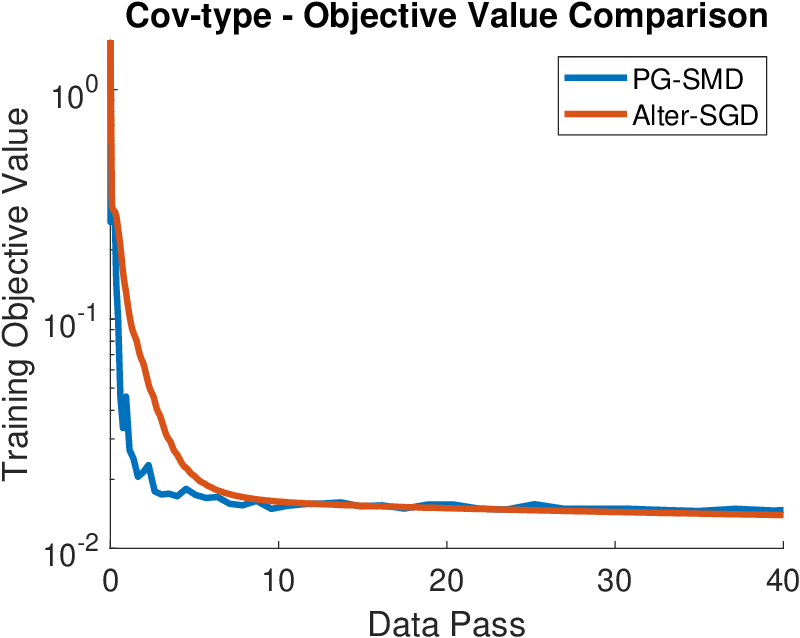}
	\includegraphics[scale=0.33]{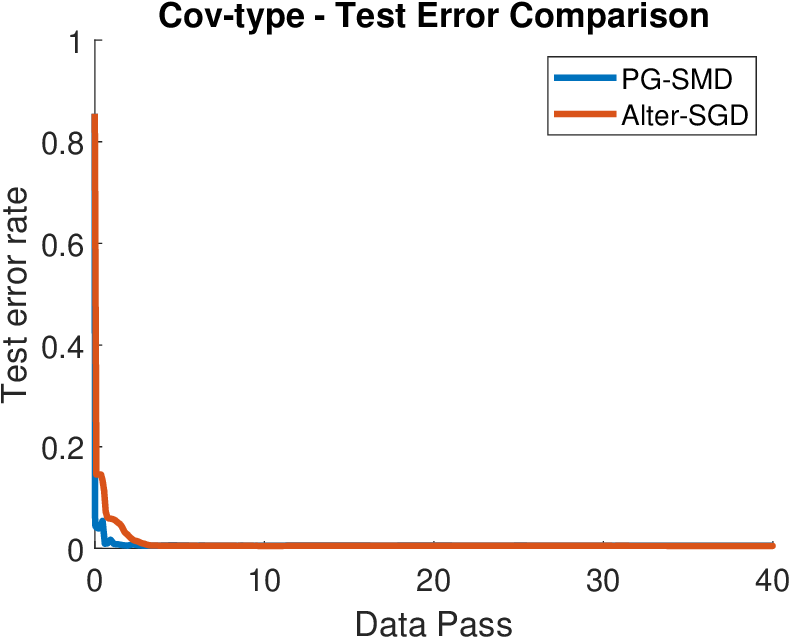}
	\includegraphics[scale=0.33]{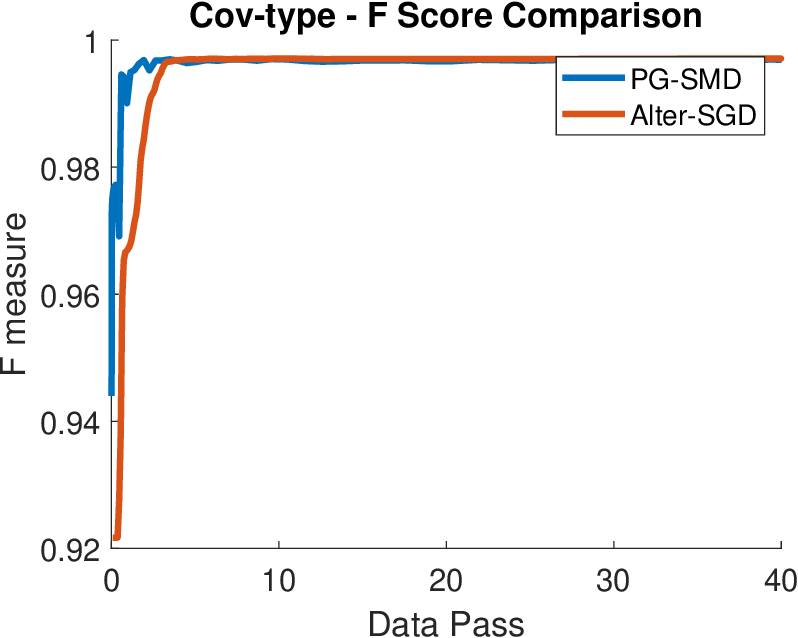}\\
	\includegraphics[scale=0.33]{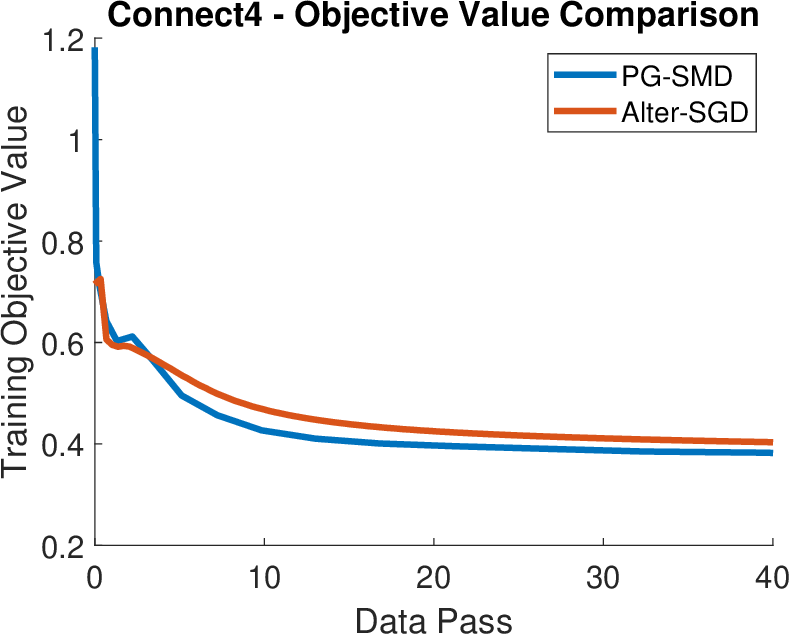}
	\includegraphics[scale=0.33]{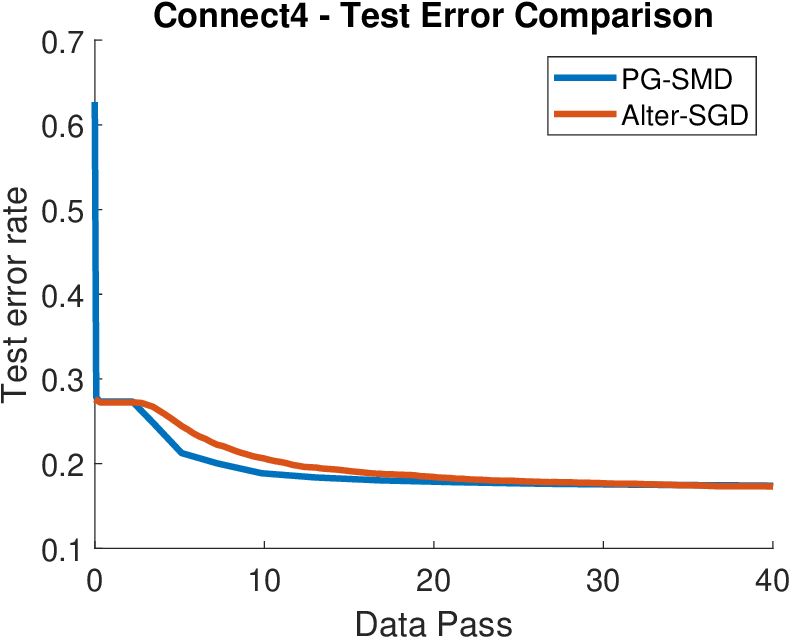}
	\includegraphics[scale=0.33]{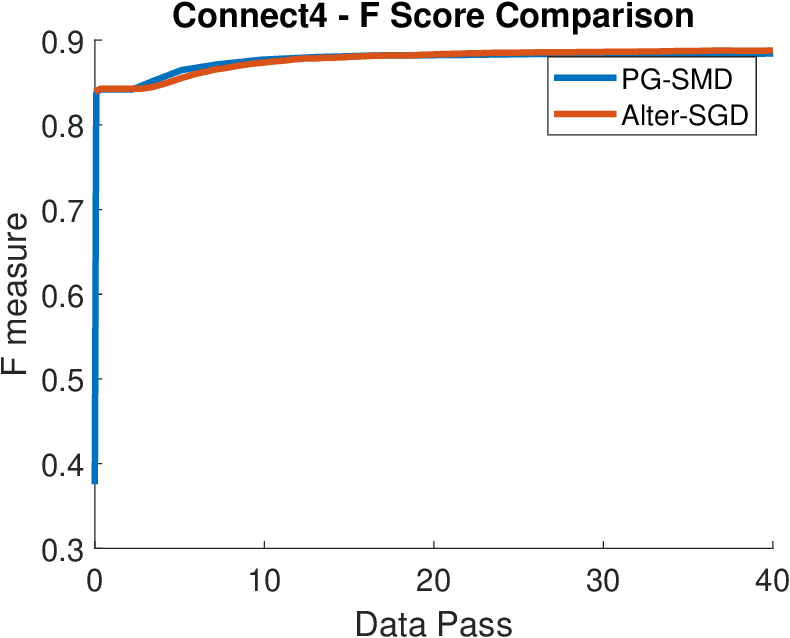}
	\vspace{-0.1in}
	\caption{Comparisons of PG-SMD and Alter-SGD methods on robust learning with imbalanced
		data and non-smooth loss. }
	\label{fig:1nonsmooth}
\end{figure*}

\subsection{Robust Learning from Multiple Distribution}
\label{sec:expexp}
In this section, we present the numerical experiments on the robust learning problem from multiple distributions described in \eqref{eqn:P4} with $m=5$ different data distributions $P_i$, $i=1,\dots,5$. We consider the case when $P_i$'s are continuous distributions and we generate $P_i$'s by simulation. 

Let $\bxi=(\mathbf a,b)$ be a random data point from one of the five distributions, where $\mathbf a=(a_1,\dots,a_{50})^\top\in\mathbb{R}^{50}$ is the feature vector and $b\in\{-1,1\}$ is its label. If $\mathbf a$ is generated from $P_i$, let the coordinates $(a_{10(i-1)+1},a_{10(i-1)+2},\cdots,a_{10i})$ and $(a_{10(j-1)+1},a_{10(j-1)+2},\cdots,a_{10j})$ be generated independently from a uniform distribution on $[-1,1]$, where $j= i+1$ if $i\leq 4$ and $j= 1$ if $i= 5$. Then let the remaining coordinates be zeros. Next, we generate the binary label $b$ corresponding to $\mathbf a$ using a two-layer fully-connected neural network. To do that, we first randomly generate the true weights $\bx^*=(\mathbf{X}_1^*,\bx_2^*)$ with $\mathbf{X}_1^*\in\mathbb{R}^{10\times 50}$ and $\bx_2^*\in\mathbb{R}^{10}$ and each entry of $\mathbf{X}_1^*$ and $\bx_2^*$ generated independently from a uniform distribution on $[-1,1]$. Then, we fix $\bx^*$ and, for each generated feature vector $\mathbf a$, we generate a noise $z$ from a uniform distribution on $[-0.01,0.01]$ and generate the label $b$ as 
$$
b=\text{sign}\left(\sigma(\mathbf{X}_1^*\mathbf a)^\top\bx_2^*+z\right).
$$
With the data distribution $\bxi=(\mathbf a,b)\sim P_i$ for $i=1,\dots,5$ defined in this way, we solve \eqref{eqn:P4} with $\mathcal{X}=\{\bx=(\mathbf{X}_1,\bx_2)| \|\mathbf{X}_1\|_F\leq 50,~\|\bx_2\|_2\leq 50 \}$ and $F(\bx; \bxi)=f(\bx; \mathbf a, b)$, where $f(\bx; \mathbf a, b)$ is defined as in \eqref{eq:annloss}. 

It can be proved that \eqref{eqn:P4} with this loss function is weakly convex in $\bx$. However, since $F$ is non-smooth and each $f_i$ is defined with an expectation rather than a finite sum, we cannot apply PG-SVRG or the PL methods, so we will again focus on comparing PG-SMD and Alter-SGD. A mini-batch of size $20$ was used when we compute stochastic gradients in both methods. Each sample $\bxi=(\mathbf a,b)$ in a mini-batch is generated using $\bx^*$ and the aforementioned procedure. The values of control parameters of both methods are selected in the same way as in Section~\ref{sec:expnonsmooth}. Besides the mini-batches used in optimization, we generate $5000$ data points from each $P_i$ separately and use them to closely approximate the objective value $\psi(\bx)$ of \eqref{eqn:P4}, testing classification error rate and F-score. These values are shown in Figure \ref{fig:1expect} as the number of data passes increases. Similar to \eqref{fig:1nonsmooth}, the proposed PG-SMD method is better than the Alter-SGD method in this specific example. 

\begin{figure*}[t!]
	\centering
	\includegraphics[scale=0.33]{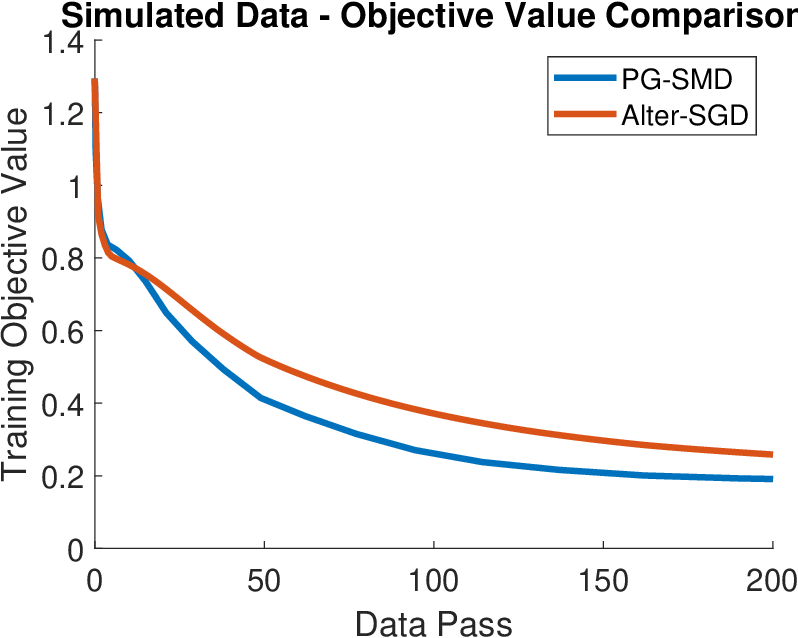}
	\includegraphics[scale=0.33]{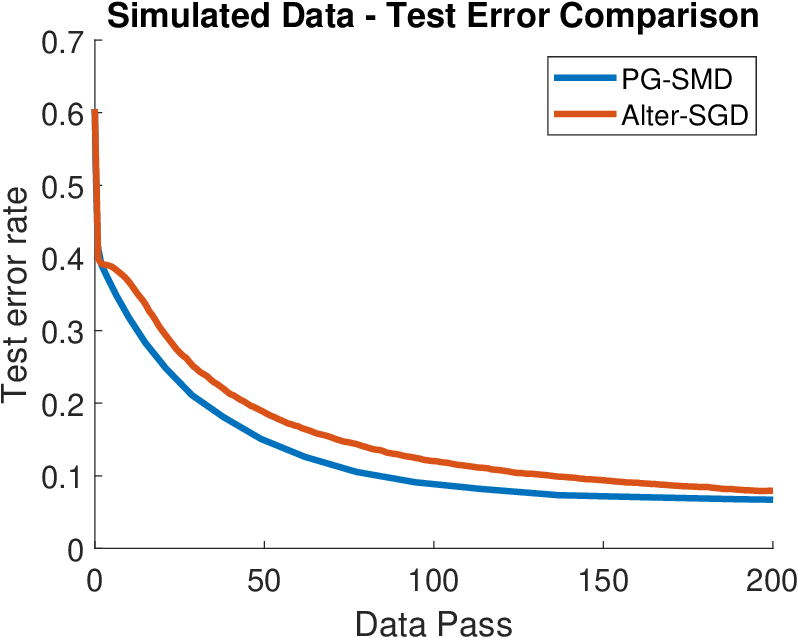}
	\includegraphics[scale=0.33]{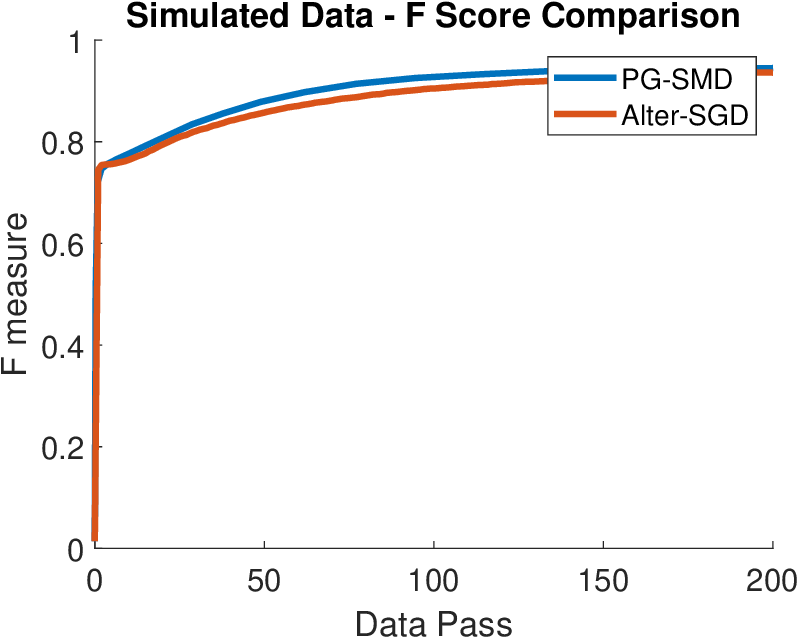}
	\vspace{-0.1in}
	\caption{Comparisons of PG-SMD and Alter-SGD methods on robust learning from multiple distributions with non-smooth loss. }
	\label{fig:1expect}
\end{figure*}


\section{Conclusion}
This paper contributes to the numerical solvers for non-convex min-max problems. We focus on a class of weakly-convex-concave min-max problems where the minimization part is weakly convex, and the maximization part is concave. This class of problems covers many important applications, such as robust learning. We consider two different scenarios: (i) the min-max objective function involves expectation, and (ii) the objective function is smooth and has a finite-sum structure. Stochastic mirror descent method and stochastic variance-reduced gradient method are developed to solve the subproblem in scenario (i) and (ii), respectively. We analyze the computational complexity of our methods for finding a nearly $\epsilon$-stationary for the original problem under both scenarios. Numerical experiments on distributionally robust learning with non-convex losses demonstrate the effectiveness of the proposed methods. 


\bibliographystyle{tfs}
\bibliography{references}


\appendix

\section{Proof of Theorem~\ref{thm:totalcomplexitySD}}
\label{sec:prooftheorem1}
We first present the convergence result of SMD for solving (\ref{subminmaxproblem0}) proved by \cite{nemirovski2009robust}. The proof is well-known and included only for the sake of completeness.
\begin{proposition}[Convergence of Algorithm~\ref{alg:SMD}]
	\label{dualgapresult}
	Suppose Assumptions~\ref{assume:stochastic} and~\ref{assume:stochastic1} hold. 
	Algorithm $\ref{alg:SMD}$ guarantees 
	\small
	\begin{align}
	\nonumber
	&~~~\mathbb{E}\bigg[\frac{1-\gamma\rho}{2\gamma}\|\bar\bx-\bx_\dagger\|_2^2\bigg]\leq \mathbb{E}\left[\max_{\by\in\mathbb{R}^q}\phi_{\gamma} (\widehat\bx,\by;\bar\bx)-\max_{\by\in\mathbb{R}^q}\phi_{\gamma} (\bx_\dagger,\by;\bar\bx)\right]	\\\label{eq:mainineq1}
	\leq&~~~\frac{5\eta_xM_x^2}{2}+\frac{5\eta_yM_y^2}{2}+\frac{1}{J}\left[\left(\frac{1}{\eta_x}+\frac{\rho}{2}\right)\|\bx_\dagger-\bar\bx\|_2^2+\frac{2}{\eta_y}\mathbb{E}V_y(\widehat\by_*(\widehat\bx),\bar\by)+Q_g+Q_r\right],
	\end{align}
	\normalsize
	where  $\bx_\dagger=\text{prox}_{\gamma\psi}(\bar\bx)$ and $\widehat\by_*(\widehat\bx)\in\argmin_{\by\in\mathbb{R}^q}\phi_{\gamma}(\widehat\bx,\by;\bar\bx)$.
\end{proposition}
\begin{proof}
	For simplicity of notation, we write $\phi_{\gamma}(\bx,\by;\bar\bx)$ as  $\phi(\bx,\by)$ in this proof and define
	$$
	\bar g(\bx):=\frac{1}{2\gamma}\|\bx-\bar\bx\|_2^2+g(\bx).
	$$
	Next, we will perform the standard analysis of SMD (e.g. the proof of Proposition 3.2 in \cite{nemirovski2009robust}).
	According to the updating equations of $(\bx^{(j+1)},\by^{(j+1)})$, we have 
	\small
	\begin{eqnarray*}
		&&(\bx^{(j+1)})^\top\bg_x^{(j)} +\frac{1}{2\eta_x}\|\bx^{(j+1)}-\bx^{(j)}\|_2^2+\bar g(\bx^{(j+1)})+\left(\frac{1}{2\eta_x}+\frac{1}{2\gamma}\right)\|\bx-\bx^{(j+1)}\|_2^2\\
		&\leq&
		\bx^\top\bg_x^{(j)} +\frac{1}{2\eta_x}\|\bx-\bx^{(j)}\|_2^2+\bar g(\bx)\\
		&&-(\by^{(j+1)})^\top\bg_y^{(j)} +\frac{1}{\eta_y}V_y(\by^{(j+1)},\by^{(j)})+r(\by^{(j+1)})+\frac{1}{\eta_y}V_y(\by,\by^{(j+1)})\\
		&\leq&
		-\by^\top\bg_y^{(j)} +\frac{1}{\eta_y}V_y(\by,\by^{(j)})+r(\by).
	\end{eqnarray*}
	\normalsize
	Summing these two inequalities and organizing terms imply
	\small
	\begin{eqnarray}
	\nonumber
	&&(\bx^{(j)}-\bx)^\top\mathbb{E}\bg_x^{(j)} -(\by^{(j)}-\by)^\top\mathbb{E}\bg_y^{(j)}
	+(\bx^{(j)}-\bx)^\top\bdelta_x^{(j)}-(\by^{(j)}-\by)^\top\bdelta_y^{(j)}\\\nonumber
	&& +\frac{1}{2\eta_x}\|\bx^{(j+1)}-\bx^{(j)}\|_2^2+\frac{1}{\eta_y}V_y(\by^{(j+1)},\by^{(j)})+(\bx^{(j+1)}-\bx^{(j)})^\top\bg_x^{(j)}-(\by^{(j+1)}-\by^{(j)})^\top\bg_y^{(j)}\\\nonumber
	&\leq&
	\frac{1}{2\eta_x}\|\bx-\bx^{(j)}\|_2^2-\left(\frac{1}{2\eta_x}+\frac{1}{2\gamma}\right)\|\bx-\bx^{(j+1)}\|_2^2+\bar g(\bx)-\bar g(\bx^{(j+1)})\\\label{eq:sgdproof1}
	&&+\frac{1}{\eta_y}V_y(\by,\by^{(j)})-\frac{1}{\eta_y}V_y(\by,\by^{(j+1)})+r(\by)-r(\by^{(j+1)}),
	\end{eqnarray}	
	\normalsize
	where  
	$$
	\bdelta_x^{(j)}:=\bg_x^{(j)}-\mathbb{E}\bg_x^{(j)}\quad\text{ and }\quad
	\bdelta_y^{(j)}:=\bg_y^{(j)}-\mathbb{E}\bg_y^{(j)}.
	$$
	By Young's inequality, we can show that
	\small
	\begin{eqnarray*}
		&&-(\bx^{(j+1)}-\bx^{(j)})^\top\bg_x^{(j)}+(\by^{(j+1)}-\by^{(j)})^\top\bg_y^{(j)}\\
		&\leq&\frac{1}{2\eta_x}\|\bx^{(j+1)}-\bx^{(j)}\|_2^2+\frac{\eta_x\|\bg_x^{(j)}\|_2^2}{2}
		+\frac{1}{2\eta_y}\|\by^{(j+1)}-\by^{(j)}\|^2+\frac{\eta_y\|\bg_y^{(j)}\|_*^2}{2}\\
		&\leq&\frac{1}{2\eta_x}\|\bx^{(j+1)}-\bx^{(j)}\|_2^2+\frac{\eta_xM_x^2}{2}
		+\frac{1}{\eta_y}V_y(\by^{(j+1)},\by^{(j)})+\frac{\eta_yM_y^2}{2},
	\end{eqnarray*}	
	\normalsize
	where, in the last inequality, we use Assumption~\ref{assume:stochastic1}B and the strong convexity of $d_y(\by)$.
	Adding both sides of this inequality to those of \eqref{eq:sgdproof1} yields
	\small
	\begin{eqnarray}
	\nonumber
	&&(\bx^{(j)}-\bx)^\top\mathbb{E}\bg_x^{(j)} -(\by^{(j)}-\by)^\top\mathbb{E}\bg_y^{(j)}+(\bx^{(j)}-\bx)^\top\bdelta_x^{(j)}-(\by^{(j)}-\by)^\top\bdelta_y^{(j)}	\\\nonumber
	&\leq& \frac{\eta_xM_x^2}{2}+\frac{1}{2\eta_x}\|\bx-\bx^{(j)}\|_2^2-\left(\frac{1}{2\eta_x}+\frac{1}{2\gamma}\right)\|\bx-\bx^{(j+1)}\|_2^2-\bar g(\bx)+\bar g(\bx^{(j+1)})\\\nonumber
	&&+\frac{\eta_yM_y^2}{2}+\frac{1}{\eta_y}V_y(\by,\by^{(j)})-\frac{1}{\eta_y}V_y(\by,\by^{(j+1)})- r(\by)+ r(\by^{(j+1)}).
	\end{eqnarray}
	\normalsize
	Recall that $\mathbb{E}\bg_x^{(j)}\in\partial_x[f(\bx^{(j)},\by^{(j)})]$ and $-\mathbb{E}\bg_y^{(j)}\in\partial_y[-f(\bx^{(j)},\by^{(j)})]$ according to Assumption~\ref{assume:stochastic1}A. The inequality above, the $\rho$-weak convexity of $f$, and the concavity of $f$ in $\by$ together imply
	\small
	\begin{eqnarray}
	\nonumber
	&&\phi (\bx^{(j)},\by)-\phi (\bx,\by^{(j)})+(\bx^{(j)}-\bx)^\top\bdelta_x^{(j)}-(\by^{(j)}-\by)^\top\bdelta_y^{(j)}	\\\nonumber
	&\leq& \frac{\eta_xM_x^2}{2}+\left(\frac{1}{2\eta_x}+\frac{\rho}{2}\right)\|\bx-\bx^{(j)}\|_2^2-\left(\frac{1}{2\eta_x}+\frac{1}{2\gamma}\right)\|\bx-\bx^{(j+1)}\|_2^2-\bar g(\bx^{(j+1)})+\bar g(\bx^{(j)})\\\label{eq:ineq1}
	&&+\frac{\eta_yM_y^2}{2}+\frac{1}{\eta_y}V_y(\by,\by^{(j)})-\frac{1}{\eta_y}V_y(\by,\by^{(j+1)})+ r(\by^{(j+1)})- r(\by^{(j)}),
	\end{eqnarray}
	\normalsize
	for any $\bx\in\mathcal{X}$ and any $\by\in\Y$.
	
	Let $(\tilde \bx^{(0)},\tilde \by^{(0)})=( \bx^{(0)}, \by^{(0)})$ and let 
	\begin{eqnarray*}
		\tilde\bx^{(j+1)}&=&\argmin_{\bx\in\mathbb{R}^p}-\big\langle \bx,\bdelta_x^{(j)}\big\rangle +\frac{1}{2\eta_x}\|\bx-\tilde\bx^{(j)}\|_2^2\\
		\tilde\by^{(j+1)}&=&\argmin_{\by\in\mathbb{R}^q}\big\langle \by,\bdelta_y^{(j)}\big\rangle +\frac{1}{\eta_y}V_y(\by,\tilde\by^{(j)})
	\end{eqnarray*}
	for $j=0,1,\dots,J-2$. Using a similar analysis as above, we can obtain the following inequality similar to \eqref{eq:ineq1}
	\small
	\begin{eqnarray}
	\label{eq:ineq2}
	&&-(\tilde\bx^{(j)}-\bx)^\top\bdelta_x^{(j)}+(\tilde\by^{(j)}-\by)^\top\bdelta_y^{(j)}	\\\nonumber
	&\leq& \frac{4\eta_xM_x^2}{2}+\frac{1}{2\eta_x}\|\bx-\tilde\bx^{(j)}\|_2^2-\frac{1}{2\eta_x}\|\bx-\tilde\bx^{(j+1)}\|_2^2	+\frac{4\eta_yM_y^2}{2}+\frac{1}{\eta_y}V_y(\by,\tilde\by^{(j)})-\frac{1}{\eta_y}V_y(\by,\tilde\by^{(j+1)}).
	\end{eqnarray}
	\normalsize
	
	The following inequality is then obtained by adding \eqref{eq:ineq1} and \eqref{eq:ineq2} and reorganizing terms:
	\small
	\begin{eqnarray}
	\nonumber
	&&\phi (\bx^{(j)},\by)-\phi (\bx,\by^{(j)})+(\bx^{(j)}-\tilde\bx^{(j)})^\top\bdelta_x^{(j)}-(\by^{(j)}-\tilde\by^{(j)})^\top\bdelta_y^{(j)}	\\\nonumber
	&\leq& \frac{5\eta_xM_x^2}{2}+\left(\frac{1}{2\eta_x}+\frac{\rho}{2}\right)\|\bx-\bx^{(j)}\|_2^2-\left(\frac{1}{2\eta_x}+\frac{1}{2\gamma}\right)\|\bx-\bx^{(j+1)}\|_2^2\\\nonumber
	&&+ \frac{5\eta_yM_y^2}{2}+\frac{1}{\eta_y}V_y(\by,\by^{(j)})-\frac{1}{\eta_y}V_y(\by,\by^{(j+1)})
	-\bar g(\bx^{(j+1)})+\bar g(\bx^{(j)})- r(\by^{(j+1)})+ r(\by^{(j)})\\\nonumber
	&&+\frac{1}{2\eta_x}\|\bx-\tilde\bx^{(j)}\|_2^2-\frac{1}{2\eta_x}\|\bx-\tilde\bx^{(j+1)}\|_2^2
	+\frac{1}{\eta_y}V_y(\by,\tilde\by^{(j)})-\frac{1}{\eta_y}V_y(\by,\tilde\by^{(j+1)}).
	\end{eqnarray}
	\normalsize
	Summing the inequality above for $j=0,1,\dots,J-1$ and using the fact that $\rho<\frac{1}{\gamma}$, we obtain the following inequality after reorganizing terms: 
	\small
	\begin{align}
	\nonumber
	&~~~~\sum_{j=0}^{J-1}\left(\phi (\bx^{(j)},\by)-\phi (\bx,\by^{(j)})\right)
	\\ \nonumber 
	&\leq~\sum_{j=0}^{J-1}\left[-(\bx^{(j)}-\tilde\bx^{(j)})^\top\bdelta_x^{(j)}+(\by^{(j)}-\tilde\by^{(j)})^\top\bdelta_y^{(j)}\right]+\frac{5J\eta_xM_x^2}{2}+\frac{5J\eta_yM_y^2}{2}
	\\\nonumber 
	&~~~~+\left(\frac{1}{\eta_x}+\frac{\rho}{2}\right)\|\bx-\bx^{(0)}\|_2^2+\frac{2}{\eta_y}V_y(\by,\by^{(0)})-g(\bx^{(J)})+g(\bx^{(0)})- r(\by^{(J)})+ r(\by^{(0)})\\\nonumber
	&~~~~+\frac{1}{2\gamma}\|\bx^{(0)}-\bar\bx\|_2^2-\frac{1}{2\gamma}\|\bx^{(J)}-\bar\bx\|_2^2.
	\end{align}
	\normalsize
	Since $\bx^{(0)}=\bar \bx$ and Assumption~\ref{assume:stochastic1}C holds, the inequality above implies
	\small
	\begin{align}
	\nonumber
	&~~~~\sum_{j=0}^{J-1}\left(\phi (\bx^{(j)},\by)-\phi (\bx,\by^{(j)})\right)\\\nonumber
	&\leq~\sum_{j=0}^{J-1}\left[-(\bx^{(j)}-\tilde\bx^{(j)})^\top\bdelta_x^{(j)}+(\by^{(j)}-\tilde\by^{(j)})^\top\bdelta_y^{(j)}\right]+\frac{5J\eta_xM_x^2}{2}+\frac{5J\eta_yM_y^2}{2}\\\nonumber
	&~~~~+\left(\frac{1}{\eta_x}+\frac{\rho}{2}\right)\|\bx-\bx^{(0)}\|_2^2+\frac{2}{\eta_y}V_y(\by,\by^{(0)})+Q_g+Q_r.
	\end{align}
	\normalsize
	Similar to $\widehat\bx$, we define $\widehat\by = \frac{1}{J} \sum^{J- 1}_{j=0} \by^{(j)} $.
	Dividing the inequality above by $J$ and applying the convexity $\phi $ in $\bx$ and the concavity of $\phi $ in $\by$, we obtain for any $\bx\in\X$ and any $\by\in\Y$
	\small
	\begin{eqnarray}
	\nonumber
	\phi (\widehat\bx,\by)-\max_{\by'\in\mathbb{R}^q}\phi (\bx,\by')&\leq&\phi (\widehat\bx,\by)-\phi (\bx,\widehat\by)	\\\nonumber
	&\leq&\frac{1}{J}\sum_{j=0}^{J-1}\left[-(\bx^{(j)}-\tilde\bx^{(j)})^\top\bdelta_x^{(j)}+(\by^{(j)}-\tilde\by^{(j)})^\top\bdelta_y^{(j)}\right]+\frac{5\eta_xM_x^2}{2}+\\\nonumber
	&&\frac{5\eta_yM_y^2}{2}+\frac{1}{J}\left[\left(\frac{1}{\eta_x}+\frac{\rho}{2}\right)\|\bx-\bx^{(0)}\|_2^2+\frac{2}{\eta_y}V_y(\by,\by^{(0)})+Q_g+Q_r\right].
	\end{eqnarray}
	\normalsize
	We then choose $\bx=\bx_\dagger$ in the inequality above and maximize the left-hand side of the inequality above over $\by\in\mathbb{R}^q$. Since the maximum solution is $\widehat\by_*(\widehat\bx)$, we just replace $\by$ on the right-hand side by $\widehat\by_*(\widehat\bx)$ and obtained
	\small
	\begin{eqnarray}
	\nonumber
	&&\max_{\by\in\mathbb{R}^q}\phi (\widehat\bx,\by)-\max_{\by\in\mathbb{R}^q}\phi (\bx_\dagger,\by)\\\nonumber
	&\leq&\frac{1}{J}\sum_{j=0}^{J-1}\left[-(\bx^{(j)}-\tilde\bx^{(j)})^\top\bdelta_x^{(j)}+(\by^{(j)}-\tilde\by^{(j)})^\top\bdelta_y^{(j)}\right]+\frac{5\eta_xM_x^2}{2}+\frac{5\eta_yM_y^2}{2}\\\nonumber
	&&+\frac{1}{J}\left[\left(\frac{1}{\eta_x}+\frac{\rho}{2}\right)\|\bx_\dagger-\bx^{(0)}\|_2^2+\frac{2}{\eta_y}V_y(\widehat\by_*(\widehat\bx)\,\by^{(0)})+Q_g+Q_r\right].
	\end{eqnarray}
	\normalsize
	Since $(\bx^{(0)},\by^{(0)})=(\bar\bx,\bar\by)$, $\mathbb{E}\bdelta_x^{(j)}=0$ and $\mathbb{E}\bdelta_y^{(j)}=0$ conditioning on all the stochastic events up to iteration $j$ for $j=0,1,\dots,J-1$, we obtain the second inequality in \eqref{eq:mainineq1}  by  taking expectation on both sides of the inequality above. 
	The first inequality in \eqref{eq:mainineq1} is because of the  $(\frac{1}{\gamma}-\rho)$-strong convexity of $\phi$ in $\bx$.

\end{proof}

\textbf{Proof of Theorem~\ref{thm:totalcomplexitySD}}
\begin{proof}
	Note that $\max_{\by\in\mathbb{R}^q}\phi_{\gamma} (\bx,\by;\bar\bx)=\psi (\bx) + \frac{1}{2 \gamma} \| \bx- \bar \bx\|^2_2$.
	By the definition of $\bx_\dagger^{(t)}$ in \eqref{eq:subproblem}, we have
	\small
	\begin{eqnarray}
	\label{eq:fact1}
	\psi (\bx_\dagger^{(t)}) + \frac{1}{2 \gamma} \| \bx_\dagger^{(t)} - \bar \bx^{(t)}\|^2_2 \leq \psi (\bar\bx^{(t)}).
	\end{eqnarray}
	\normalsize
	Borrowing the inequalities derived by \cite{davis2017proximally} in the proof of their Theorem 2, we have  
	\small
	\begin{align}\nonumber
	\| \bx_\dagger^{(t)} - \bar \bx^{(t)}\|^2_2 - \|\bar \bx^{(t+1)} - \bar \bx^{(t)} \|^2_2
	&= (\| \bx_\dagger^{(t)} - \bar \bx^{(t)}\| + \|\bar \bx^{(t+1)} - \bar \bx^{(t)} \|)(\| \bx_\dagger^{(t)} - \bar \bx^{(t)}\| - \|\bar \bx^{(t+1)} - \bar \bx^{(t)} \|) \\\nonumber
	&\leq ( 2\| \bx_\dagger^{(t)} - \bar \bx^{(t)}\| + \| \bx_\dagger^{(t)} - \bar \bx^{(t+1)}\| )\| \bx_\dagger^{(t)} - \bar \bx^{(t+1)}\|\\\nonumber
	&= 2\| \bx_\dagger^{(t)} - \bar \bx^{(t)}\|  \| \bx_\dagger^{(t)} - \bar \bx^{(t+1)}\| + \| \bx_\dagger^{(t)} - \bar \bx^{(t+1)}\|^2\\\label{eq:fact2}
	&\leq \frac{1}{3} \| \bx_\dagger^{(t)} - \bar \bx^{(t)}\|^2_2 + 4 \| \bx_\dagger^{(t)} - \bar \bx^{(t+1)}\|^2_2.
	\end{align}
	\normalsize
	
	In iteration $t$ of  Algorithm~\ref{alg:PGM}, Algorithm~\ref{alg:SMD} is called as $$(\bar\bx^{(t+1)},\bar\by^{(t+1)})=\text{SMD}(\bar\bx^{(t)},\bar\by^{(t)},\eta_x^t,\eta_y^t,\gamma, j_t),$$ 
	where the choices of $\bar\by^{(t)}$ and $(\eta_x^t,\eta_y^t, j_t)$ depend on whether case D1 or case D2 holds in Assumption~\ref{assume:stochastic1}. We then prove the convergence property of Algorithm~\ref{alg:PGM} by applying Proposition~\ref{dualgapresult} under case D1 and case D2, separately.\\
	
	\emph{Case D1:} According to Algorithm~\ref{alg:PGM},  Algorithm~\ref{alg:SMD} is called with  
	\small
	$$
	\bar \by^{(t)}\in\argmin_{\by\in\mathbb{R}^q}d_y(\by)\text{ and }(\eta_x^t,\eta_y^t,j_t)=\left(\frac{D_x}{M_x\sqrt{j_t}},\frac{D_y}{M_y\sqrt{j_t}},(t+3)^2\right).
	$$
	\normalsize
	Let  $\bar \bx=\bar\bx^{(t)}$, $\bar \by=\bar \by^{(t)}$, and $(\eta_x,\eta_y,J)=(\eta_x^t,\eta_y^t,j_t)$ in Proposition~\ref{dualgapresult}. By Assumption~\ref{assume:stochastic1}, we have $\|\bx_\dagger-\bar\bx\|_2^2\leq D_x^2$ and $V_y(\widehat\by_*(\widehat\bx),\bar\by)\leq \max_{\by\in\Y}d_y(\by)-\min_{\by\in\Y}d_y(\by)\leq D_y^2/2$ in Proposition~\ref{dualgapresult}. Applying these two inequalities and the definitions of $(\eta_x^t,\eta_y^t,j_t)$ above to~\ref{eq:mainineq1}, we have	
	\small
	\begin{align}
		\nonumber
		&~\mathbb{E}\left[\psi(\bar \bx^{(t+1)})+\frac{1}{2 \gamma} \|\bar \bx^{(t+1)} - \bar \bx^{(t)}\|^2_2-\psi( \bx_\dagger^{(t)}) -\frac{1}{2 \gamma} \| \bx_\dagger^{(t)} - \bar \bx^{(t)}\|^2_2\right]\\\nonumber
		\leq&~\frac{5\eta_xM_x^2}{2}+\frac{5\eta_yM_y^2}{2}+\frac{1}{J}\left[\frac{D_x^2}{\eta_x}+\frac{\rho}{2}\mathbb{E}\|\bx_\dagger^{(t)}-\bar\bx^{(t)}\|_2^2+\frac{D_y^2}{\eta_y}+Q_g+Q_r\right]	\\\label{eq:SMDconvergemainineq1}
		\leq&~\frac{1}{4\gamma (t+3)^2}\mathbb{E}\|\bx_\dagger^{(t)}-\bar\bx^{(t)}\|_2^2+ E_t\leq\frac{1}{36\gamma}\mathbb{E}\|\bx_\dagger^{(t)}-\bar\bx^{(t)}\|_2^2+ E_t,
	\end{align}
	\normalsize
	where the second inequality is because $\rho\leq \frac{1}{2\gamma}$ and 
	\small
	$$
	E_t:=\frac{7M_xD_x+7M_yD_y}{2\sqrt{j_t}}+\frac{Q_g+Q_r}{j_t}.
	$$
	\normalsize
	The first inequality in \eqref{eq:mainineq1} implies
	\small
	\begin{eqnarray}
		\label{eq:SMDconvergemainineq2}
		\mathbb{E}\bigg[\frac{1-\gamma\rho}{2\gamma}\|\bar\bx^{(t+1)}-\bx_\dagger^{(t)}\|_2^2\bigg]
		\leq \frac{1}{36\gamma}\mathbb{E}\|\bx_\dagger^{(t)}-\bar\bx^{(t)}\|_2^2+E_t.
	\end{eqnarray}
	\normalsize
	Using  \eqref{eq:SMDconvergemainineq1} and \eqref{eq:SMDconvergemainineq2} together with \eqref{eq:fact2}, we can show that
	\small
	\begin{align*}
\mathbb{E}\psi(\bar \bx^{(t+1)}) 
	&\leq \mathbb{E}\psi( \bx_\dagger^{(t)}) + \frac{1}{2 \gamma} \mathbb{E}\| \bx_\dagger^{(t)} - \bar \bx^{(t)}\|^2_2 - \frac{1}{2 \gamma} \mathbb{E}\|\bar \bx^{(t+1)} - \bar \bx^{(t)}\|^2_2+ \frac{1}{36\gamma}\mathbb{E}\|\bx_\dagger^{(t)}-\bar\bx^{(t)}\|_2^2+E_t\\
	&\leq \mathbb{E}\psi( \bx_\dagger^{(t)}) + \frac{1}{2 \gamma} \bigg( \frac{1}{3} \mathbb{E}\|\bx_\dagger^{(t)} - \bar \bx^{(t)}\|^2_2 + 4 \mathbb{E}\|\bx_\dagger^{(t)} - \bar \bx^{(t+1)}\|^2_2    \bigg) + \frac{1}{36\gamma}\mathbb{E}\|\bx_\dagger^{(t)}-\bar\bx^{(t)}\|_2^2+ E_t\\
	&\leq \mathbb{E}\psi( \bx_\dagger^{(t)}) + \frac{1}{6 \gamma} \mathbb{E}\|\bx_\dagger^{(t)} - \bar \bx^{(t)}\|^2_2 + \left(\frac{4}{1-\gamma \rho} +1\right)\left(\frac{1}{36\gamma}\mathbb{E}\|\bx_\dagger^{(t)}-\bar\bx^{(t)}\|_2^2+ E_t\right)\\
	&\leq \mathbb{E}\psi( \bx_\dagger^{(t)}) + \frac{5}{12 \gamma} \mathbb{E}\|\bx_\dagger^{(t)} - \bar \bx^{(t)}\|^2_2 + 9E_t,
	\end{align*}
	\normalsize
	where the first three inequalities are because of \eqref{eq:SMDconvergemainineq1}, \eqref{eq:fact2}, and \eqref{eq:SMDconvergemainineq2}, respectively, and the last inequality is because $1-\gamma\rho\geq \frac{1}{2}$ as $\gamma\in (0,1/(2\rho)]$.
	Combining this inequality with \eqref{eq:fact1} leads to 
	\small
	\begin{align*}
	\mathbb{E}\psi(\bar \bx^{(t+1)}) &\leq \mathbb{E}\psi(\bar \bx^{(t)}) - \frac{1}{12\gamma} \mathbb{E}\|\bx_\dagger^{(t)} - \bar \bx^{(t)}\|^2_2+ 9E_t.
	\end{align*}
	\normalsize
	Adding the inequalities for $t=0,...,T-1$ yields
	\small
	\begin{align*}
	\mathbb{E}\psi(\bar \bx^{(T)}) \leq &\mathbb{E}\psi(  \bar\bx^{(0)}) - \sum^{T-1}_{t=0}  \frac{1}{12 \gamma}\mathbb{E} \|\bx_\dagger^{(t)} - \bar \bx^{(t)}\|^2_2+9\sum^{T-1}_{t=0} E_t.
	\end{align*}
	\normalsize
	Rearranging the inequality above gives
	\small
	\begin{align*}
	\sum^{T-1}_{t=0}  \mathbb{E} \|\bx_\dagger^{(t)} - \bar \bx^{(t)}\|^2_2 \leq  3 \gamma\bigg( \mathbb{E} \psi(\bar \bx^{(0)})-\psi^*+ 9\sum^{T-1}_{t=0} E_t\bigg),
	\end{align*}
	\normalsize
	where $\psi^* = \min_{\bx \in \mathbb{R}^p} \psi(\bx)>-\infty$. Let $\tau$ be a uniform random variable supported on $\{0,1,...,T-1\}$. Then, we have 
	\small
	\begin{align}
	\label{eqn:boundonxx_SMD}
	\mathbb{E}\|\bx_\dagger^{(\tau)} - \bar \bx^{(\tau)}\|^2_2 \leq  \frac{3 \gamma}{T}\bigg(  \psi(\bar \bx^{(0)})-\psi^*+ 9\sum^{T-1}_{t=0} E_t\bigg).
	\end{align}
	\normalsize
	
	Since $j_t=(t+3)^2$ in Algorithm~\ref{alg:PGM} in Case D1, we have
	\small
	\begin{align}
	\nonumber
	&\sum^{T-1}_{t=0} E_t
	=\sum^{T-1}_{t=0}\bigg[ \frac{7M_xD_x+7M_yD_y}{2(t+3)}+\frac{Q_g+Q_r}{(t+3)^2}\bigg]\\\label{eq:upperboundE}
	\leq&  \ln(T+1)\left(3.5M_xD_x+3.5M_yD_y\right)+2Q_g+2Q_r.
	\end{align}
	\normalsize
	Recall that $\mathbb{E}[\text{Dist} (\mathbf{0}, \partial \psi(\bx_\dagger^{(\tau)}) )^2] \leq  \frac{1}{\gamma^2}\mathbb{E}\|\bx_\dagger^{(\tau)} - \bar \bx^{(\tau)}\|^2_2$ by the definition of $\bx_\dagger^{(\tau)}$.  According to \eqref{eqn:boundonxx_SMD} and \eqref{eq:upperboundE}, we can ensure $\mathbb{E}[\text{Dist} (\mathbf{0}, \partial \psi(\bx_\dagger^{(\tau)}) )^2]\leq \epsilon^2$ by setting $T$ large enough so that 
	\small
	$$
	\frac{3 }{T\gamma}\bigg(  \psi(\bar \bx^{(0)})-\psi^*+ 18Q_g+18Q_r\bigg)\leq \frac{\epsilon^2}{2}\text{  and  }
	\frac{94.5 \ln(T+1)}{T\gamma}\left(M_xD_x+M_yD_y\right)\leq \frac{\epsilon^2}{2}.
	$$
	\normalsize
	Solving the two inequalities for $T$, we find that such $T$ can be chosen as
	\small
	$$ 
	T=\left\lceil\max\left\{
	\begin{array}{c}
	\frac{6[\psi(\bar \bx^{(0)})-\psi^*+18Q_g+18Q_r]}{\gamma\epsilon^2},\\
	\frac{378(M_xD_x+M_yD_y)}{\gamma\epsilon^2}\ln\left(\frac{378(M_xD_x+M_yD_y)}{\gamma\epsilon^2}\right)
	\end{array}
	\right\}\right\rceil=\tilde O\left(\frac{1}{\gamma\epsilon^2}\right).
	$$
	\normalsize
	
	In Case D1, only one stochastic subgradient of $f$ is computed in each iteration of Algorithm~\ref{alg:SMD}. Hence, the total number of subgradients computed  
	to generate $\bar \bx^{(\tau)}$ is bounded by
	\small
	$$
	\sum^{T-1}_{t=0} j_t =  {O}(T^3) =  \tilde {O}\left(\frac{1}{\gamma^3\epsilon^6}\right).
	$$
	\normalsize

	\emph{Case D2:}
	In this case, we have $f(\bx,\by)=\frac{1}{n}\sum_{i=1}^n\by^\top\mathbf{c}_i(\bx)=\by^\top\mathbf{c}(\bx)$,   $\bar\by^{(t)}=\argmax_{\by\in\mathbb{R}^q}\by^\top \mathbf{c}(\bar\bx^{(t)})-r(\by)$ and $\bar\by^{(t+1)}=\argmax_{\by\in\mathbb{R}^q}\by^\top \mathbf{c}(\bar\bx^{(t+1)})-r(\by)$ in Algorithm $\ref{alg:PGM}$. 
	By the $\mu$-strong convexity of $r$ with respect to $V_y$, we have the following two inequalities:
	\small
	\begin{align*}
	\mu V_y(\bar\by^{(t+1)},\bar\by^{(t)})&\leq -(\bar\by^{(t+1)})^\top \mathbf{c}(\bar\bx^{(t)})+r(\bar\by^{(t+1)})
	+(\bar\by^{(t)})^\top \mathbf{c}(\bar\bx^{(t)})-r(\bar\by^{(t)})\\
	\mu V_y(\bar\by^{(t)},\bar\by^{(t+1)})&\leq-(\bar\by^{(t)})^\top \mathbf{c}(\bar\bx^{(t+1)})+r(\bar\by^{(t)})+(\bar\by^{(t+1)})^\top \mathbf{c}(\bar\bx^{(t+1)})-r(\bar\by^{(t+1)}).
	\end{align*}
	\normalsize
	Adding these two inequalities and organizing terms give
	\small
	\begin{align*}
	\mu V_y(\bar\by^{(t+1)},\bar\by^{(t)})+\mu V_y(\bar\by^{(t)},\bar\by^{(t+1)})
	\leq& ~~\left[\bar\by^{(t)}-\bar\by^{(t+1)}\right]^\top\left[ \mathbf{c}(\bar\bx^{(t)})-\mathbf{c}(\bar\bx^{(t+1)})\right]\\
	\leq& ~~\left\|\bar\by^{(t)}-\bar\by^{(t+1)}\right\|\left\| \mathbf{c}(\bar\bx^{(t)})-\mathbf{c}(\bar\bx^{(t+1)})\right\|_*\\
	\leq& ~~\sqrt{V_y(\bar\by^{(t+1)},\bar\by^{(t)})+V_y(\bar\by^{(t)},\bar\by^{(t+1)})}M_c\left\|\bar\bx^{(t)}-\bar\bx^{(t+1)}\right\|_2,
	\end{align*}
	\normalsize
	where the last inequality is due to the $1$-strong convexity of the distance generating function $d_y$ with respect to the norm $\|\cdot\|$. This inequality simply implies 
	\small
	\begin{eqnarray}
	\label{eq:sccaseineq1}
	\mu\sqrt{V_y(\bar\by^{(t+1)},\bar\by^{(t)})}
	\leq
	\mu\sqrt{V_y(\bar\by^{(t+1)},\bar\by^{(t)})+V_y(\bar\by^{(t)},\bar\by^{(t+1)})}\leq M_c\left\|\bar\bx^{(t)}-\bar\bx^{(t+1)}\right\|_2.
	\end{eqnarray}	
	\normalsize
	
	According to Algorithm~\ref{alg:PGM},  Algorithm~\ref{alg:SMD} is called with 
	\small 
	$$
	\bar\by^{(t)}=\argmax_{\by\in\mathbb{R}^q}\by^\top \mathbf{c}(\bar\bx^{(t)})-r(\by)\text{ and }\left(\frac{60\gamma}{(j_t-30)},\frac{8M_c^2\gamma}{\mu^2j_t},t+32\right).
	$$
	\normalsize
	Setting  $\bar \bx=\bar\bx^{(t)}$, $\bar \by=\bar \by^{(t)}$, and $(\eta_x,\eta_y,J)=(\eta_x^t,\eta_y^t,j_t)$ in Proposition~\ref{dualgapresult} gives
	\small
	\begin{eqnarray}
	\nonumber
	&&\mathbb{E}\left[\psi(\bar \bx^{(t+1)})+\frac{1}{2 \gamma} \|\bar \bx^{(t+1)} - \bar \bx^{(t)}\|^2_2-\psi( \bx_\dagger^{(t)}) -\frac{1}{2 \gamma} \| \bx_\dagger^{(t)} - \bar \bx^{(t)}\|^2_2\right]	\\\nonumber
	&\leq&\frac{5\eta_x^tM_x^2}{2}+\frac{5\eta_y^tM_y^2}{2}+\frac{1}{j_t}\left[\left(\frac{1}{\eta_x^t}+\frac{\rho}{2}\right)\mathbb{E}\|\bx_\dagger^{(t)}- \bar \bx^{(t)}\|_2^2+\frac{2}{\eta_y^t}\mathbb{E}V_y(\bar\by^{(t+1)},\bar\by^{(t)})+Q_g+Q_r\right]\\\nonumber
	&\leq&\frac{5\eta_x^tM_x^2}{2}+\frac{5\eta_y^tM_y^2}{2}+\frac{1}{j_t}\left[\left(\frac{1}{\eta_x^t}+\frac{\rho}{2}\right)\mathbb{E}\|\bx_\dagger^{(t)}- \bar \bx^{(t)}\|_2^2+\frac{2M_c^2}{\mu^2\eta_y^t}\mathbb{E}\|\bar\bx^{(t)}-\bar\bx^{(t+1)}\|_2^2+Q_g+Q_r\right]\\\label{eq:complicatedineq1}
	&\leq&\frac{1}{60\gamma}\mathbb{E}\|\bx_\dagger^{(t)}- \bar \bx^{(t)}\|_2^2+\frac{1}{4\gamma}\mathbb{E}\|\bar\bx^{(t)}-\bar\bx^{(t+1)}\|^2+\frac{150M_x^2\gamma}{(j_t-30)}+\frac{20M_c^2M_y^2\gamma}{\mu^2j_t}+\frac{Q_g+Q_r}{j_t},
	\end{eqnarray}
	\normalsize
	where the second inequality is due to \eqref{eq:sccaseineq1} and the third inequality is because of the choices of $\eta_x^t$, $\eta_y^t$ and $j_t$ as well as the fact that $\rho<\frac{1}{\gamma}$. Let 
	\small
	$$
	E_t':=\frac{150M_x^2\gamma}{j_t-30}+\frac{20M_c^2M_y^2\gamma}{\mu^2j_t}+\frac{Q_g+Q_r}{j_t}.
	$$
	\normalsize
	Reorganizing \eqref{eq:complicatedineq1}, we can obtain 
	\small
	\begin{align}
	\nonumber
	\frac{1}{4\gamma}\mathbb{E}\|\bar\bx^{(t)}-\bar\bx^{(t+1)}\|^2
	&\leq \mathbb{E}\psi( \bx_\dagger^{(t)}) -\mathbb{E}\psi(\bar \bx^{(t+1)})+ \left(\frac{1}{60\gamma}+\frac{1}{2\gamma} \right)\mathbb{E}\| \bx_\dagger^{(t)} - \bar \bx^{(t)}\|^2_2+ E_t'\\\label{eq:complicatedineq2}
	&\leq\mathbb{E}\psi(\bar \bx^{(t)})-\mathbb{E}\psi(\bar \bx^{(t+1)}) + \frac{1}{60 \gamma} \mathbb{E}\|\bx_\dagger^{(t)} - \bar \bx^{(t)}\|^2_2 + E_t',
	\end{align}
	\normalsize
	where the second inequality is from \eqref{eq:fact1}.  Applying \eqref{eq:complicatedineq2} to the right hand side of \eqref{eq:complicatedineq1}, we have
	\small
	\begin{eqnarray}
	\nonumber
	&&\mathbb{E}\left[\psi(\bar \bx^{(t+1)})+\frac{1}{2 \gamma} \|\bar \bx^{(t+1)} - \bar \bx^{(t)}\|^2_2-\psi( \bx_\dagger^{(t)}) -\frac{1}{2 \gamma} \| \bx_\dagger^{(t)} - \bar \bx^{(t)}\|^2_2\right]	\\
\label{eq:complicatedineq3}
	&\leq&\frac{1}{30\gamma}\|\bx_\dagger^{(t)}- \bar \bx^{(t)}\|_2^2+\mathbb{E}\left[\psi(\bar \bx^{(t)})-\psi(\bar \bx^{(t+1)})\right] + 2E_t'.
	\end{eqnarray}
	\normalsize
	Reorganizing this inequality gives 
	\small
	\begin{align}\nonumber
	\mathbb{E}\psi(\bar \bx^{(t+1)}) 
	\leq& ~\mathbb{E}\psi( \bx_\dagger^{(t)}) + \left(\frac{1}{30\gamma}+\frac{1}{2\gamma}\right)\mathbb{E}\| \bx_\dagger^{(t)} - \bar \bx^{(t)}\|^2_2-\frac{1}{2\gamma}\mathbb{E}\|\bar\bx^{(t)}-\bar\bx^{(t+1)}\|^2\\\nonumber
	&+ \mathbb{E}\left[\psi(\bar \bx^{(t)})-\psi(\bar \bx^{(t+1)})\right]+2E_t'\\\nonumber
	\leq&~\mathbb{E}\psi( \bar\bx^{(t)}) + \frac{1}{30\gamma} \mathbb{E}\| \bx_\dagger^{(t)} - \bar \bx^{(t)}\|^2_2-\frac{1}{2\gamma}\mathbb{E}\|\bar\bx^{(t)}-\bar\bx^{(t+1)}\|^2\\\nonumber
	&+ \mathbb{E}\left[\psi(\bar \bx^{(t)})-\psi(\bar \bx^{(t+1)})\right]+2E_t'\\\nonumber
	\leq&~\mathbb{E}\psi( \bar\bx^{(t)}) + \left(\frac{1}{30\gamma}-\frac{1-1/a}{2\gamma}\right) \mathbb{E}\| \bx_\dagger^{(t)} - \bar \bx^{(t)}\|^2_2+\frac{a-1}{2\gamma}\mathbb{E}\|\bx_\dagger^{(t)}-\bar\bx^{(t+1)}\|^2\\\label{eq:complicatedineq4}
	&+ \mathbb{E}\left[\psi(\bar \bx^{(t)})-\psi(\bar \bx^{(t+1)})\right]+2E_t',
	\end{align}
	\normalsize
	where the second inequality is due to \eqref{eq:fact1} and the third  is from the Young's inequality, i.e.,  
	$\|\bar\bx^{(t)}-\bar\bx^{(t+1)}\|^2\geq (1-1/a)\|\bar\bx^{(t)}-\bx_\dagger^{(t)}\|^2+(1-a)\|\bx_\dagger^{(t)}-\bar\bx^{(t+1)}\|^2$ for a positive constant $a$ to be determined later.

	By the $(\frac{1}{\gamma}-\rho)$-strong convexity of $\psi(\bx)+\frac{1}{2\gamma}\|\bx-\bar\bx^{(t)}\|_2^2$ and \eqref{eq:complicatedineq3}, we have 
	\small
	\begin{align}
	\nonumber
	\mathbb{E}\| \bx_\dagger^{(t)} - \bar \bx^{(t+1)}\|^2_2\leq&\frac{2\gamma}{1-\gamma\rho}\left[\mathbb{E}\psi(\bar \bx^{(t+1)})+\frac{1}{2 \gamma}\mathbb{E} \|\bar \bx^{(t+1)} - \bar \bx^{(t)}\|^2_2-\psi( \bx_\dagger^{(t)}) -\frac{1}{2 \gamma}\mathbb{E} \| \bx_\dagger^{(t)} - \bar \bx^{(t)}\|^2_2\right]\\\nonumber
	\leq&\frac{1}{15(1-\gamma\rho)}\mathbb{E}\|\bx_\dagger^{(t)}- \bar \bx^{(t)}\|_2^2+\frac{2\gamma}{1-\gamma\rho}\mathbb{E}\left[\psi(\bar \bx^{(t)})-\psi(\bar \bx^{(t+1)})\right] + \frac{4\gamma}{1-\gamma\rho}E_t'.
	\end{align}
	\normalsize
	Applying this inequality to the right hand side of \eqref{eq:complicatedineq4}, we have
	\small
	\begin{align}\nonumber
	\mathbb{E}\psi(\bar \bx^{(t+1)}) 
	\leq&~ \mathbb{E}\psi( \bar\bx^{(t)})+\left(\frac{1}{30\gamma}-\frac{1-1/a}{2\gamma}+\frac{a-1}{30\gamma(1-\gamma\rho)}\right)  \mathbb{E}\| \bx_\dagger^{(t)} - \bar \bx^{(t)}\|^2_2\\\nonumber
	&~+\left(1+\frac{a-1}{1-\gamma\rho}\right) \mathbb{E}\left[\psi(\bar \bx^{(t)})-\psi(\bar \bx^{(t+1)})\right]+\left(2+\frac{2(a-1)}{1-\gamma\rho}\right)E_t'.
	\end{align}
	\normalsize
	Choosing $a=5$ and using the fact that $1-\gamma\rho\geq \frac{1}{2}$ as $\gamma\in (0,1/(2\rho)]$, we derived the following inequality from the one  above 
		\small
	\begin{align}\nonumber
	\mathbb{E}\psi(\bar \bx^{(t+1)}) 
	\leq\mathbb{E}\psi( \bar\bx^{(t)})-\frac{1}{10\gamma}  \mathbb{E}\| \bx_\dagger^{(t)} - \bar \bx^{(t)}\|^2_2+9\mathbb{E}\left[\psi(\bar \bx^{(t)})-\psi(\bar \bx^{(t+1)})\right]+18E_t',
	\end{align}
	\normalsize
	or equivalently,
	\small
	\begin{align}\nonumber
	\mathbb{E}\psi(\bar \bx^{(t+1)}) 
	\leq\mathbb{E}\psi( \bar\bx^{(t)})-\frac{1}{100\gamma}  \mathbb{E}\| \bx_\dagger^{(t)} - \bar \bx^{(t)}\|^2_2+1.8E_t'.
	\end{align}
	\normalsize
	Adding the inequality above for $t=0,...,T-1$ yields
	\begin{align*}
	\mathbb{E}\psi(\bar \bx^{(T)}) \leq &\psi(  \bar\bx^{(0)}) - \sum^{T-1}_{t=0}  \frac{1}{100 \gamma}\mathbb{E}\|\bx_\dagger^{(t)} - \bar \bx^{(t)}\|^2_2+1.8\sum^{T-1}_{t=0} E_t'.
	\end{align*}
	Rearranging the inequality above gives
	\small
	\begin{align*}
	\sum^{T-1}_{t=0}  \mathbb{E} \|\bx_\dagger^{(t)} - \bar \bx^{(t)}\|^2_2 \leq  100 \gamma\bigg(  \psi(\bar \bx^{(0)})-\psi^*+ 1.8\sum^{T-1}_{t=0} E_t'\bigg),
	\end{align*}
	\normalsize
	where $\psi^* = \min_{\bx \in \mathcal{X}} \psi(\bx)>-\infty$. Let $\tau$ be a uniform random variable supported on $\{0,1,...,T-1\}$. Then, we have 
	\small
	\begin{align}
	\label{eqn:boundxx2}
	\mathbb{E}\|\bx_\dagger^{(\tau)} - \bar \bx^{(\tau)}\|^2_2 \leq  \frac{100 \gamma}{T}\bigg(  \psi(\bar \bx^{(0)})-\psi^*+ 1.8\sum^{T-1}_{t=0} E_t'\bigg).
	\end{align}
	\normalsize
	
	Since $j_t=t+32$ in  Algorithm~\ref{alg:PGM}in Case D2, we have
	\small
	\begin{align*}
	\sum^{T-1}_{t=0} E_t'
	&=\sum^{T-1}_{t=0}
	\frac{300M_x^2\gamma}{t+2}+\sum^{T-1}_{t=0}\frac{20M_c^2M_y^2\gamma}{\mu^2(t+32)}+\sum^{T-1}_{t=0}\frac{Q_g+Q_r}{t+32}\\
	&\leq 
	300M_x^2\gamma\ln(T+1)+\frac{20M_c^2M_y^2\gamma}{\mu^2}\ln\left(\frac{T+31}{31}\right)+(Q_g+Q_r)\ln\left(\frac{T+31}{31}\right)\\
	&\leq 
	\left(300M_x^2\gamma+\frac{20M_c^2M_y^2\gamma}{\mu^2}+Q_g+Q_r\right)\ln(T+1).
	\end{align*}
	\normalsize
	Recall that $\mathbb{E}[\text{Dist} (\mathbf{0}, \partial \psi(\bx_\dagger^{(\tau)}) )^2] \leq  \frac{1}{\gamma^2}\mathbb{E}\|\bx_\dagger^{(\tau)} - \bar \bx^{(\tau)}\|^2_2$ by the definition of $\bx_\dagger^{(\tau)}$. According to \eqref{eqn:boundxx2}, we can ensure 
	$\frac{1}{\gamma^2}\mathbb{E}\|\bx_\dagger^{(\tau)} - \bar \bx^{(\tau)}\|^2_2\leq \epsilon^2$
	by setting $T$ large enough so that 
	\small
	$$
	\frac{100 }{T\gamma}\bigg(  \psi(\bar \bx^{(0)})-\psi^*\bigg)\leq \frac{\epsilon^2}{2}\text{  and  }
	\left(300M_x^2\gamma+\frac{20M_c^2M_y^2\gamma}{\mu^2}+Q_g+Q_r\right)\frac{180\ln(T+1)}{T \gamma}\leq \frac{\epsilon^2}{2}.
	$$
	\normalsize
	Solving the two inequalities for $T$, we find that such $T$ can be chosen as
	
	\small
	\begin{eqnarray*}
	T&=&\left\lceil\max\left\{
	\begin{array}{c}
	\frac{200(\psi(\bar \bx^{(0)})-\psi^*)}{\gamma\epsilon^2},\\
	\frac{720\left(300M_x^2\gamma+\frac{20M_c^2M_y^2\gamma}{\mu^2}+Q_g+Q_r\right)}{\gamma\epsilon^2}\ln\left(\frac{720\left(300M_x^2\gamma+\frac{20M_c^2M_y^2\gamma}{\mu^2}+Q_g+Q_r\right)}{\gamma\epsilon^2}\right)
	\end{array}\right\}\right\rceil\\
	&=&\tilde O\left(\left(\frac{1}{\gamma}+\frac{1}{\mu^2}\right)\frac{1}{\epsilon^2}\right).
	\end{eqnarray*}
	\normalsize
	In Case D2, only one stochastic subgradient $(\bg_x^{(j)},\bg_y^{(j)})$ is computed in each iteration of Algorithm~\ref{alg:SMD}. However, at the beginning of each iteration of Algorithm~\ref{alg:PGM}, we need to compute $\bar\by^{(t)}=\argmax\limits_{\by\in\mathbb{R}^q}f(\bar\bx^{(t)},\by)-r(\by)$ through computing $\bc_i(\bar\bx^{(t)})$ for all $i$'s whose computation complexity is equivalent to computing $n$ stochastic subgradients of $f$. Hence, the total (equivalent) number of subgradients computed to generate $\bar \bx^{(\tau)}$ is bounded by
	\small
	$$
	 Tn+\sum^{T-1}_{t=0} j_t =O(Tn+T^2)  =  \tilde O\left(\left(\frac{1}{\gamma}+\frac{1}{\mu^2}\right)\frac{n}{\epsilon^2}+\left(\frac{1}{\gamma^2}+\frac{1}{\mu^4}\right)\frac{1}{\epsilon^4}\right). 
	 $$
	 \normalsize
\end{proof}

\section{Proof of Theorem~\ref{thm:totalcomplexitySVRG} }
\label{sec:prooftheorem2}
We first present the convergence property of the SVRG method (Algorithm~\ref{alg:PGSVRG}) in solving (\ref{subminmaxproblem0finitesum}). This convergence result is well known and can be found in several works~\cite{palaniappan2016stochastic,shi2017bregman,lin2018level}. For simplicity of notation, we write $\phi_{\gamma,\lambda}(\bx,\by;\bar\bx,\bar\by)$ in  (\ref{subminmaxproblem}) as  $\phi(\bx,\by)$ and define
$$
\bar g(\bx)=\frac{1}{2\gamma}\|\bx-\bar\bx\|_2^2+g(\bx)~\text{ and }~\bar r(\by)=\frac{1}{\lambda}V_y(\by,\bar\by)+r(\by).
$$
We need the following lemma to facilitate the proof.
\begin{lemma}
	\label{lemma:svrg}
	Suppose Assumptions~\ref{assume:stochastic} and \ref{assume:finitesum} hold and $\gamma\in(0,1/(2\rho)]$ in $\bar g$ above. Let $\mu_x=\frac{1}{2\gamma}$ and $\mu_y=\mu+\frac{1}{\lambda}$. 
	The function (of $\bx$)
	$\max_{\by\in\Y}f(\bx,\by)-\bar r(\by)$ is smooth and its gradient is $L_x\sqrt{1+(1+L_x/\mu_y)^2}$-Lipschitz continuous. The function (of $\by$)
	$\min_{\bx\in\X}f(\bx,\by)+\bar g(\bx)$ is smooth and its gradient is $L_y\sqrt{1+(1+L_y/\mu_x)^2}$-Lipschitz continuous.
\end{lemma}
\begin{proof}
	We will only prove the first conclusion because the proof of the second conclusion is similar.
	
	Because $\bar r$ is $\mu_y$-strongly convex ($\mu_y>0$) with respect to $V_y$ (and thus with respect to the norm $\|\cdot\|$), 
	according to Assumption~\ref{assume:stochastic}, the solution of $\max_{\by\in\Y}f(\bx,\by)-\bar r(\by)$  is unique and we denoted it by $\by_{\bx}$ to reflect its dependency on $\bx$. According to Danskin's theorem,  $\max_{\by\in\Y}f(\bx,\by)-\bar r(\by)$ is differentiable and its gradient is 
	$\nabla_x f(\bx,\by_{\bx})$.
	
	We next show $\nabla_x f(\bx,\by_{\bx})$ is Lipschitz continuous with respect to $\bx$ on $\mathcal{X}$. Given any $\bx,\bx'\in\mathcal{X}$, by the definition of $\by_{\bx}$ and $\by_{\bx'}$ and the strong convexity of $\bar r$  with respect to  norm $\|\cdot\|$, the following two inequalities hold
	\small
	\begin{eqnarray*}
		\bar r(\by_{\bx})-f(\bx,\by_{\bx})+\frac{\mu_y}{2}\|\by_{\bx}-\by_{\bx'}\|^2&\leq& \bar r(\by_{\bx'})-f(\bx,\by_{\bx'})\\	
		\bar r(\by_{\bx'})-f(\bx',\by_{\bx'})+\frac{\mu_y}{2}\|\by_{\bx'}-\by_{\bx}\|^2&\leq& \bar r(\by_{\bx})-f(\bx',\by_{\bx}).
	\end{eqnarray*}
\normalsize
	Summing up these two inequality gives us
	\small
	\begin{eqnarray*}
		\mu_y\|\by_{\bx}-\by_{\bx'}\|^2&\leq& f(\bx,\by_{\bx})-f(\bx',\by_{\bx})+f(\bx',\by_{\bx'})-f(\bx,\by_{\bx'})\\
		&\leq&(\bx-\bx')^\top\nabla_xf(\bx',\by_{\bx})-(\bx-\bx')^\top\nabla_xf(\bx,\by_{\bx'})+L_x\|\bx-\bx'\|_2^2\\
		&\leq&\|\bx'-\bx\|_2\|\nabla_xf(\bx',\by_{\bx})-\nabla_xf(\bx,\by_{\bx'})\|_2+L_x\|\bx-\bx'\|_2^2\\
		&\leq&\|\bx'-\bx\|_2L_x\sqrt{\|\bx-\bx'\|_2^2+\|\by_{\bx}-\by_{\bx'}\|^2}+L_x\|\bx-\bx'\|_2^2\\
		&\leq&\frac{L_x^2}{2\mu_y}\|\bx'-\bx\|_2^2+\frac{\mu_y}{2}\left(\|\bx-\bx'\|_2^2+\|\by_{\bx}-\by_{\bx'}\|^2\right)+L_x\|\bx-\bx'\|_2^2,
	\end{eqnarray*}
\normalsize
	where the second and the fourth inequalities are because of the $L_x$-Lipschitz continuity of $\nabla_xf(\bx,\by)$, the third is from Cauchy-Schwarz inequality, and the last one is from Young's inequality. Reorganizing the terms in the inequality above leads to 
	\small
	\begin{eqnarray}
	\label{eq:svrglemma1}
	\|\by_{\bx}-\by_{\bx'}\|^2\leq \frac{L_x^2}{\mu_y^2}\|\bx'-\bx\|_2^2+\|\bx-\bx'\|_2^2+\frac{2L_x}{\mu_y}\|\bx-\bx'\|_2^2=\left(1+\frac{L_x}{\mu_y}\right)^2\|\bx-\bx'\|_2^2.
	\end{eqnarray}
	\normalsize
	In addition, by $L_x$-Lipschitz continuity of $\nabla_xf(\bx,\by)$ again, we have
	\small
	$$
	\|\nabla_x f(\bx,\by_{\bx})-\nabla_x f(\bx',\by_{\bx'})\|_2\leq L_x\sqrt{\|\bx-\bx'\|_2^2+\|\by_{\bx}-\by_{\bx'}\|^2}.
	$$
	\normalsize
	Applying \eqref{eq:svrglemma1} to this inequality implies the first conclusion of this lemma.
\end{proof}

\begin{proposition}[Convergence of Algorithm~\ref{alg:SVRG}]
	\label{dualgapresultsvrg}
	Suppose Assumptions~\ref{assume:stochastic} and~\ref{assume:finitesum} hold. 
	Algorithm $\ref{alg:SVRG}$ guarantees 
	\small
	\begin{align}
	\nonumber
		&~\mathbb{E}\Bigg[\frac{\mu_x}{2}\|\widehat\bx^{(K-1)}-\bx_\dagger\|_2^2\bigg]\leq
	\mathbb{E}\left[\max_{\by\in\mathbb{R}^q}\phi_{\gamma} (\widehat\bx^{(K-1)},\by;\bar\bx)-\min_{\bx\in\mathbb{R}^p}\max_{\by\in\mathbb{R}^q}\phi_{\gamma} (\bx,\by;\bar\bx)\right]\\	\label{eq:mainineq1svrg}
	\leq&~\left(\dfrac{3}{4}\right)^{K-1}\left(\dfrac{1}{4}+\dfrac{\Lambda}{2}\right)\left[\mu_xD_x^2+\mu_yD_y^2\right]+\dfrac{D_y^2}{2\lambda}.
	\end{align}
	\normalsize
	where $\mu_x=\frac{1}{2\gamma}$, $\mu_y=\frac{1}{\lambda}+\mu$, and $\bx_\dagger=\text{prox}_{\gamma\psi}(\bar\bx)$.
\end{proposition}
\begin{proof}
	We	denote the saddle-point of \eqref{subminmaxproblem0finitesum} as $(\bx_*,\by_*)$.  
	We first focus on the $k$th outer iteration of in Algorithm~\ref{alg:SVRG}. Let $\mathbb{E}_j$ represent the conditional expectation conditioning on $(\bx^{(0)},\by^{(0)})=(\widehat\bx^{(k)},\widehat\by^{(k)})$ and all random events that happen before the $j$th inner iteration of Algorithm~\ref{alg:SVRG}.

	The optimality conditions of the updating equations of $(\bx^{(j+1)},\by^{(j+1)})$ imply
	\small
	\begin{eqnarray}
	\label{eq:ineqprimal1}\nonumber
	&&\left(\frac{1}{\gamma}+\frac{1}{\eta_x }\right)\frac{\|\bx-\bx^{(j+1)}\|_2^2}{2}+(\bx^{(j+1)})^\top \bg_x^{(j)}+\bar g(\bx^{(j+1)})+\frac{\|\bx^{(j)}-\bx^{(j+1)}\|_2^2}{2\eta_x }\\
	&\leq&
	\bx^\top \bg_x^{(j)}+\bar g(\bx)+\frac{\|\bx-\bx \|_2^2}{2\eta_x }\\
	\label{eq:ineqdual1}\nonumber
	&&\left(\mu_y+\frac{1}{\eta_y }\right)V_y(\by,\by^{(j+1)})-(\by^{(j+1)})^\top \bg_y^{(j)}+\bar r(\by^{(j+1)})+\frac{V_y(\by^{(j+1)},\by^{(j)})}{\eta_y }\\
	&\leq&-\by^\top \bg_y^{(j)}+\bar r(\by)+\frac{V_y(\by,\by^{(j)})}{\eta_y }
	\end{eqnarray}
	\normalsize
	for any $\bx\in\mathcal{X}$ and $\by\in\Y$.	We define
	\begin{eqnarray}
	\tilde\Pc(\bx)&:=& f(\bx,\by_* )+\bar g(\bx)-f(\bx_* ,\by_* )-\bar g(\bx_* )\\
	\tilde\Dc(\by)&:=& f(\bx_* ,\by)-\bar r(\by)-f(\bx_* ,\by_* )+\bar r(\by_* ).
	\end{eqnarray}
	Note that $\min_{\bx\in\mathcal{X}}\tilde\Pc(\bx)=\tilde\Pc(\bx_* )=0$ and $\max_{\by\in\Y}\tilde\Dc(\by)=\tilde\Dc(\by_* )=0$.
	By the $\mu_x$-strong convexity of $\tilde\Pc(\bx)$ with respect to Euclidean distance\footnote{This is because $f(\bx,\by_* )+\bar g(\bx)$ is $(\frac{1}{\gamma}-\rho)$-strongly convex and $\frac{1}{\gamma}-\rho\geq \frac{1}{2\gamma}=\mu_x$ as $\gamma\in(0,1/(2\rho)]$.} and the $\mu_y$-strong convexity of $\tilde\Dc(\by)$ with respect to $V_y$, we can show that
	\begin{eqnarray}
	\label{eq:convproof1}
	\tilde\Pc(\bx)\geq\frac{\mu_x\|\bx-\bx_* \|_2^2}{2}\quad \text{and}\quad-\tilde\Dc(\by)\geq\mu_y V_y(\by,\by_* ).
	\end{eqnarray}
	
	We choose $\bx=\bx_* $ in \eqref{eq:ineqprimal1} and $\by=\by_* $ in \eqref{eq:ineqdual1}, and then add \eqref{eq:ineqprimal1} and \eqref{eq:ineqdual1} together. After organizing terms, we obtain 
	\small
	\begin{eqnarray}
	\nonumber
	&&\left(\frac{1}{\gamma}+\frac{1}{\eta_x }\right)\frac{\|\bx_* -\bx^{(j+1)}\|_2^2}{2}+\frac{\|\bx^{(j)}-\bx^{(j+1)}\|_2^2}{2\eta_x } +\left(\mu_y+\frac{1}{\eta_y }\right)V_y(\by_* ,\by^{(j+1)})+f(\bx_* ,\by)\\\nonumber
	&&+\frac{V_y(\by^{(j+1)},\by^{(j)})}{\eta_y } + \tilde\Pc(\bx^{(j+1)})-\tilde\Dc(\by^{(j+1)})-\frac{\|\bx_* -\bx^{(j)}\|_2^2}{2\eta_x }- \frac{V_y(\by_* ,\by^{(j)})}{\eta_y }-f(\bx,\by_* )\\\nonumber
	&\leq&(\bx_* -\bx^{(j+1)})^\top \bg_x^{(j)}-(\by_*  -\by^{(j+1)})^\top \bg_y^{(j)}.
	\end{eqnarray}
	\normalsize
	We reorganize the right hand side of the inequality above as follows
	\small
	\begin{eqnarray}
	\nonumber
	&&(\bx_* -\bx^{(j+1)})^\top \bg_x^{(j)}-(\by_*  -\by^{(j+1)})^\top \bg_y^{(j)}
	\\\nonumber
	&=&(\bx_* -\bx^{(j)})^\top \left[\bg_x^{(j)}-\nabla_x f(\bx^{(j)},\by^{(j)})\right]
	-(\by_*  -\by^{(j)})^\top \left[\bg_y^{(j)}-\nabla_y f(\bx^{(j)},\by^{(j)})\right]\\\nonumber
	&&+(\bx^{(j)}-\bx^{(j+1)})^\top \left[\bg_x^{(j)}-\nabla_x f(\bx^{(j)},\by^{(j)})\right]
	-(\by^{(j)}-\by^{(j+1)})^\top \left[\bg_y^{(j)}-\nabla_y f(\bx^{(j)},\by^{(j)})\right]\\
	&&	+(\bx_* -\bx^{(j+1)})^\top\nabla_x f(\bx^{(j)},\by^{(j)})-(\by_* -\by^{(j+1)})^\top \nabla_y f(\bx^{(j)},\by^{(j)}).
	\label{eq:ineq1svrg}
	\end{eqnarray}
	\normalsize

	Next, we study the three lines on the right hand side of \eqref{eq:ineq1}, respectively. 
	Since the random index $l$ is independent of $\bx^{(j)}$ and $\by^{(j)}$, we have 
	\small
	\begin{eqnarray}
	\label{eq:ineq2svrg}
	\E_j\left[(\bx_* -\bx^{(j)})^\top \left[\bg_x^{(j)}-\nabla_x f(\bx^{(j)},\by^{(j)})\right]
	-(\by_*  -\by^{(j)})^\top \left[\bg_y^{(j)}-\nabla_y f(\bx^{(j)},\by^{(j)})\right]\right]=0
	\end{eqnarray}
	\normalsize
	by the definition of $\bg_x^{(j)}$ and $\bg_y^{(j)}$. 	By the definition of $\bg_x^{(j)}$, Cauchy-Schwarz inequality and Young's inequality, we have
	\small
	\begin{align}
	\nonumber
	&~\E_j\left[(\bx^{(j)}-\bx^{(j+1)})^\top \left[\bg_x^{(j)}- \nabla_x f(\bx^{(j)},\by^{(j)}) \right]\right]\\\nonumber
	\leq&~\frac{1}{2a}\E_j\|\bx^{(j)}-\bx^{(j+1)}\|_2^2\\\nonumber
	&~+\frac{a}{2}\E_j\|\nabla_x f(\widehat\bx^{(k)},\widehat\by^{(k)})-\nabla_x  f_l(\widehat\bx^{(k)},\widehat\by^{(k)})+\nabla_x f_l(\bx^{(j)},\by^{(j)})- \nabla_x f(\bx^{(j)},\by^{(j)}) \|_2^2\\\nonumber
	\leq&~\frac{1}{2a}\E_j\|\bx^{(j)}-\bx^{(j+1)}\|_2^2 +2aL_x^2\|\bx^{(j)}-\bx_* \|_2^2 +2aL_x^2\|\widehat\bx^{(k)}-\bx_* \|_2^2	\\\label{eq:ineq3svrg}
	&~+4aL_x^2V_y(\by^{(j)},\by_* )+4aL_x^2V_y(\widehat\by^{(k)},\by_* ),
	\end{align}
	\normalsize
	where $a>0$ is a constant to be determined later. Similarly, we can prove that 
	\small
	\begin{align}
	\nonumber	
	&~\E_j\left[(\by^{(j)}-\by^{(j+1)})^\top \left[\nabla_y f(\bx^{(j)},\by^{(j)})-\bg_y^{(j)}\right]\right]\\\nonumber
	\leq&~\frac{1}{b}V_y(\by^{(j+1)},\by^{(j)})
	+2bL_y^2\|\widehat\bx^{(k)}-\bx_* \|_2^2+2bL_y^2\|\bx^{(j)}-\bx_* \|_2^2\\\label{eq:ineq4svrg}
	&+4bL_y^2V_y(\by^{(j)},\by_* )+4bL_y^2V_y(\widehat\by^{(k)},\by_* ),
	\end{align}
	\normalsize
	where $b>0$ is a constant to be determined later.
	Next, we consider the third line in the right hand side of \eqref{eq:ineq1svrg}. We first show that
	\small
	\begin{align}
	\nonumber
	&(\bx_* -\bx^{(j+1)})^\top \nabla_x f(\bx^{(j)},\by^{(j)})\\\nonumber
	\leq&(\bx_* -\bx^{(j+1)})^\top \nabla_x f(\bx^{(j+1)},\by^{(j+1)})+\frac{1}{2c}\|\bx_* -\bx^{(j+1)}\|_2^2\\\nonumber
	&+\frac{c}{2}L_x^2\|\bx^{(j)}-\bx^{(j+1)}\|_2^2+\frac{c}{2}L_x^2\|\by^{(j)}-\by^{(j+1)}\|^2\\\nonumber
	\leq&f(\bx_* ,\by^{(j+1)})-f(\bx^{(j+1)},\by^{(j+1)})+\frac{\rho}{2}\|\bx_* -\bx^{(j+1)}\|_2^2+\frac{1}{2c}\|\bx_* -\bx^{(j+1)}\|_2^2\\\label{eq:ineq5svrg}
	&+\frac{c}{2}L_x^2\|\bx^{(j)}-\bx^{(j+1)}\|_2^2+cL_x^2V_y(\by^{(j+1)},\by^{(j)}),
	\end{align}
	\normalsize
	where $c>0$ is a constant to be determined later.
	Similarly, we have  
	\small
	\begin{align}
	\label{eq:ineq6svrg}\nonumber
	&-(\by_* -\by^{(j+1)})^\top \nabla_y f(\bx^{(j)},\by^{(j)})\\\nonumber
	\leq&-f(\bx^{(j+1)},\by_* )+f(\bx^{(j+1)},\by^{(j+1)})+\frac{1}{d}V_y(\by^{(j+1)},\by_* )+\frac{d}{2}L_y^2\|\bx^{(j)}-\bx^{(j+1)}\|_2^2\\
	&+d L_y^2V_y(\by^{(j+1)},\by^{(j)}),
	\end{align}
	\normalsize
	where $d>0$ is a constant to be determined later.
	
	Applying inequalities from \eqref{eq:ineq2svrg} to \eqref{eq:ineq6svrg} to \eqref{eq:ineq1svrg} leads to the following inequality 
	\small
	\begin{align}
	\nonumber
	&\left(\mu_x+\frac{1}{\eta_x }-\frac{1}{c}\right)\frac{\E_j\|\bx_* -\bx^{(j+1)}\|_2^2}{2}
	+\left(\mu_y+\frac{1}{\eta_y }-\frac{1}{d}\right)\E_jV_y(\by_* ,\by^{(j+1)})\\\nonumber
	&+\E_j\left(\tilde\Pc(\bx^{(j+1)})-\tilde\Dc(\by^{(j+1)})\right)\\\nonumber
	\leq&\left(4aL_x^2+4bL_y^2+\frac{1}{\eta_x }\right)\frac{\|\bx_* -\bx^{(j)}\|_2^2}{2}
	+\left(4aL_x^2+4bL_y^2+\frac{1}{\eta_y }\right)V_y(\by_* ,\by )\\\nonumber
	&+\left(2aL_x^2+2bL_y^2\right)\|\bx_* -\widehat\bx^{(k)}\|_2^2+\left(4aL_x^2+4bL_y^2\right)V_y(\by_* ,\widehat\by^{(k)})\\\nonumber
	&+\left(\frac{1}{a}+cL_x^2+dL_y^2-\frac{1}{\eta_x }\right)\frac{\E_j\|\bx^{(j)}-\bx^{(j+1)}\|_2^2}{2}\\\label{eq:ineq6}
	&+\left(\frac{1}{b}+cL_x^2+dL_y^2-\frac{1}{\eta_y }\right)\E_jV_y(\by^{(j+1)},\by^{(j)}).
	\end{align}
	\normalsize
	Choosing $a=b=\frac{1}{48}\frac{\min\left\{\mu_x,\mu_y\right\}}{\max\left\{L_x^2,L_y^2\right\}} $, $c=\frac{2}{\mu_x}$, and $d=\frac{2}{\mu_y}$ in the inequality above, we obtain 
	\small
	\begin{align}
	\nonumber
	&\left(\frac{\mu_x}{2}+\frac{1}{\eta_x }\right)\frac{\E_j\|\bx_* -\bx^{(j+1)}\|_2^2}{2}
	+\left(\frac{\mu_y}{2}+\frac{1}{\eta_y }\right)\E_jV_y(\by_* ,\by^{(j+1)})\\\nonumber
	&+ \E_j\left(\tilde\Pc(\bx^{(j+1)})-\tilde\Dc(\by^{(j+1)})\right)\\\nonumber
	\leq&\left(\frac{\mu_x}{6}+\frac{1}{\eta_x }\right)\frac{\|\bx_* -\bx^{(j)}\|_2^2}{2}
	+\left(\frac{\mu_y}{6}+\frac{1}{\eta_y }\right)V_y(\by_* ,\by )+\frac{\mu_x}{12}\|\bx_* -\widehat\bx^{(k)}\|_2^2+\frac{\mu_y}{6}V_y(\by_* ,\widehat\by^{(k)})\\ \nonumber
	&+\left(52\frac{\max\left\{L_x^2,L_y^2\right\}}{\min\left\{\mu_x,\mu_y\right\}}-\frac{1}{\eta_x }\right)\frac{\E_j\|\bx^{(j)}-\bx^{(j+1)}\|_2^2}{2}\\\label{eq:ineq7}
	&+\left(52\frac{\max\left\{L_x^2,L_y^2\right\}}{\min\left\{\mu_x,\mu_y\right\}}-\frac{1}{\eta_y }\right)\E_jV_y(\by^{(j+1)},\by^{(j)}).
	\end{align}
	\normalsize
	
	Given that	$\eta_x=\frac{1}{\mu_x\Lambda}$ and $\eta_y=\frac{1}{\mu_y\Lambda}$
	with
	$\Lambda=52\frac{\max\left\{L_x^2,L_y^2\right\}}{\min\left\{\mu_x^2,\mu_y^2\right\}}$, we have $52\frac{\max\left\{L_x^2,L_y^2\right\}}{\min\left\{\mu_x,\mu_y\right\}}\leq\frac{1}{\eta_x }$ and $52\frac{\max\left\{L_x^2,L_y^2\right\}}{\min\left\{\mu_x,\mu_y\right\}}\leq\frac{1}{\eta_y }$. After organizing terms, inequality \eqref{eq:ineq7} becomes 
	\small
	\begin{align}
	\nonumber
	&\frac{\mu_x\E_j\|\bx_* -\bx^{(j+1)}\|_2^2}{2}
	+\mu_y\E_jV_y(\by_* ,\by^{(j+1)})+ \left(\frac{1}{2}+\Lambda\right)^{-1}\E_j\left(\tilde\Pc(\bx^{(j+1)})-\tilde\Dc(\by^{(j+1)})\right)\\\nonumber
	\leq&\left(1-\frac{1}{3/2+3\Lambda}\right)\bigg[\frac{\mu_x\|\bx_* -\bx^{(j)}\|_2^2}{2}
	+\mu_yV_y(\by_* ,\by )\bigg]\\\nonumber
	&+\frac{1}{3/2+3\Lambda}\left(\frac{\mu_x\|\bx_* -\widehat\bx^{(k)}\|_2^2}{2}+\mu_yV_y(\by_* ,\widehat\by^{(k)})\right).
	\end{align}
	\normalsize
	Applying this inequality recursively for $j=0,1,\dots,J-2$ and organizing terms, we have
	\small
	\begin{align}
	\nonumber
	&\frac{\mu_x\E \|\bx_* -\bx^{(J-1)}\|_2^2}{2}
	+\mu_y\E V_y(\by_* ,\by^{(J-1)})+ \left(\frac{1}{2}+\Lambda\right)^{-1}\E\left(\tilde\Pc(\bx^{(J-1)})-\tilde\Dc(\by^{(J-1)})\right)\\\nonumber
	\leq&\left(1-\frac{1}{3/2+3\Lambda}\right)^{J-1}\bigg[\frac{\mu_x\|\bx_* -\bx^{(0)}\|_2^2}{2}
	+\mu_yV_y(\by_* ,\by^{(0)})\bigg]\\\nonumber
	&+\frac{1}{2}\left(\frac{\mu_x\|\bx_* -\widehat\bx^{(k)}\|_2^2}{2}+\mu_yV_y(\by_* ,\widehat\by^{(k)})\right).
	\end{align}
	\normalsize
	Recall that $(\bx^{(0)},\by^{(0)})=(\widehat\bx^{(k)},\widehat\by^{(k)})$ and $(\bx^{(J-1)},\by^{(J-1)})=(\widehat\bx^{(k+1)},\widehat\by^{(k+1)})$ and the fact that $J=1+(3/2+3\Lambda)\log(4)$ in Algorithm~\ref{alg:SVRG}. The inequality above implies 
	\small
	\begin{eqnarray}
	\nonumber
	&&\frac{\mu_x\E_j\|\bx_* -\widehat\bx^{(k+1)}\|_2^2}{2}
	+\mu_y\E_jV_y(\by_* ,\widehat\by^{(k+1)})+ \left(\frac{1}{2}+\Lambda\right)^{-1}\left(\tilde\Pc(\widehat\bx^{(k+1)})-\tilde\Dc(\widehat\by^{(k+1)})\right)\\\label{eq:ineq8}
	&\leq&\frac{3}{4}\left[\frac{\mu_x\|\bx_* -\widehat\bx^{(k)}\|_2^2}{2}+\mu_yV_y(\by_* ,\widehat\by^{(k)})\right].
	\end{eqnarray}
	\normalsize
	
	Accordiong to Lemma~\ref{lemma:svrg}, the function (of $\bx$)
	$\max_{\by\in\Y}f(\bx,\by)-\bar r(\by)$ is smooth and its gradient is $L_x\sqrt{1+(1+L_x/\mu_y)^2}$-Lipschitz continuous, and the function (of $\by$)
	$\min_{\bx\in\X}f(\bx,\by)+\bar g(\bx)$ is smooth and its gradient is $L_y\sqrt{1+(1+L_y/\mu_x)^2}$-Lipschitz continuous. Therefore, we define
	\begin{align*}
		\Pc(\bx):= \max_{\by\in\mathbb{R}^q}\phi(\bx,\by)~\text{ and }~
		\Dc(\by):= \min_{\bx\in\mathbb{R}^p}\phi(\bx,\by)
	\end{align*}
	and we can easily show that  
	\small
	\begin{align}
		\nonumber
		&\Pc (\bx)-\Dc(\by)	\\\label{eq:ineq88}
	\leq	& \tilde\Pc(\bx)-\tilde\Dc(\by)+\frac{L_x}{\mu_x}\sqrt{1+\left(1+\frac{L_x}{\mu_y}\right)^2}\frac{\mu_x\|\bx-\bx_* \|_2^2}{2}+\frac{L_y}{\mu_y}\sqrt{1+\left(1+\frac{L_y}{\mu_x}\right)^2}\mu_yV_y(\by_* ,\by)
	\end{align}
	\normalsize
	for any $\bx\in\mathcal{X}$ and $\by\in\Y$. Note that, by the definition of $\Lambda$, we have
	\small
\begin{align*}
	&\frac{L_x}{\mu_x}\sqrt{1+\left(1+\frac{L_x}{\mu_y}\right)^2}
	\frac{1}{2}\left[\left(\frac{L_x}{\mu_x}\right)^2+1+\left(1+\frac{L_x}{\mu_y}\right)^2\right]\\
	\leq&\frac{1}{2}\left[\left(\frac{L_x}{\mu_x}+1\right)^2+\left(1+\frac{L_x}{\mu_y}\right)^2\right]
	\leq\left(1+\frac{\max\{L_x,L_y\}}{\min\{\mu_x,\mu_y\}}\right)^2 \leq \frac{1}{2}+\Lambda.
\end{align*}
\normalsize
By a similar argument, we can show that $\frac{L_y}{\mu_y}\sqrt{1+\left(1+\frac{L_y}{\mu_x}\right)^2}\leq \frac{1}{2}+\Lambda$. With these results, inequality \eqref{eq:ineq88} implies
	\small
\begin{eqnarray*}
	\Pc (\bx)-\Dc(\by)&\leq& \tilde\Pc(\bx)-\tilde\Dc(\by)+\left(\frac{1}{2}+\Lambda\right)\frac{\mu_x\|\bx-\bx_* \|_2^2}{2}+\left(\frac{1}{2}+\Lambda\right)\mu_yV_y(\by_* ,\by).
\end{eqnarray*}
\normalsize
	Choosing $(\bx,\by)=(\widehat\bx^{(k+1)},\widehat\by^{(k+1)})$ and applying \eqref{eq:ineq8} to the inequality above yield
	\small
	\begin{eqnarray}
	\nonumber
	\Pc(\widehat\bx^{(k+1)})-\Dc(\widehat\by^{(k+1)})
	&\leq&\frac{3}{4}\left(\frac{1}{2}+\Lambda\right)\left[\frac{\mu_x\|\bx_* -\widehat\bx^{(k)}\|_2^2}{2}+\mu_yV_y(\by_* ,\widehat\by^{(k)})\right].
	\end{eqnarray}
	\normalsize
	Applying this inequality and \eqref{eq:ineq8} recursively for $k=0,1,\dots,K-2$, we have
	\small
	\begin{eqnarray}
	\nonumber
	\Pc(\widehat\bx^{(K-1)})-\Dc(\widehat\by^{(K-1)})
	&\leq&\left(\frac{3}{4}\right)^{K-1}\left(\frac{1}{2}+\Lambda\right)\left[\frac{\mu_x\|\bx_* -\widehat\bx^{(0)}\|_2^2}{2}+\mu_yV_y(\by_* ,\widehat\by^{(0)})\right].
	\end{eqnarray}
	\normalsize
	According to the facts that $\Dc(\widehat\by^{(K-1)})\leq\min_{\bx\in\mathbb{R}^p}\max_{\by\in\mathbb{R}^q}\phi(\bx,\by)$ and $(\widehat\bx^{(0)},\widehat\by^{(0)})=(\bar\bx ,\bar\by )$, we have
	\small
	\begin{eqnarray}
	\nonumber
	\max_{\by\in\mathbb{R}^q}\phi(\widehat\bx^{(K-1)},\by)-\min_{\bx\in\mathbb{R}^p}\max_{\by\in\mathbb{R}^q}\phi(\bx,\by)
	&\leq&\left(\frac{3}{4}\right)^{K-1}\left(\frac{1}{2}+\Lambda\right)\left[\frac{\mu_x\|\bx_* -\bar\bx \|_2^2}{2}+\mu_yV_y(\by_* ,\bar\by )\right].
	\end{eqnarray}
	\normalsize
	Because $\bar\by\in\argmin_{\by\in\mathbb{R}^p}d_y(\by)$ and Assumption~\ref{assume:finitesum}C holds, we have $V_y(\by_*,\bar\by)\leq \max_{\by\in\Y}d_y(\by)-\min_{\by\in\Y}d_y(\by)\leq D_y^2/2$. As a result, we have
	\small
	\begin{align}
	\nonumber
	&\max_{\by\in\mathbb{R}^q}\phi_{\gamma} (\widehat\bx^{(K-1)},\by;\bar\bx)-\min_{\bx\in\mathbb{R}^p}\max_{\by\in\mathbb{R}^q}\phi_{\gamma} (\bx,\by;\bar\bx)\\\nonumber
	\leq&\max_{\by\in\mathbb{R}^q}\phi_{\gamma,\lambda} (\widehat\bx^{(K-1)},\by;\bar\bx,\bar\by)-\min_{\bx\in\mathbb{R}^p}\max_{\by\in\mathbb{R}^q}\phi_{\gamma,\lambda} (\bx,\by;\bar\bx,\bar\by)+\frac{D_y^2}{2\lambda}\\\nonumber
	=&\max_{\by\in\mathbb{R}^q}\phi(\widehat\bx^{(K-1)},\by)-\min_{\bx\in\mathbb{R}^p}\max_{\by\in\mathbb{R}^q}\phi(\bx,\by)+\frac{D_y^2}{2\lambda}\\\nonumber
	\leq&\left(\frac{3}{4}\right)^{K-1}\left(\frac{1}{2}+\Lambda\right)\left[\frac{\mu_x\|\bx_* -\bar\bx \|_2^2}{2}+\mu_yV_y(\by_* ,\bar\by )\right]+\frac{D_y^2}{2\lambda},
	\end{align}
	\normalsize
	which completes the proof of the second inequality in \eqref{eq:mainineq1svrg} after applying the definitions of $D_x$ and $D_y$ in Assumption~\ref{assume:finitesum}C. 
	The first inequality in \eqref{eq:mainineq1svrg} holds because of  $(\frac{1}{\gamma}-\rho)$-strong convexity of $\phi$ in $\bx$ and the fact that $\frac{1}{\gamma}-\rho\geq \frac{1}{2\gamma}=\mu_x$ as $\gamma\in(0,1/(2\rho)]$.
\end{proof}

\textbf{Proof of Theorem~\ref{thm:totalcomplexitySVRG}}
\begin{proof}
	Recall that $\max_{\by\in\mathbb{R}^q}\phi_{\gamma} (\bx,\by;\bar\bx)=\psi (\bx) + \frac{1}{2 \gamma} \| \bx- \bar \bx\|^2_2$.
	We choose $\bar\bx=\bar\bx^{(t)}$, $\widehat\bx=\bar\bx^{(t+1)}$, $K=k_t$, $\lambda=\lambda_t$, $\mu_y=\mu_y^t$ and $\Lambda=\Lambda_t$ 
	in Proposition~\ref{dualgapresultsvrg} with $k_t$, $\lambda_t$, $\mu_y^t$ and $\Lambda_t$ defined as in the $t$th iteration of Algorithm~\ref{alg:PGSVRG}, and obtain 
	\small
	\begin{align}\label{eq:SVRGconvergemainineq1}
		\mathbb{E}\left[\psi(\bar \bx^{(t+1)})+\frac{1}{2 \gamma} \|\bar \bx^{(t+1)} - \bar \bx^{(t)}\|^2_2-\psi( \bx_\dagger^{(t)}) -\frac{1}{2 \gamma} \| \bx_\dagger^{(t)} - \bar \bx^{(t)}\|^2_2\right]		\leq F_t+\frac{D_y^2}{2\lambda_t}\\\label{eq:SVRGconvergemainineq2}	
		\mathbb{E}\bigg[\frac{\mu_x}{2}\|\bar\bx^{(t+1)}-\bx_\dagger^{(t)}\|_2^2\bigg]\leq F_t+\frac{D_y^2}{2\lambda_t},
	\end{align}
	\normalsize
	where  
	\small
	$$
	F_t:=\left(\frac{3}{4}\right)^{k_t-1}\left(\frac{1}{4}+\frac{\Lambda_t}{2}\right)\left[\mu_xD_x^2+\mu_y^tD_y^2\right].
	$$
	\normalsize
	Using  the inequalities above together with \eqref{eq:fact2}, we can show that
	\small
	\begin{align*}
	\mathbb{E}\psi(\bar \bx^{(t+1)}) &\leq \mathbb{E}\psi( \bx_\dagger^{(t)}) + \frac{1}{2 \gamma} \mathbb{E}\| \bx_\dagger^{(t)} - \bar \bx^{(t)}\|^2_2 - \frac{1}{2 \gamma} \mathbb{E}\|\bar \bx^{(t+1)} - \bar \bx^{(t)}\|^2_2+ F_t+\frac{D_y^2}{2\lambda_t}\\
	&\leq \mathbb{E}\psi( \bx_\dagger^{(t)}) + \frac{1}{2 \gamma} \bigg( \frac{1}{3} \mathbb{E}\|\bx_\dagger^{(t)} - \bar \bx^{(t)}\|^2_2 + 4\mathbb{E} \|\bx_\dagger^{(t)} - \bar \bx^{(t+1)}\|^2_2    \bigg) + F_t+\frac{D_y^2}{2\lambda_t}\\
	&\leq \mathbb{E}\psi( \bx_\dagger^{(t)}) + \frac{1}{6 \gamma} \mathbb{E}\|\bx_\dagger^{(t)} - \bar \bx^{(t)}\|^2_2 + 9\left(F_t+\frac{D_y^2}{2\lambda_t}\right).
	\end{align*}
	\normalsize
	where the first two inequalities are because of \eqref{eq:SVRGconvergemainineq1} and \eqref{eq:fact2}, respectively, and the last inequality is because  \eqref{eq:SVRGconvergemainineq2} and $\mu_x= \frac{1}{2\gamma}$.
	Combining this inequality with \eqref{eq:fact1} leads to 
	\small
	\begin{align*}
	\mathbb{E}\psi(\bar \bx^{(t+1)}) &\leq \mathbb{E}\psi(\bar \bx^{(t)}) - \frac{1}{3 \gamma} \mathbb{E}\|\bx_\dagger^{(t)} - \bar \bx^{(t)}\|^2_2 +9\left(F_t+\frac{D_y^2}{2\lambda_t}\right).
	\end{align*}
	\normalsize
	Adding the inequalities for $t=0,...,T-1$ yields
	\normalsize
	\begin{align*}
	\mathbb{E}\psi(\bar \bx^{(T)}) \leq &\psi(  \bar\bx^{(0)}) - \sum^{T-1}_{t=0}  \frac{1}{3 \gamma} \mathbb{E}\|\bx_\dagger^{(t)} - \bar \bx^{(t)}\|^2_2+9\sum^{T-1}_{t=0} \left(F_t+\frac{D_y^2}{2\lambda_t}\right).
	\end{align*}
	\normalsize
	Rearranging the inequality above gives
	\small
	\begin{align*}
	\sum^{T-1}_{t=0}   \mathbb{E}\|\bx_\dagger^{(t)} - \bar \bx^{(t)}\|^2_2 \leq  3 \gamma\bigg( \psi(\bar \bx^{(0)})-\psi^*+ 9\sum^{T-1}_{t=0}\left(F_t+\frac{D_y^2}{2\lambda_t}\right)\bigg),
	\end{align*}
	\normalsize
	where $\psi^* = \min_{\bx \in \mathcal{R}^p} \psi(\bx)>-\infty$. Let $\tau$ be a uniform random variable supported on $\{0,1,...,T-1\}$. Then, we have 
	\small
	\begin{align}
	\label{eqn:boundxx3}
	\mathbb{E}\|\bx_\dagger^{(\tau)} - \bar \bx^{(\tau)}\|^2_2 \leq  \frac{3 \gamma}{T}\bigg(  \psi(\bar \bx^{(0)})-\psi^*+ 9\sum^{T-1}_{t=0} \left(F_t+\frac{D_y^2}{2\lambda_t}\right)\bigg).
	\end{align}
	\normalsize
	
	Next, we derive the complexity of Algorithm~\ref{alg:PGSVRG} when $\mu>0$ and $\mu=0$ separately. 
	
	Suppose $\mu>0$. Algorithm~\ref{alg:PGSVRG} chooses 
	$$
	k_t = \left\lceil 1+4\log\left(9(t+1)^2\left(\frac{1}{4}+\frac{\Lambda_t}{2}\right)(\mu_xD_x^2+\mu_y^tD_y^2) \right)\right\rceil\text{ and } \lambda_t=+\infty
	$$ 
	such that $\frac{1}{\lambda_t}=0$, $\mu_y^t=\mu$, and  $\Lambda_t=\frac{52\max\left\{L_x^2,L_y^2\right\}}{\min\left\{\mu_x^2,\mu^2\right\}}=O\left(\gamma^2+\frac{1}{\mu^2}\right)$. 
	Therefore,
	\small
	\begin{align*}
	9 \sum^{T-1}_{t=0} \left(F_t+\frac{D_y^2}{2\lambda_t}\right)
	&=\sum^{T-1}_{t=0}
	\frac{1}{(t+2)^2}\leq\frac{\pi^2}{6}.
	\end{align*}
	\normalsize
	Recall that $\mathbb{E}[\text{Dist} (\mathbf{0}, \partial \psi(\bx_\dagger^{(\tau)}) )^2] \leq  \frac{1}{\gamma^2}\mathbb{E}\|\bx_\dagger^{(\tau)} - \bar \bx^{(\tau)}\|^2_2$ holds by the definition of $\bx_\dagger^{(\tau)}$. According to \eqref{eqn:boundxx3}, we can ensure
$\mathbb{E}[\text{Dist} (\mathbf{0}, \partial \psi(\bx_\dagger^{(\tau)}) )^2]\leq \epsilon^2$
	by choosing 
	\small
	$$
	T=\frac{3(\psi(\bar \bx^{(0)})-\psi^*+\pi^2/6)}{\gamma\epsilon^2}.
	$$
	\normalsize
	In this case, there are $J=j_t:=\left\lceil 1+(\frac{3}{2}+3\Lambda_t)\log(4)\right\rceil$ inner iterations in Algorithm~\ref{alg:SVRG} at iterate $t$ of  Algorithm~\ref{alg:PGSVRG}. 
	One stochastic gradient $(\bg_x^{(j)},\bg_y^{(j)})$ of $f$ is computed in each inner iteration of Algorithm~\ref{alg:SVRG} and thus there are $k_tj_t$ stochastic gradients computed in the $t$th iteration of Algorithm~\ref{alg:PGSVRG}. 
	One deterministic gradient $(\hbg_x^{(k)},\hbg_y^{(k)})$ of $f$ is computed in each outer iteration of Algorithm~\ref{alg:SVRG} whose complexity is equivalent to computing $n$ stochastic gradients of $f$. Hence, there are $k_tn+k_tj_t$ (equivalent) stochastic gradients computed in the $t$th iteration of Algorithm~\ref{alg:PGSVRG}. Therefore, the total number of stochastic gradients computed to generate $\bar \bx^{(\tau)}$ is bounded by
	\begin{align*}
	&\sum^{T-1}_{t=0} k_t(n+j_t) \\
	=&  \sum^{T-1}_{t=0}\left(1+4\log\left(9(t+1)^2\left(\frac{1}{4}+\frac{\Lambda_t}{2}\right)(\mu_xD_x^2+\mu_y^tD_y^2) \right)\right)\left(n+ 1+\left(\frac{3}{2}+3\Lambda_t\right)\log(4)\right)\\
	=&\tilde {O}(nT)= \tilde {O}\left(\frac{n}{\gamma\epsilon^2}+\left(\gamma^2+\frac{1}{\mu^2}\right)\frac{1}{\gamma\epsilon^2}\right)
	=\tilde {O}\left(\frac{n}{\gamma\epsilon^2}+\frac{\gamma}{\epsilon^2}+\frac{1}{\mu^2\gamma\epsilon^2}\right).
	\end{align*}

	Suppose $\mu=0$. Algorithm~\ref{alg:PGSVRG} chooses 
	$$
	k_t = \left\lceil 1+4\log\left(9(t+1)^2\left(\frac{1}{4}+\frac{\Lambda_t}{2}\right)(\mu_xD_x^2+\mu_y^tD_y^2) \right)\right\rceil\text{ and } \lambda_t=t+2
	$$ 
	such that $\mu_y^t=\frac{1}{\lambda_t}=\frac{1}{t+2}$ and  $\Lambda_t=\frac{52\max\left\{L_x^2,L_y^2\right\}}{\min\left\{\mu_x^2,(t+2)^{-2}\right\}}=O\left(\gamma^2+t^2\right)$. We denote $J$ in Algorithm~\ref{alg:SVRG} under this case
	as $j_t$ so that $j_t=\left\lceil 1+(\frac{3}{2}+3\Lambda_t)\log(4)\right\rceil$. Therefore,
	\small
	\begin{align*}
	9\sum^{T-1}_{t=0} \left(F_t+\frac{D_y^2}{2\lambda_t}\right)
	&=\sum^{T-1}_{t=0}
	\frac{1}{(t+2)^2}+9\sum^{T-1}_{t=0}\frac{D_y^2}{2(t+2)}
	\leq\frac{\pi^2}{6}+\ln(T+2)\frac{9D_y^2}{2}.
	\end{align*}
	\normalsize
	Recall that $\mathbb{E}[\text{Dist} (\mathbf{0}, \partial \psi(\bx_\dagger^{(\tau)}) )^2] \leq  \frac{1}{\gamma^2}\mathbb{E}\|\bx_\dagger^{(\tau)} - \bar \bx^{(\tau)}\|^2_2$ holds by the definition of $\bx_\dagger^{(\tau)}$. According to \eqref{eqn:boundxx3}, we can ensure
	$\mathbb{E}[\text{Dist} (\mathbf{0}, \partial \psi(\bx_\dagger^{(\tau)}) )^2]\leq \epsilon^2$
	by choosing 
	\small
	$$
	T=\max\left\{\frac{6(\psi(\bar \bx^{(0)})-\psi^*+\pi^2/6)}{\gamma\epsilon^2},
	\frac{54 D_y^2}{\gamma\epsilon^2}\ln\left(\frac{54 D_y^2}{\gamma\epsilon^2}\right)
	\right\}.
	$$
	\normalsize
	Following the same scheme of counting as in the case when $\mu>0$, the total (equivalent) number of stochastic gradients computed to generate $\bar \bx^{(\tau)}$ is bounded by
	$$ \sum^{T-1}_{t=0} k_t(n+j_t) = \sum^{T-1}_{t=0}\tilde {O} (n+\gamma^2+t^2)= \tilde {O}\left(\frac{n}{\gamma\epsilon^2}+\frac{\gamma}{\epsilon^2}+\frac{1}{\gamma^3\epsilon^6}\right). $$
\end{proof}

\end{document}